\pdfoutput=1

\documentclass[
	12pt%
	,pagesize%
	,headings=big%
	,paper=a4%
	,parskip=false%
	,headsepline=true%
	,twoside=false%
	,toc=bibliography
	]{scrartcl}

\addtokomafont{sectioning}{\normalfont\bfseries}
\setkomafont{title}{\normalfont\LARGE}
\setkomafont{subtitle}{\normalfont\Large}
\usepackage{authblk}

\usepackage{csquotes}
\usepackage[english]{babel}
\usepackage{graphicx}
\graphicspath{{Pictures/}{/}}
\usepackage[usenames,dvipsnames]{xcolor}
\usepackage{tikz-cd}
\usepackage[draft]{fixme}
\usepackage[nointlimits]{amsmath}
\usepackage{amssymb,mathrsfs,mathtools}
\usepackage{txfonts}
\usepackage{upgreek}
\usepackage{braket}
\usepackage{cancel}
\usepackage{siunitx}
\usepackage{aliascnt} %alias counter
\usepackage[amsmath,thmmarks,hyperref,thref]{ntheorem}
\usepackage[expansion=true,protrusion=true]{microtype}
\usepackage{xkeyval}

\newcommand{\Vol}{{\on{Vol}}}
\newcommand{\dVol}{{d_{\on{Vol}}}}
\newcommand{\vol}{{\on{vol}}}

\newcommand{\Gram}{{\mathbf{G}}}

\newcommand{\PosDef}{{P}}
\newcommand{\dPosDef}{{d_P}}
\newcommand{\SymBil}{\on{Sym}}

\newcommand{\Inj}{{\on{Mon}}}
\newcommand{\orGr}{\on{Gr}}
\newcommand{\Ray}[1][1]{{G_{#1}}}
\newcommand{\dRay}[1][1]{{d_{G_{#1}}}}
\newcommand{\dinfty}{{d_\infty}}

\newcommand{\PosDefSec}{{\cP}}
\newcommand{\dPosDefSec}{{d_\cP}}

\newcommand{\Imm}{{\on{Imm}}}

\newcommand{\dImm}{{d_{\on{Imm}}}}
\newcommand{\Immg}{{\on{Imm}_{\Curve}}}

\newcommand{\strongImm}{{\on{Imm^\times}}}
\newcommand{\dstrongImm}{{d_{\on{Imm}}^\times}}

\newcommand{\Diff}{{\on{Diff}}}

\newcommand{\Shape}[1][]{{\on{Shape}_{#1}}}
\newcommand{\dShape}[1][]{{d_{\on{Shape}_{#1}}}}
\newcommand{\ShapeQuot}[1][]{{\varPi}}

\newcommand{\strongShape}{{\on{Shape}}}
\newcommand{\dstrongShape}{{d_{\on{Shape}}}}
\newcommand{\strongShapeQuot}{{\varPi}}

\newcommand{\res}{\on{res}}

\newcommand{\invisible}[1]{}

\newcommand{\AmbSpace}{\R^m}

\let\originalleft\left
\let\originalright\right
\renewcommand{\left}{\mathopen{}\mathclose\bgroup\originalleft}
\renewcommand{\right}{\aftergroup\egroup\originalright}

\newcommand{\barB}{\bar{B}}
\newcommand{\MinF}[2]{\mathcal{M}_{#1}^{#2}}
\newcommand{\inter}[1][n]{{\mathcal{S}_{#1}}}
\newcommand{\rever}[1][n]{{\mathcal{R}_{#1}}}

\newcommand{\reverwo}{\mathcal{R}}
\newcommand{\treverwo}{\widetilde{\mathcal{R}}}
\newcommand{\interwo}{\mathcal{S}}

\newcommand{\qand}{\quad \text{and} \quad}

\newcommand{\Curve}{\gamma}

\newcommand{\pull}{\#}
\newcommand{\push}{\#}

\DeclareMathOperator{\conv}{conv}

\DeclareMathOperator{\Gr}{Gr}

\DeclareMathOperator*{\esssup}{ess\,sup}

\DeclareMathOperator{\argmin}{arg\,min}

\DeclareMathOperator*{\Ls}{Ls}
\DeclareMathOperator*{\Li}{Li}
\DeclareMathOperator*{\Lt}{Lt}

\newcommand{\converges}[1]{\stackrel{#1}{\longrightarrow}}

\newcommand{\transp}{{\mathsf{T}}}
\newcommand{\T}{{\on{T}}}

\newcommand{\cA}{{\mathcal{A}}}
\newcommand{\cB}{{\mathcal{B}}}
\newcommand{\cC}{{\mathcal{C}}}
\newcommand{\cD}{{\mathcal{D}}}

\newcommand{\cF}{{\mathcal{F}}}
\newcommand{\cG}{{\mathcal{G}}}
\newcommand{\cH}{{\mathcal{H}}}

\newcommand{\cJ}{{\mathcal{J}}}
\newcommand{\cK}{{\mathcal{K}}}

\newcommand{\cM}{{\mathcal{M}}}

\newcommand{\cP}{{\mathcal{P}}}

\newcommand{\cT}{{\mathcal{T}}}

\newcommand{\cX}{{\mathcal{X}}}

\DeclareMathOperator{\dom}{dom}

\DeclareMathOperator{\Hess}{Hess}
\newcommand{\dd}{{\operatorname{d}}}

\newcommand{\at}{|}

\newcommand{\Mat}{{\operatorname{Mat}}}

\newcommand{\Ortho}{{\operatorname{O}}}

\newcommand{\Gl}{{\operatorname{GL}}}

\DeclareMathOperator{\Sym}{Sym}

\newcommand{\ee}{{\operatorname{e}}}

\newcommand{\ceq}{\coloneqq}

\newcommand{\R}{{\mathbb{R}}}

\newcommand{\N}{\mathbb{N}}

\DeclareMathOperator{\id}{id}

\newcommand{\abs}[1]{\left\lvert#1\right\rvert} %absolute value
\newcommand{\nabs}[1]{\lvert{#1}\rvert} %absolute value
\newcommand{\bigabs}[1]{\big\lvert{#1}\big\rvert} %absolute value
\newcommand{\norm}[1]{\left\lVert#1\right\rVert}
\newcommand{\nnorm}[1]{\lVert{#1}\rVert}

\newcommand{\innerprod}[1]{\left\langle#1\right\rangle}
\newcommand{\ninnerprod}[1]{\langle{#1}\rangle}

\newcommand{\paren}[1]{{}\left(#1\right){}}
\newcommand{\nparen}[1]{{}(#1){}}
\newcommand{\bigparen}[1]{{}\big(#1\big){}}
\newcommand{\biggparen}[1]{{}\bigg(#1\bigg){}}
\newcommand{\Bigparen}[1]{{}\Big(#1\Big){}}

\newcommand{\intervaloo}[1]{\left]#1\right[}
\newcommand{\intervalco}[1]{\left[#1\right[}
\newcommand{\intervalcc}[1]{\left[#1\right]}
\newcommand{\intervaloc}[1]{\left]#1\right]}

\newcommand{\on}[1]{\operatorname{#1}}

\DeclareMathOperator{\dist}{dist}  %Distance
\DeclareMathOperator{\Hom}{Hom}    %Homomorphisms

\DeclareMathOperator{\tr}{tr}

\newcommand{\mynewtheorem}[4] %{BEZEICHNER}{COUNTER}{TITEL} - Für Numerierung mit \autoref aus dem \hyperref-Packet
{
\newaliascnt{#1}{#2}
\newtheorem{#1}[#1]{#3}
\aliascntresetthe{#1}
\expandafter\def\csname #1autorefname\endcsname{%
#4%
}%
}

\newtheorem{theorem}{Theorem}[section]
\mynewtheorem{theo}{theorem}{Theorem}{Theorem} 
\mynewtheorem{lem}{theorem}{Lemma}{Lemma} 
\mynewtheorem{lemma}{theorem}{Lemma}{Lemma} 
\mynewtheorem{prop}{theorem}{Proposition}{Proposition} 
\mynewtheorem{cor}{theorem}{Corollary}{Corollary}
\mynewtheorem{corollary}{theorem}{Corollary}{Corollary}
\mynewtheorem{quest}{theorem}{Question}{Question}

\mynewtheorem{problem}{theorem}{Problem}{Problem}

\theoremstyle{break}
\mynewtheorem{btheo}{theorem}{Theorem}{Theorem} 
\mynewtheorem{blem}{theorem}{Lemma}{Lemma}
\mynewtheorem{bcor}{theorem}{Corollary}{Corollary}
\mynewtheorem{bproblem}{theorem}{Problem}{Problem}

\theoremstyle{plain}
\theorembodyfont{\normalfont}
\mynewtheorem{dfn}{theorem}{Definition}{Definition}
\mynewtheorem{definition}{theorem}{Definition}{Definition}
\mynewtheorem{ex}{theorem}{Example}{Example}
\mynewtheorem{example}{theorem}{Example}{Example}
\mynewtheorem{rem}{theorem}{Remark}{Remark}
\mynewtheorem{remark}{theorem}{Remark}{Remark}

\theoremstyle{break}
\mynewtheorem{bdfn}{theorem}{Definition}{Definition}
\mynewtheorem{bex}{theorem}{Example}{Example}
\mynewtheorem{brem}{theorem}{Remark}{Remark}
\theoremheaderfont{\itshape}
\theorembodyfont{\upshape}
\theoremstyle{nonumberplain}
\theoremseparator{.}
\theoremsymbol{\ensuremath{\Box}}
\newtheorem{proof}{\textsc{Proof}}

\usepackage[
	backend=bibtex%
	,isbn=false%
	,doi=false%
	,url=false%
	]{biblatex}
\usepackage[final,%
	bookmarks,%
	bookmarksdepth=3,%
	breaklinks=true,%
	colorlinks=true,%
	urlcolor=RoyalBlue,%
	linkcolor=RoyalBlue,%
	citecolor=ForestGreen,%	
]{hyperref}%
\usepackage[all]{hypcap}

\addto\extrasenglish{%
}

\renewcommand{\qed}{}
\renewcommand{\Affilfont}{\normalsize  \normalfont}

\begin{document}

\title{Variational Convergence of Discrete Minimal Surfaces}

\author[1]{Henrik Schumacher\thanks{
\href{mailto:henrik.schumacher@uni-hamburg.de}{henrik.schumacher@uni-hamburg.de}
}}%
\author[2]{Max Wardetzky\thanks{
\href{mailto:wardetzky@math.uni-goettingen.de}{wardetzky@math.uni-goettingen.de}
}}
\affil[1]{Centre for Optimization and Approximation,
University of Hamburg}

\affil[2]{Institute for Numerical and Applied Mathematics,
Georg-August University G\"{o}ttingen}

\maketitle

\begin{abstract}
Building on and extending  tools  from variational analysis, we prove Kuratowski convergence of sets of simplicial area minimizers to minimizers of the smooth Douglas-Plateau problem under simplicial refinement.  This convergence is with respect to a topology that is stronger than uniform convergence of both positions and surface normals.
\end{abstract}

% !TEX root = main.tex

\section{Introduction}
\label{intro}
The question of finding surfaces of minimum area for a given boundary in $\R^m$ is an extensively studied problem, at least since the work of Lagrange: Let a finite set $\varGamma= \set{\varGamma_1, \varGamma_2, \varGamma_3,\dotsc}$ of closed embedded curves in $\R^m$ be given.
Among all surfaces of \emph{prescribed} topology spanning~$\varGamma$ find those with least (or more precisely critical) area. Solutions of this problem are termed \emph{minimal surfaces}.
In the 1930s, Rad\'{o}~\cite{Rado1933} and Douglas~\cite{MR1501590} independently solved the least area problem for disk-like, immersed surfaces by showing existence of minimizers:
Let $D$ be the unit disk and let $\varGamma \subset \R^m$ be a simple closed curve. Then there exists an area minimizer in
\begin{align*}
	\set{f \in C^0(\bar D;\R^m) \cap C^2(D;\R^m) | f(\partial D)=\varGamma}.
\end{align*}
Douglas' proof involves minimizing the Dirichlet energy $\cD(f) = \int_D \abs{\dd f}^2 \dd x$ and a considerable amount of conformal mapping theory. Less is known about the existence of minimizers of more general least area problems (or least volume problems for the case of higher dimensional manifolds immersed into $\R^m$, which we also treat here). Indeed, one of the difficulties consists in the fact that minimizers in a prescribed topological class might simply not exist.  

A natural question to ask is how to compute minimal surfaces using finite dimensional approximations. Indeed, already Douglas  \cite{MR1502829} followed this approach using finite differences. A more flexible option is to consider a given finite set $\varGamma$ of embedded {boundary} curves in $\R^3$, followed by spanning a triangle mesh into $\varGamma$ and moving the positions of interior vertices such that the overall area of the triangle mesh is minimized. Following this approach, Wagner \cite{MR0458923} applied Newton-like methods for finding critical points of the area functional, Dziuk \cite{MR1083523} and  Brakke \cite{MR1203871} applied $L^2$-gradient descent (the discrete mean curvature flow) in order to produce discrete minimizers, and Pinkall and Polthier  \cite{MR1246481} presented an iterative algorithm for the minimization of area that can be interpreted as $H^1$-gradient descent. With these tools at hand, the question remains whether the so obtained discrete minimizers (e.g., those depicted in \autoref{fig:Borro}) converge to smooth minimal surfaces and if so in which sense?

\begin{figure}[t]
\begin{center}
\begin{minipage}{0.32\textwidth}
\includegraphics[width=\textwidth]{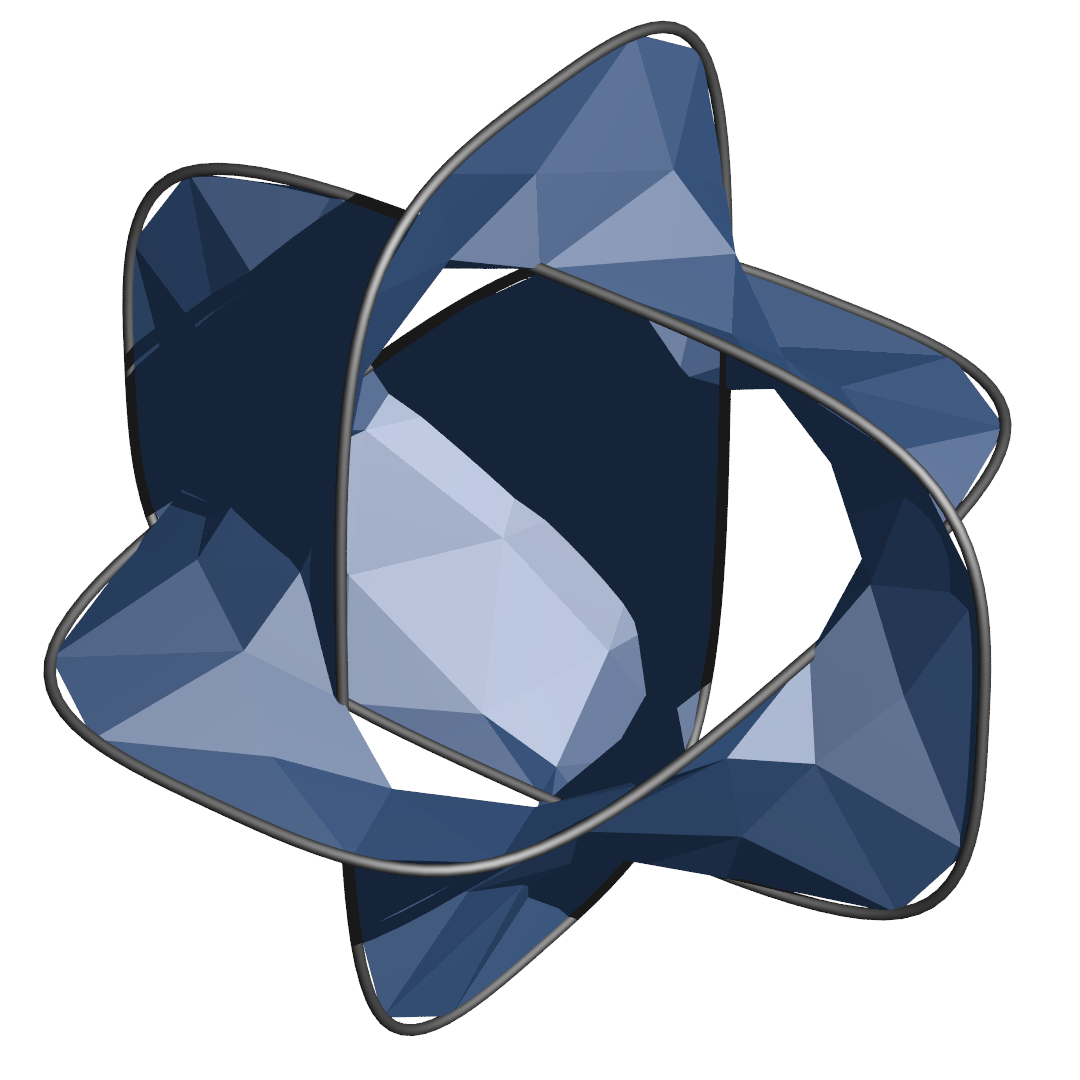}
\end{minipage}
\begin{minipage}{0.32\textwidth}
\includegraphics[width=\textwidth]{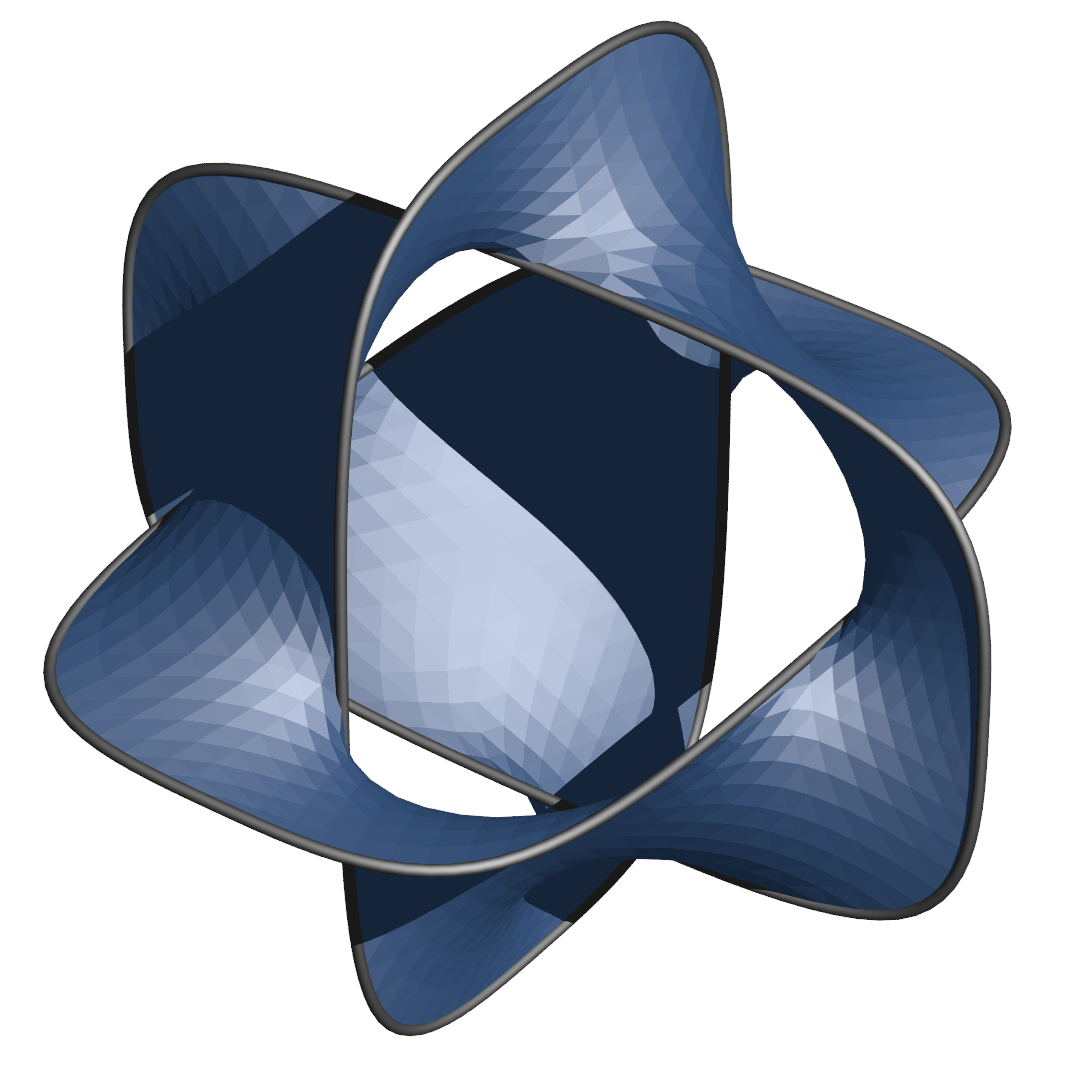}
\end{minipage}
\begin{minipage}{0.32\textwidth}
\includegraphics[width=\textwidth]{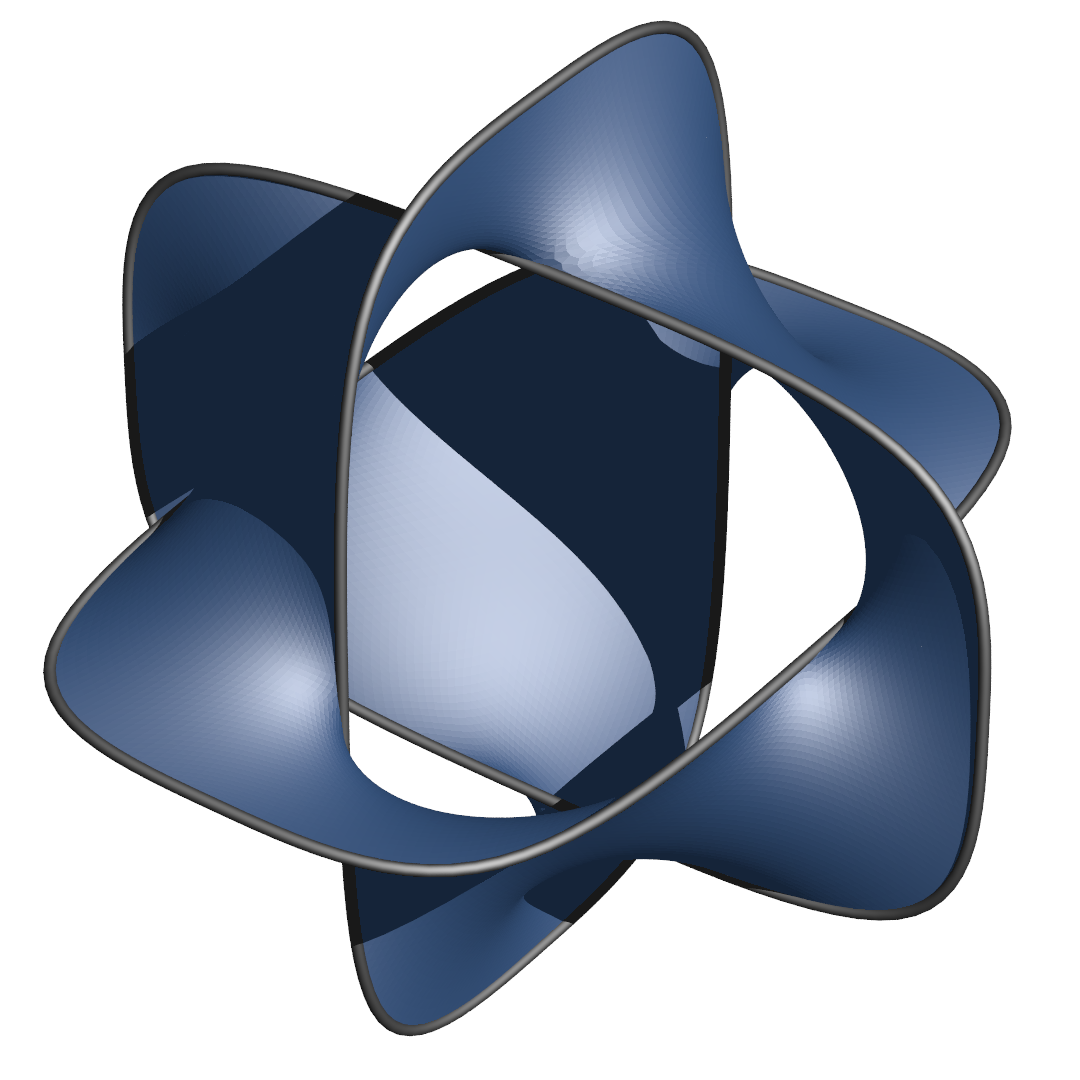}
\end{minipage}
\caption{Some minimizers of the discrete least area problem with Borromean rings as boundary at increasing mesh resolutions.
}
\label{fig:Borro}
\end{center}
\end{figure}

The only approach for which such convergence has been established is based on Douglas' existence proof for disk-like minimal surfaces: Instead of the area of (unparameterized) surfaces, the Dirichlet energy of surface parameterizations is minimized under the constraint of the so-called \emph{three point condition}.
Several authors utilize this idea in order to compute numerical approximations of minimal surfaces via finite element analysis, e.g., Wilson \cite{MR0137309}, Tsuchyia \cite{MR890259}, Hinze \cite{MR1379282}, Dziuk and Hutchinson \cite{MR1613695}, \cite{MR1613699}, and Pozzi \cite{MR2126207}. 

However, these energy methods (which are based on minimizing the Dirichlet energy instead of the area functional) face certain difficulties:
\begin{itemize}
	\item In dimension greater than two, Dirichlet energy is no longer conformally invariant and minimizers %of the Dirichlet energy
	need not minimize area.
	\item For a non-disk surface, one also needs to vary the surface's conformal structure. 
	E.g., in the case of cylindrical topology, the space of conformal structures can be parameterized by the aspect ratio of a reference cylinder. This makes the cylindrical case accessible to energy methods (see \cite{MR2079604} and \cite{MR2126207}). Alas, more general cases have not been treated so far.
	\item Energy methods do not apply when surface area is coupled to some other, conformally non-invariant functional.
\end{itemize}

\begin{figure}[tp]
\begin{center}
\begin{minipage}{0.40\textwidth}
\includegraphics[width=\textwidth]{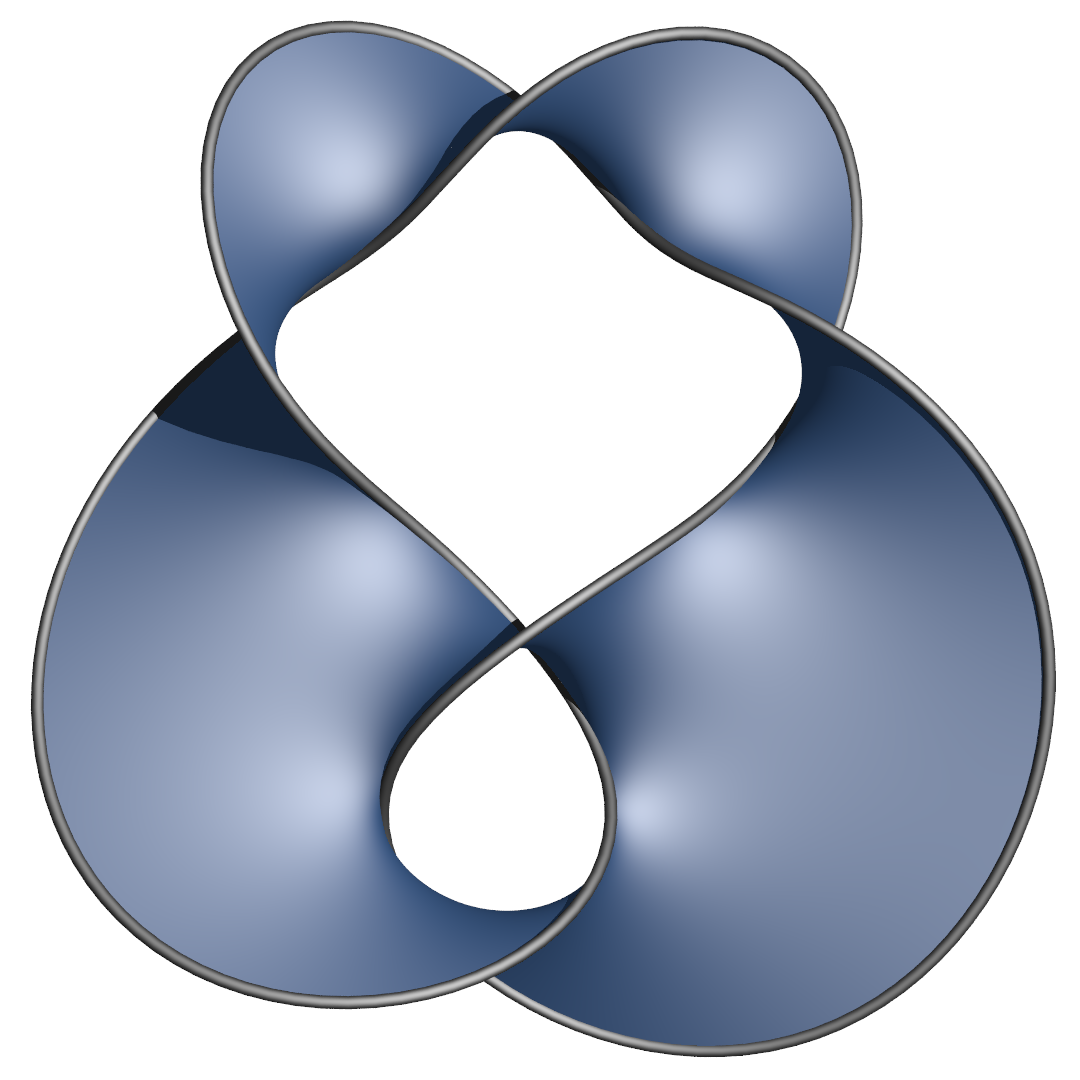}

\includegraphics[width=\textwidth]{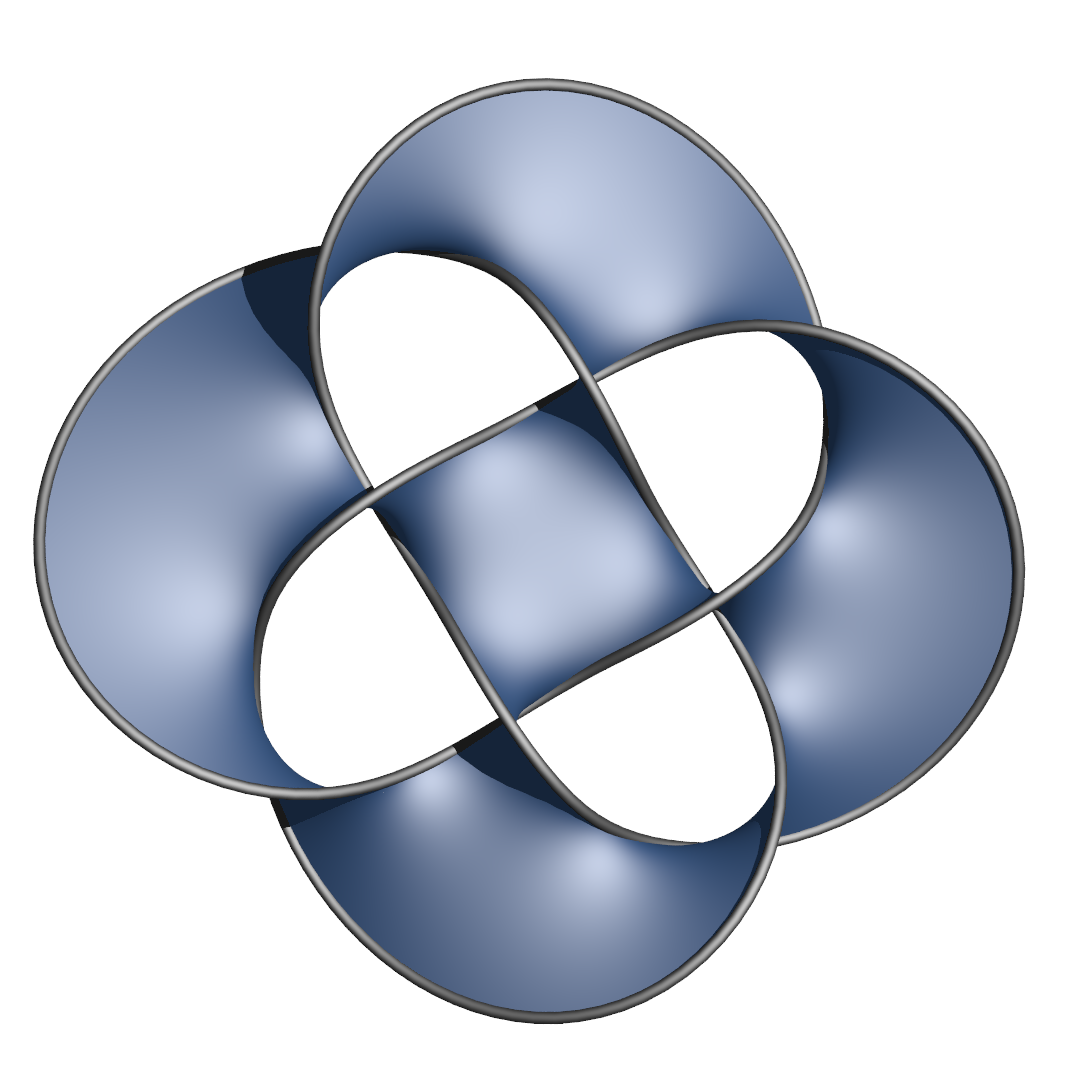}

\includegraphics[width=\textwidth]{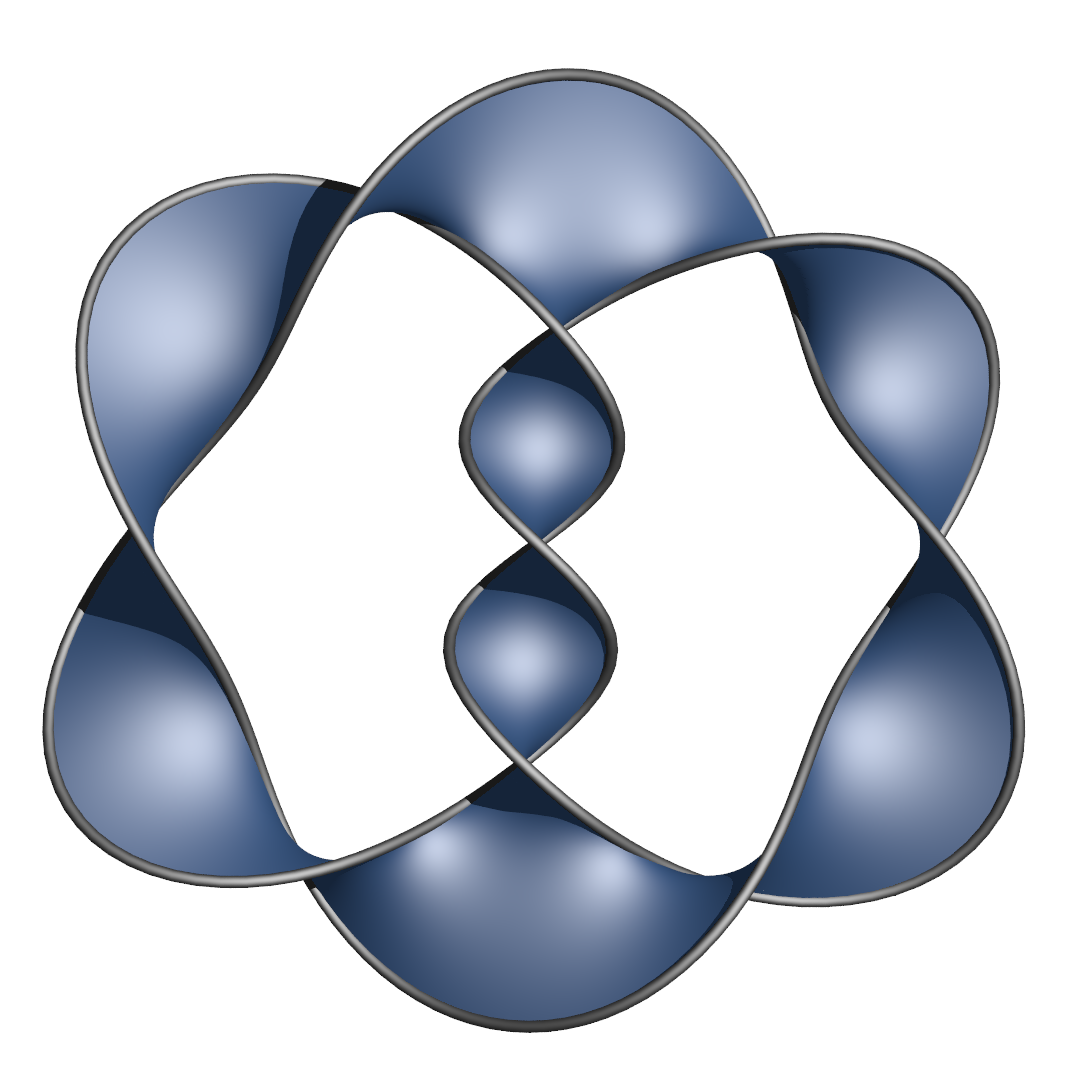}
\end{minipage}
\begin{minipage}{0.40\textwidth}
\includegraphics[width=\textwidth]{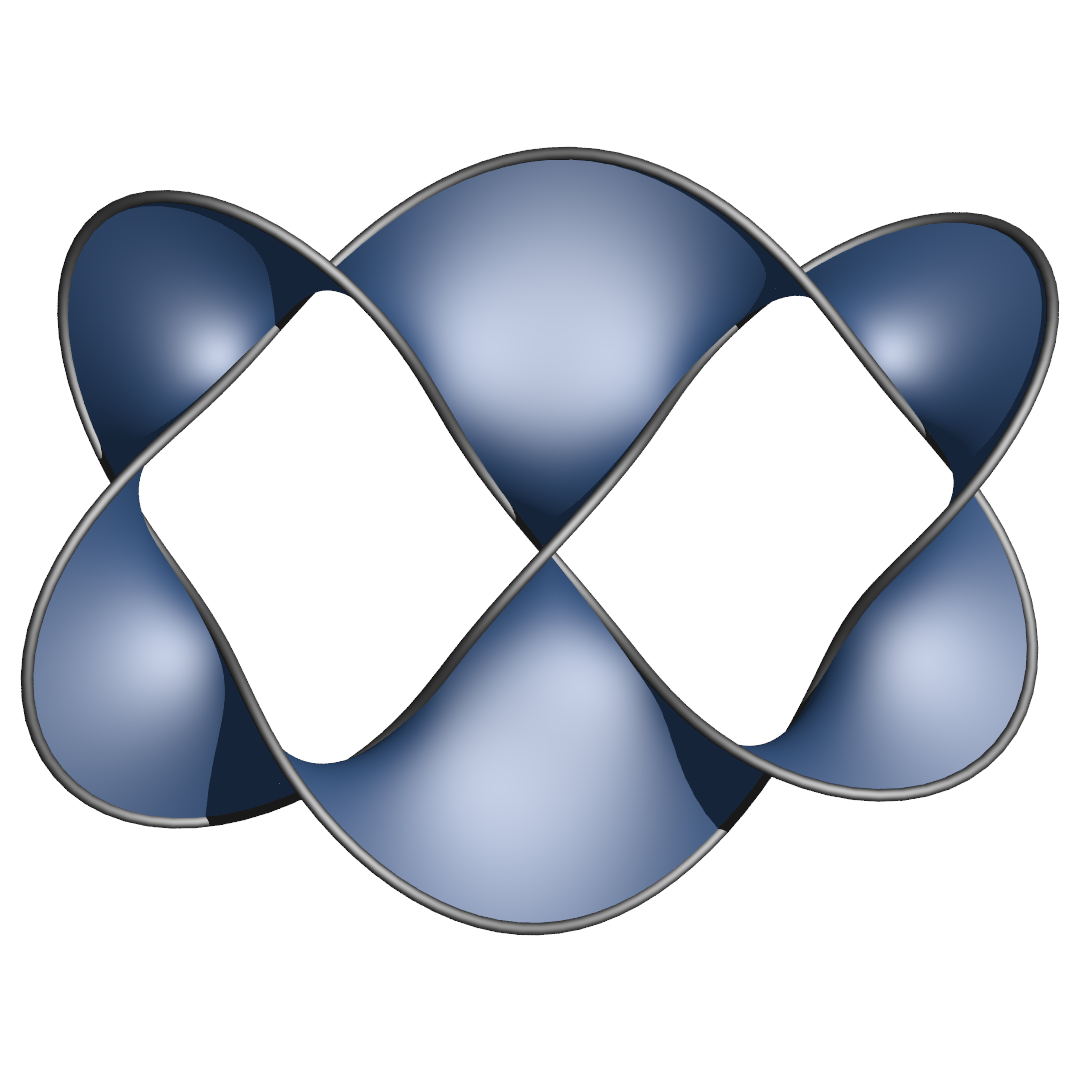}

\includegraphics[width=\textwidth]{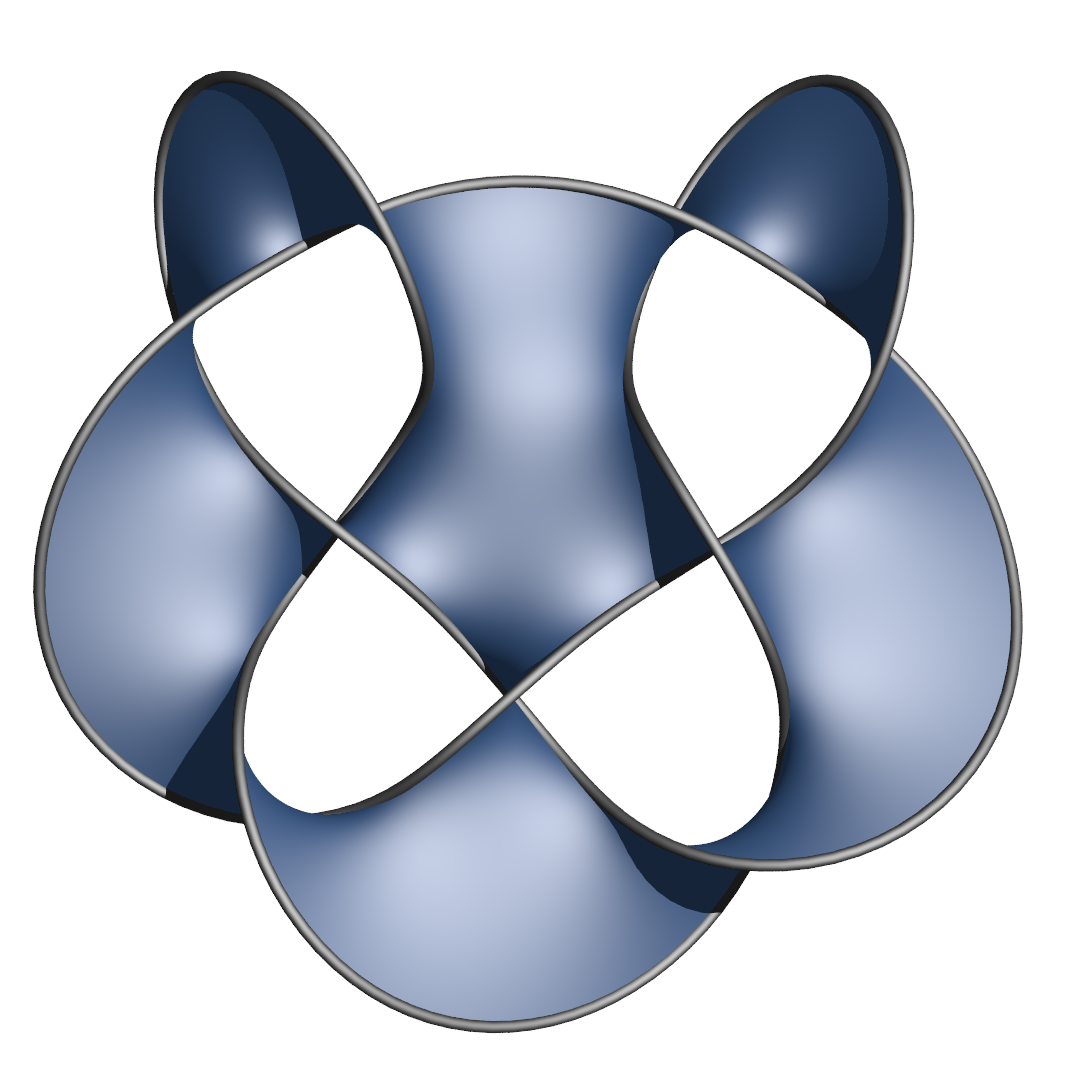}

\includegraphics[width=\textwidth]{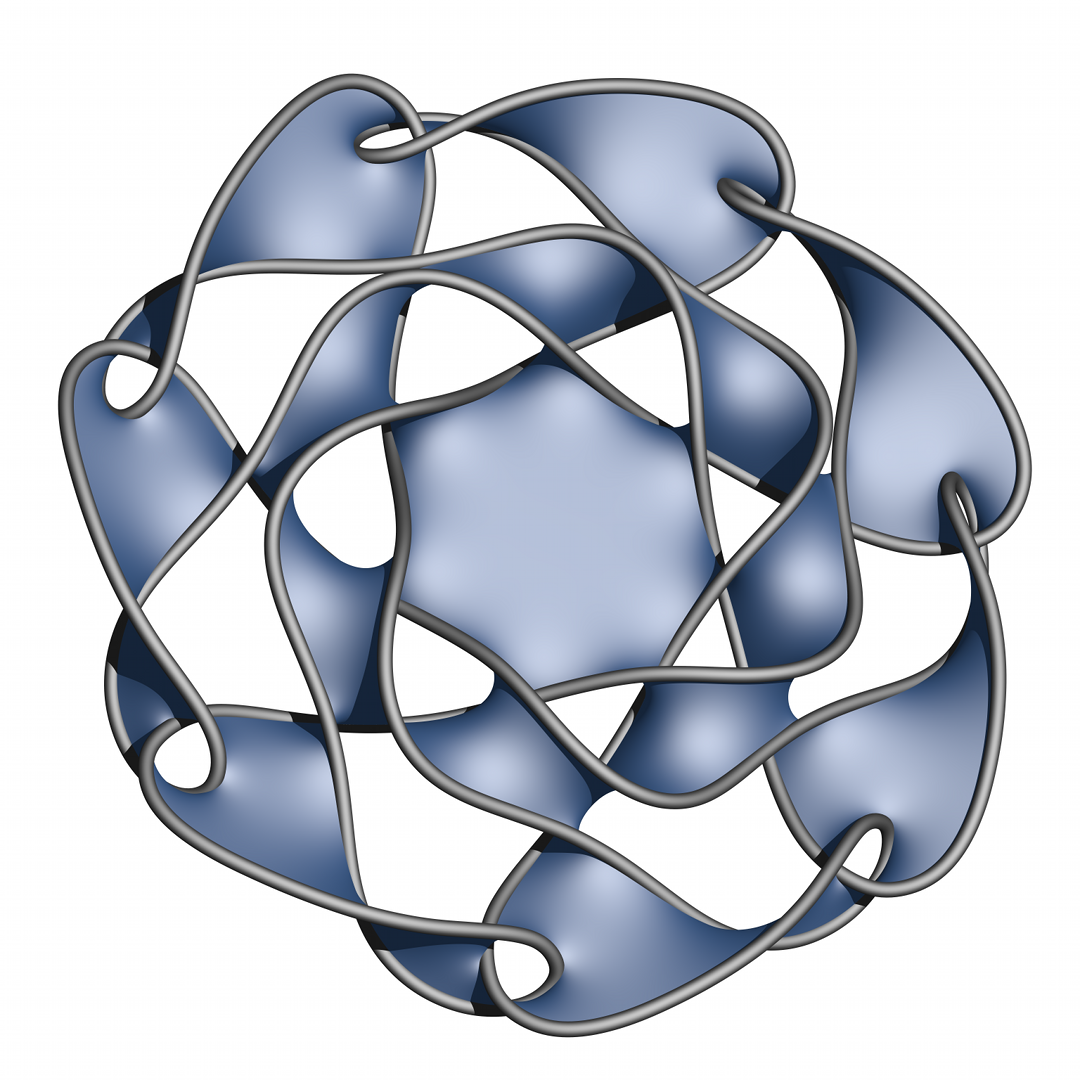}
\end{minipage}
\caption[Orientable and nonorientable discrete minimal surfaces of nontrivial genus.]{Orientable and nonorientable discrete minimal surfaces of nontrivial genus.
Boundary curves by Robert G. Scharein (see \url{http://knotplot.com}.)}
\label{fig:knots}
\end{center}
\end{figure}

In contrast, the Ritz method, i.e., the approach of minimizing area (or volume) among simplicial manifolds, is in principle capable of treating any dimension, codimension, and topological class with a single algorithm (for a variety of examples see \autoref{fig:knots}). This is the approach we follow here. Even non-manifold examples can be treated with this method (for examples see \autoref{fig:Cube} and \cite{MR1246481}). These advantages come at a cost, however: Showing convergence of the Ritz method is hampered by the following difficulties:
\begin{itemize}
	\item Simplicial manifolds capture smooth boundary conditions only in an approximate sense; hence, they cannot be utilized to minimize area (or volume) in the space of surfaces with smooth prescribed boundaries.
	 \item  Smooth minimal surfaces are known to satisfy strong regularity properties (e.g., they are analytic for sufficiently nice boundary data), leading to the question in which space and topology smooth and simplicial area minimizers ought to be compared.
	 \item In general, the least area problem is far from being convex.
	 \item Area minimizers need neither be unique nor isolated; rather, they are \emph{sets} in general.
\end{itemize}
These obstacles render the use of convex optimization approaches and monotone operators 
inappropriate (if not impossible) for showing convergence of discrete (i.e., simplicial) area minimizers. 

For these reasons, we suggest a different route for exploring convergence of discrete minimizers which in particular is capable of dealing with convergence of sets. Building on %(and extending) techniques from nonconvex optimization and in particular 
variational analysis, we prove \emph{Kuratowski convergence} of discrete area (and volume) minimizers  to their smooth counterparts. 
A sequence of sets $A_n$ in a topological space Kuratowski converges to a set $A$ if and only if all cluster points of $(A_n)_{n \in \N}$ belong to $A$ and if each $a \in A$ is %not only a cluster point but also 
additionally a \emph{limit} point of a sequence $(a_n)_{n \in \N}$ with $a_n \in A_n$. While Kuratowski convergence is weaker than Hausdorff convergence in general, both notions coincide in compact metric spaces. Kuratowski convergence is related to the perhaps more familiar notion of $\Gamma$-convergence: A sequence of functionals $\Gamma$-converges if and only if their epigraphs converge in the sense of Kuratowski. 

We establish the requisite concepts from variational analysis in Section~\ref{sec:paramopt}, where we introduce the notions of \emph{consistency} and  \emph{stability}, following the often repeated mantra from numerical analysis that consistency and stability imply convergence. In our setting, consistency refers to the existence of sampling and reconstruction operators $\inter$ and $\rever$ that take smooth manifolds to simplicial ones and vice-versa, respectively, such that the discrete and smooth area functionals $\cF_n$ and $\cF$ satisfy $\cF \approx \cF_n\circ \inter$ and $\cF_n \approx \cF \circ \rever$ in a sense made precise below. Stability refers to a notion of growth of sublevel sets of the smooth area functional $\cF$  near its (set of) minimizers. Additionally, we require the notion of \emph{proximity}, which is motivated by finding a space in which discrete and smooth minimizers can be compared. We suggest this comparison to be made in some metric space $(\cX, d_\cX)$, together with certain mappings $\varPsi_n$ and $\varPsi$ that take discrete and smooth surfaces to $(\cX, d_\cX)$, respectively. 
In particular, we are faced with the problem of choosing $d_\cX$: If the induced topology is too coarse, then convergence might be readily established but geometrically insignificant; vice-versa, a too fine topology may prohibit convergence. 

A balanced choice of $(\cX, d_\cX)$ is the main subject of Section~\ref{sec:ShapeSpace}, guided by the observation that the volume functional of an immersed surface is \emph{independent} of any parameterization. We choose $\cX = \strongShape(\varSigma;\R^m)=\strongImm(\varSigma; \R^m)/\Diff(\varSigma)$ as the space of \emph{(strong) Lipschitz immersions} $\strongImm(\varSigma)$ of a given manifold $\varSigma$ modulo bi-Lipschitz homeomorphisms $\Diff(\varSigma)$. This choice has the additional advantage that both simplicial and smooth immersions fall into the Lipschitz category. 
We construct $d_\cX$ from a parameterization invariant distance on $\strongImm(\varSigma)$ and we show that the resulting quotient semi-metric $\dstrongShape$ on  $\strongShape(\varSigma; \R^m)$ is indeed a proper and complete metric (see \autoref{quotisHsd}). Our construction is such that the topology generated by $\dstrongShape$ is stronger than the topology of uniform convergence of both positions and normals. 

Other proposals for distances on shape spaces, in particular on spaces of immersed circles, can be found, e.g., in the works by Michor and Mumford (\cite{MR2148075}, \cite{MR2333829}, and \cite{MR3098790}).
There, the authors focus on the study of geodesics with respect to weak Riemannian structures on $\Imm(\varSigma; \R^m) \cap C^\infty(\varSigma;\R^m)$ and on curvature properties implied by these. However, the induced geodesic distances do either not descent to a proper metric on shape space (see \cite{MR2148075}) or the metrics involve more than one derivative, rendering them inappropriate for the use with simplicial manifolds.

In Section~\ref{sec:MinimalSurfaces} we combine the main threads of Sections~\ref{sec:paramopt} and~\ref{sec:ShapeSpace} by constructing sampling and reconstruction operators.
For these operators, together with our choice of $(\cX, d_\cX)$, we prove consistency, proximity, and stability, leading to Kuratowski convergence of discrete area minimizers. To this end we rely on certain a~priori information $\cA$ and $\cA_n$ of smooth and discrete minimizers.
In particular, for smooth minimizers we assume the existence of a parameterization of Sobolev class $W^{2, \infty}$ whose differential has singular values uniformly bounded from above and below. Analogously, for discrete minimizers we assume that all embedded simplices have uniformly bounded aspect ratio. As a consequence of Kuratowski convergence we obtain our main result: Every cluster point of discrete area minimizers is a smooth minimal surface and every smooth minimal surface that globally minimizes area is the limit of a sequence of discrete area minimizers (see \autoref{theo:ConvofMinimalSurfaces}). 

\clearpage

% !TEX root = main.tex

\section{Parameterized Optimization Problems}\label{sec:paramopt}

Let $\cC$ be a topological space and let $\cF \colon \cC \to \R$ be a function. We denote the set of minimizers by
\begin{align*}
	\MinF{}{} \ceq \argmin (\cF) = \set{ x \in \cC | \cF(x) = \inf (\cF)}
\end{align*}
and sets of $\delta$-minimizers by
\begin{align*}
	\MinF{}{\delta} \ceq \argmin^\delta (\cF) = \set{ x \in \cC | \forall y \in \cC \colon \; \cF(x) \leq \cF(y) + \delta},
	\quad \text{for $\delta \in \intervalcc{0,\infty}$.}
\end{align*}
Moreover, let topological spaces $\cC_n$ and functions $\cF_n \colon \cC_n \to \R$ with minimizers $\MinF{n}{}\ceq \argmin(\cF_n)$ and $\delta$-minimizers $\MinF{n}{\delta}\ceq \argmin^\delta(\cF_n)$ be given for each $n \in \N$. One may think of $\cF_n$ as small perturbations of $\cF$ or---in the context of Ritz-Galerkin methods---as a discretization of $\cF$.
We are interested in the behavior of the sets $\MinF{n}{}$ as $n \to \infty$.
%\footnote{We point out that this framework also covers the more general problem of analyzing $\lim_{\lambda \to \lambda_\infty} \MinF{\lambda}{}$, where  $(\cF_\lambda \colon \cC_\lambda \to \R)_{\lambda \in \varLambda}$ is a family of functions parameterized by a sequential topological space $\varLambda$. 
%For the sake of brevity, we only discuss $\varLambda= \N \cup \{\infty\}$ with the topology $\mathfrak{T} = \{\emptyset\} \cup \set{ U \subset \N \cup \{\infty\} | \text{$\infty \in U$ and $\on{card}(U) = \infty$}}$.}  
Ideally, $\MinF{n}{}$ should approximate $\MinF{}{}$ ``in some way''.
Therefore, we need a method to relate these sets. A rather general way is letting $\cC$ and $\cC_n$ 
communicate with each other via some mappings
\begin{align*}
	\inter \colon \dom(\inter) \subset \cC \to \cC_n
	\qand
	\rever \colon \dom(\rever) \subset \cC_n \to \cC. 
\end{align*}
We are going to refer to $\inter$ as the \emph{sampling operator} and to $\rever$ as the \emph{reconstruction operator}.
If one has $\MinF{n}{} \subset \dom(\rever)$ and $\MinF{}{} \subset \dom(\inter)$, the pairs of sets $(\MinF{}{}, \rever(\MinF{n}{}))$ and $(\MinF{n}{}, \inter(\MinF{}{}))$ lie in common spaces $\cC$ and $\cC_n$, respectively. Thus, they can be compared.

\begin{example}
In the finite element method, $\cC_n$ is often a finite-dimensional affine subspace of a Banach space $\cC$, $\rever$ is the canonical embedding and $\inter$ is an interpolation operator such that $\rever \circ \inter$ is a projection onto the ansatz space $\rever(\cC_n)$. Then, the Hausdorff distance between $\MinF{}{}$ and $\rever(\MinF{n}{})$ with respect to $d_\cC$ yields a canonical measure of approximation.
\end{example}

More generally, one may consider mappings $\varPsi \colon \cC \to \cX$, $\varPsi_n \colon \cC_n \to \cX$ to some metric space $(\cX,d_\cX)$ and analyze the Hausdorff distance between the sets $\varPsi(\MinF{}{})$ and $\varPsi_n(\MinF{n}{})$ therein. For example, $\varPsi$, $\varPsi_n$ could be embeddings into a space $\cX$ whose metric topology is weaker than those of $\cC$, $\cC_n$.
But $\varPsi$, $\varPsi_n$ need not to be injective at all, yielding only partial information: They could also represent restriction or trace mappings, truncations in infinite decompositions (e.g. Fourier or modal representations, projections on subspaces etc.), state variables in physical systems, (locally) averaged quantities or even quotient mappings. We suggest to view $\varPsi$, $\varPsi_n$ as nonlinear variants of \emph{test functions}.

In practice, one might profit considerably from using a~priori information on minimizers (such as higher regularity or energy bounds) in order to achieve quantitative approximation results. We are going to incoorporate a~priori information in the form of subsets $\cA \subset \cC$, $\cA_n \subset \cC_n$ such that $\varPsi(\cA \cap \MinF{}{})$, $\varPsi_n(\cA_n \cap \MinF{n}{})$ contain $\varPsi(\MinF{}{})$, $\varPsi_n(\MinF{n}{})$ respectively.\footnote{In general, it is not required that $\MinF{}{} \subset \cA$ and $\MinF{n}{} \subset \cA_n$ are subsets. In the case that $\varPsi$ is a quotient map, this is a crucial advantage (see e.g. our treatment of minimal surfaces in \autoref{sec:MinimalSurfaces}).}

We summarize the information given so far in the following (not necessarily commutative) diagrams
\begin{equation}
\begin{tikzcd}[column sep=scriptsize, row sep=scriptsize]
	\cA \cap \MinF{}{} 
		\ar[r, hook]
	& 
	\cA 
		\ar[r, hook]
	& 
	\dom(\inter) 
		\ar[rr, hook]
		\ar[ddrr, "\inter", near start]
	& 
	&
	\cC 
		\ar[dr, "\cF"]
	\\ 
	&& && &\R
	\\
	\cA_n \cap \MinF{n}{} 
		\ar[r, hook]
	& 
	\cA_n 
		\ar[r, hook]
	& 
	\dom(\rever) 
		\ar[rr, hook]
		\ar[uurr, "\rever"', near start] 
	& 
	&
	\cC_n 
		\ar[ur, "\cF_n"']
\end{tikzcd}
	,
	\label{consistency}
\end{equation}
\begin{equation}
\begin{tikzcd}[column sep=scriptsize, row sep=scriptsize]
	\cA \cap \MinF{}{} 
		\ar[r, hook]
	& 
	\cA 
		\ar[r, hook]
	& 
	\dom(\inter) 
		\ar[rr, hook]
		\ar[ddrr, "\inter", near start]
	& 
	&
	\cC 
		\ar[dr, "\varPsi"]
	\\ 
	&& && &\cX
	\\
	\cA_n \cap \MinF{n}{} 
		\ar[r, hook]
	& 
	\cA_n 
		\ar[r, hook]
	& 
	\dom(\rever) 
		\ar[rr, hook]
		\ar[uurr, "\rever"', near start] 
	& 
	&
	\cC_n 
		\ar[ur, "\varPsi_n"']
\end{tikzcd}
	.
	\label{approxi}
\end{equation}

Suppose for the moment that $\emptyset \neq \MinF{}{} \subset \cA$, $\emptyset \neq \MinF{n}{} \subset \cA_n$.
If \eqref{consistency} were commutative, one would have the following implications
\begin{enumerate}
	\item For each $x \in \MinF{}{}$ and $y \in \cC_n$:
	\begin{align*}
		\cF_n\circ \inter(x) = \cF(x) \leq \cF \circ \rever(y) = \cF_n(y),
		\quad
		\text{thus}
		\quad 
		\inter(\MinF{}{}) \subset \MinF{n}{}.
	\end{align*}
	\item For each $y \in \MinF{n}{}$ and $x \in \cC$:
	\begin{align*}
		\cF\circ \rever(y) = \cF_n(y) \leq \cF_n \circ \inter(x) = \cF(x),
		\quad
		\text{thus}
		\quad 
		\rever(\MinF{n}{}) \subset \MinF{}{}.
	\end{align*}
\end{enumerate}
If, in addition, \eqref{approxi} were commutative, this would lead to
\begin{enumerate}
	\setcounter{enumi}{2}
	\item $\varPsi(\MinF{}{}) = \varPsi_n \circ \inter(\MinF{}{}) \subset \varPsi_n(\MinF{n}{})$,
	\item $\varPsi_n(\MinF{n}{}) = \varPsi \circ \rever(\MinF{n}{}) \subset \varPsi(\MinF{}{})$,
\end{enumerate}
hence $\varPsi(\MinF{}{}) = \varPsi_n(\MinF{n}{})$.

Alas, in practice, these diagrams rarely commute. But one may hope that they \emph{almost} commute, i.e., they commute up to some errors that can be uniformly bounded, at least on the sets of a~priori information.
	
In \autoref{sec:consistency} and \autoref{sec:proximity}, we will name these non-commutativity errors and analyze what information can be deduced if these errors are sufficiently small. Afterwards, we will single out additional conditions that guarantee convergence of minimizers (\autoref{sec:stability}). 

% !TEX root = main.tex

\subsection{Consistency}\label{sec:consistency}

We start with the first diagram \eqref{consistency}.
For non-empty sets $\cA \subset \dom(\inter)$ and $\cA_n \subset \dom(\rever)$, there are two qudrilaterals of interest:
\begin{equation}
\begin{tikzcd}[column sep=scriptsize, row sep=scriptsize]
	 &\cC \ar[dr, "\cF"] 
	\\ 
	\cA  \ar[ur,hook] \ar[dr, "\inter"']
	&& \R
	\\
	& \cC_n \ar[ur, "\cF_n"']
\end{tikzcd}
\qand
\begin{tikzcd}[column sep=scriptsize, row sep=scriptsize]
	&\cC \ar[dr, "\cF"] 
	\\ 
	\cA_n  \ar[dr, hook] \ar[ur, "\rever"]
	&& \R
	\\
	& \cC_n \ar[ur,"\cF_n"']
\end{tikzcd}
.	\label{consistentdiagrampair}
\end{equation}
Each quadrilateral is equipped with its own non-cummutativity error:
\begin{definition}[Consistency]
For non-empty sets $\cA \subset \dom(\inter)$, $\cA_n \subset \dom(\rever)$, define
\begin{enumerate}
	\item the \emph{sampling consistency error}
	\begin{align}
		\delta^{\interwo}_n &\ceq \delta(\cF,\cF_n,\inter,\cA) \ceq \sup_{a\in \cA} \,\paren{\cF_n \circ \inter(a) - \cF(a)}^+,\label{ucon1}
	\end{align}
	\item the \emph{reconstruction consistency error}
	\begin{align}
		\delta^{\reverwo}_n &\ceq \delta(\cF,\cF_n,\rever,\cA_n) \ceq \sup_{a\in \cA_n} \,\paren{\cF \circ \rever(a) - \cF_n(a)}^+,\label{ucon2}
	\end{align}
	\item the \emph{total consistency error}
	\begin{align}
		\delta_n \ceq \delta(\cF,\cF_n,\inter,\rever,\cA,\cA_n) &\ceq
		\delta^{\interwo}_n + \delta^{\reverwo}_n,\label{ucon}
	\end{align}
\end{enumerate}
where $t^+ \ceq \max \{t,0\}$ denotes the non-negative part of $t\in \R$.\newline
We say, the sequence $\bigparen{(\cF_n,\inter,\rever)}_{n\in\N}$  is \emph{consistent with respect to $\cF$ on $\bigparen{(\cA,\cA_n)}_{n\in\N}$}, if its consistency error $\delta_n$ converges to $0$ for $n \to \infty$.
In that case, we also say that the sequence $\bigparen{(\cF,\cF_n,\inter,\rever)}_{n\in\N}$ is \emph{consistent on $\bigparen{(\cA,\cA_n)}_{n\in\N}$}.
\end{definition}

\begin{remark}\label{rem:lowerandupperconsistency}
A stronger notion of (total) consistency error (but also one which would be harder to verify) would be
\begin{align*}
	\sup_{a\in \cA} \,\abs{\cF_n \circ \inter(a) - \cF(a)}
	+
	\sup_{a\in \cA_n} \,\abs{\cF \circ \rever(a) - \cF_n(a)}.
\end{align*}
In light of the latter expression, our definition of consistency error could be termed \emph{upper consistency error}. 
However, our definition is sufficient for our needs and we omit ``upper'' for the sake of brevity.
Of course, one may also define a \emph{lower consistency error}, which would be the notion of choice for maximization problems.
\end{remark}

\begin{definition}\label{def:valid}
We call $\cA \subset \cC$ \emph{valid}\emph{ with respect to the pair $(\cF, \inter)$}, if $\emptyset \neq \cA \subset \dom(\inter)$ and $\inf(\cF) = \inf(\cF|_{\cA})$ hold.
Analogously, we call $\cA_n \subset \cC_n$ \emph{valid with respect to the pair $(\cF_n,\rever)$}, if $\emptyset \neq \cA_n \subset \dom(\rever)$ and $\inf(\cF_n) = \inf(\cF_n|_{\cA_n})$ hold.
For the sake of brevity, we will simply say that $\cA$, $\cA_n$ are \emph{valid} whenever $(\cF,\inter)$ and $(\cF_n,\rever)$ can be deduced from the context.
\end{definition}

The requirement that $\inf(\cF) = \inf(\cF|_{\cA})$ and $\inf(\cF_n) = \inf(\cF_n|_{\cA_n})$ are vital for the following result, which we repeatedly require throughout our exposition:

\begin{lemma}\label{infdist}
Let $\cA \subset \cC$, $\cA_n \subset\cC_n$ be valid with respect to 
$(\cF, \inter)$ and $(\cF_n,\rever)$, respectively. Assume that both the sampling consistency error $\delta^{\interwo}_n$ and the reconstruction consistency error $\delta^{\reverwo}_n$ are finite.
Then one has
\begin{align*}
	\inf(\cF_n) \leq \inf(\cF) + \delta^{\interwo}_n
	\quad \text{and} \quad
	\inf(\cF) \leq \inf(\cF_n) + \delta^{\reverwo}_n.
\end{align*}
Hence, one has either $\inf(\cF_n) = \inf(\cF) = -\infty$ or both $\inf(\cF_n)$ and $\inf(\cF)$ are finite with
\begin{align*}
	\abs{\inf(\cF_n) - \inf(\cF)} \leq \max \nparen{ \delta^{\interwo}_n,\delta^\reverwo_n}.
\end{align*}
\end{lemma}
\begin{proof}
Choose a minimizing sequence $(x_m)_{m\in\N}$ in $\cA$ for $\cF$ and a minimizing sequence $(y_m)_{m\in\N}$ in $\cA_n$ for $\cF_n$, i.e., 
\begin{align*}
	\inf(\cF) = \lim_{m \to \infty} \cF(x_m)
	\quad
	\text{and}
	\quad
	\inf(\cF_n) = \lim_{m \to \infty} \cF_n(y_m).
\end{align*}
Then \eqref{ucon1} and \eqref{ucon2} imply
\begin{align*}
	\inf(\cF_n) \leq \cF_n(\inter(x_m)) &\leq \cF(x_m) + \delta^{\interwo}_n \converges{m\to \infty} \inf(\cF) + \delta^{\interwo}_n, 
	\\
	\inf(\cF) \leq \cF(\rever(y_m)) &\leq \cF_n(y_m) + \delta^{\reverwo}_n \converges{m\to \infty} \inf(\cF_n) + \delta^{\reverwo}_n.	
\end{align*}
\qed
\end{proof}
Knowing the total consistency error puts one into the position to compare $\delta$-min\-imizers:
\begin{lemma}\label{lowerlevelsets}
Let $\cA \subset \cC$, $\cA_n \subset\cC_n$ be valid sets.
Denote by $\delta_n$ the total consistency error.
Then one has for $\varrho \in \intervalcc{0,\infty}$
\begin{align*}
	\inter(\cA \cap \MinF{}{\varrho}) \subset \inter(\cA) \cap \MinF{n}{\varrho + \delta_n}
	\quad \text{and} \quad
	\rever(\cA_n \cap \MinF{n}{\varrho}) \subset \rever(\cA_n) \cap \MinF{}{\varrho + \delta_n}.
\end{align*}
\end{lemma}
\begin{proof}
\emph{Case 1:}
$\varrho = \infty$ or $\delta_n = \infty$. The inclusions hold because of $\MinF{n}{\varrho+\delta_n}=\MinF{n}{\infty}=\cC_n$ and $\MinF{}{\varrho+\delta_n}=\MinF{}{\infty}=\cC_n$.
\newline
\emph{Case 2:}
Both $\varrho$ and $\delta_n$ are finite. Then, by \autoref{infdist}, either both $\inf(\cF)$ and $\inf(\cF_n)$ equal $-\infty$ or both of them are finite. 
\newline
\emph{Case 2.a:} $\inf(\cF)=\inf(\cF_n) = -\infty$.
All the sets $\MinF{}{\varrho}$, $\MinF{}{\varrho+\delta_n}$,
$\MinF{n}{\varrho}$, $\MinF{n}{\varrho+\delta_n}$ are empty such that the inclusions hold trivially.
\newline
\emph{Case 2.b:} $\inf(\cF), \inf(\cF_n) > -\infty$.
We discuss only the first inclusion; the second one follows analogously.
In the case that $\cA \cap \MinF{}{\varrho}$ is empty, there is nothing to show. Otherwise, let $x \in \cA \cap \MinF{}{\varrho}$. We apply \autoref{infdist} in order to estimate
\begin{align*}
	\cF_n \circ \inter(x) &\leq \cF(x) + \delta^{\interwo}_n \leq \inf (\cF) + \varrho + \delta^{\interwo}_n
	\leq \inf(\cF_n)  + \delta^{\reverwo}_n + \varrho + \delta^{\interwo}_n
\end{align*}
This leads to $\inter(x) \in \MinF{n}{\varrho+\delta_n}$ which shows the first inclusion. 
\qed
\end{proof}

\begin{remark}
For the moment, it may appear as a superfluous burden to drag along $\inter(\cA)$, $\rever(\cA_n)$ on the right hand side of the previous lemma's conclusions. However, this may be crucial when treating optimization problems with \emph{non-compact} lower level sets as we will see in \autoref{col:cptwouldbefine}. The area functional of immersed surfaces as discussed in \autoref{sec:MinimalSurfaces} is such an example (for a demonstration see \autoref{fig:catastrophe}).
\end{remark}

Before continuing we offer a brief definition of Kuratowski convergence and list properties that we require in the sequel.
More comprehensive treatments of Kuratowski convergence and discussions on the relationship between Kuratowski or \hbox{$K$-convergence} on the one hand and epigraphical or $\Gamma$-convergence on the other hand can be found in \cite{MR1491362} and \cite{MR1201152}.

\begin{definition}
Let $X$ be a topological space and 
denote by $\mathfrak{U}(x)$ the set of all open neighborhoods of $x \in X$.
For a sequence of sets $(A_n)_{n \in \N}$ in $X$ one defines the \emph{limit inferior} or \emph{inner limit} $\Li_{n\to\infty} A_n$ and the \emph{limit superior} or \emph{outer limit} $\Ls_{n\to \infty}A_n$, respectively, in the following way.
\begin{align*}
	\Li_{n \to \infty} A_n
	&\ceq \set{x \in X | \forall U \in \mathfrak{U}(x) \, \exists n \in \N \, \forall k \geq n \colon \; U \cap A_k \neq \emptyset},
	\\
	\Ls_{n \to \infty} A_n
	&\ceq \set{x \in X | \forall U \in \mathfrak{U}(x) \, \forall n \in \N \, \exists k \geq n \colon \; U \cap A_k \neq \emptyset}
%	= \bigcap_{n\in \N} \overline{\bigcup_{k\geq n}A_k}.
\end{align*}
If $A\ceq \Ls_{n \to \infty} A_n = \Li_{n \to \infty} A_n$ agree, one says that \emph{$A_n$ converges to $A$ in the sense of Kuratowski} and writes $\Lt_{n\to \infty} A_n = A$.
\end{definition}
Both the lower and the upper limit are closed sets and one has $\Li_{n \to \infty} A_n \subset \Ls_{n \to \infty} A_n$.
One often refers to $\Ls_{n \to \infty} A_n$ as the \emph{set of cluster points} since $x$ is an element of $\Ls_{n \to \infty} A_n$ 
if and only if there is a sequence of elements $x_n \in A_n$ that has $x$ as a cluster point, i.e., there is a subsequence $(x_{n_k})_{k \in \N}$ that converges to $x$ as $k \to \infty$.
A very useful identity is
\begin{align*} 
	\Ls_{n\to \infty} A_n = \bigcap_{n\in \N} \overline{\bigcup_{k\geq n}A_k}.
\end{align*}
Lower and upper limit are monotone in the sense that
\begin{align*}
	\Li_{n\to \infty} A_n \subset \Li_{n\to \infty} B_n 
	\quad 
	\text{and} 
	\quad 
	\Ls_{n\to \infty} A_n \subset \Ls_{n\to \infty} B_n.
\end{align*}
whenever $A_n \subset B_n \subset X$, for all $n \in \N$.
These definitions allow us to formulate the following corollary.
\begin{corollary}\label{solinlowerlevelsets}
Let $\cA \subset \cC$, $\cA_n \subset \cC_n$ be valid, let $\cF \colon \cC \to \R$ be lower semi-continuous
on the set $\cB \ceq \Ls_{n \to \infty} \rever(\cA_n)$, and let the sequence $\bigparen{(\cF_n,\inter,\rever)}_{n\in\N}$ be consistent on $\bigparen{(\cA,\cA_n)}_{n\in\N}$. 
Then one has for each $\varrho \in \intervalcc{0,\infty}$:
\begin{align*}
%	\bigcap_{n \in \N} \overline{\bigparen{\bigcup_{k\geq n} \rever[k](\cA_k \cap \MinF{k}{\varrho}}} 	
	\Ls_{n \to \infty} 
	\rever(\cA_n \cap \MinF{n}{\varrho})
	\subset \cB \cap \MinF{}{\varrho}.
\end{align*}
\end{corollary}
\begin{proof}
In the case that $\varrho= \infty$, this is obviously true. Hence, let us assume that $\varrho$ is finite.
Denote by $\delta_n$ the total consistency error, $\cK^\varrho_n \ceq \cA_n \cap \MinF{n}{\varrho}$ and let $\eta_n \ceq \sup \set{\delta_k | k \geq n}$. Observe that $\eta_n \searrow 0$ as $n \nearrow \infty$ by consistency. By \autoref{lowerlevelsets}, we have for all $m$, $n \in \N$ with $m \geq n$:
\begin{align*}
	\textstyle
	\rever[m](\cK^\varrho_m) &\subset \paren{\textstyle\bigcup_{k \geq m} \rever[k](\cA_k)} \cap \MinF{}{\varrho + \eta_n}.
\end{align*}
Taking closures and intersections leads to
\begin{align*}
	\bigcap_{n \in \N} \overline{\bigcup_{k\geq n} \rever[k](\cK^\varrho_k)} 
	&= 
	\bigcap_{n \in \N} \bigcap_{m \geq n} \overline{\bigcup_{k\geq m} \rever[k](\cK^\varrho_k)}
	\\
	&\subset \bigcap_{n \in \N} \biggparen{ \Bigparen{\bigcap_{m \geq n}\overline{\bigcup_{k \geq m} \rever[k](\cA_k)}} \cap 
	\overline{\MinF{}{\varrho + \eta_n}}}
	= \bigcap_{n \in \N} \paren{\cB\cap \overline{\MinF{}{\varrho + \eta_n}}}.
%	\\
%	&= \bigcap_{n \in \N}  
%	\overline{ \cB\cap \argmin^{\varrho + \eta_n}(\cF) }
%	= \cB \cap \argmin^\varrho(\cF).
\end{align*}
Because $\cB$ is closed and $\cF|_\cB$ is lower-semicontinuous, one has 
\begin{align*}
	\cB\cap \overline{\MinF{}{\varrho + \eta_n}}
	 = \cB \cap \MinF{}{\varrho + \eta_n}.
\end{align*}
Finally, $\bigcap_{n \in \N} \MinF{}{\varrho + \eta_n} = \MinF{}{\varrho}$ completes the proof.
\qed
\end{proof}
Using $\varrho=0$, this leads immediately to
\begin{corollary}\label{Gammaconvuptorever}
Let $\cA \supset \MinF{}{}$, $\cA_n \supset \MinF{n}{}$ be valid.
Denote by $\delta_n$ the total consistency error.
Then one has
\begin{align*}
	\inter(\MinF{}{}) \subset \inter(\cA)\cap \MinF{n}{\delta_n}
	\quad
	\text{and}
	\quad
	\rever(\MinF{n}{}) \subset \rever(\cA_n)\cap \MinF{}{\delta_n}.
%	\abs{\inf(\cF_n) - \inf(\cF)} \leq \delta_n.
\end{align*}
Assuming consistency of $\bigparen{ (\cF,\cF_n, \inter,\rever) }_{n\in\N}$ on $\bigparen{ (\cA,\cA_n) }_{n\in\N}$ and lower semi-con\-tinu\-ity of $\cF$, cluster points of minimizers of $(\cF_n)_{n\in\N}$ are minimizers of $\cF$ in the sense that
\begin{align*}
%	\bigcap_{n \in \N} \overline{\bigparen{\bigcup_{k\geq n} \rever[k](\MinF{k}{})}} 
	\Ls_{n\to \infty}\rever(\MinF{n}{}) \subset \MinF{}{}.
\end{align*}
\end{corollary}
\invisible{\begin{proof}\autoref{solinlowerlevelsets}\qed\end{proof}}

% !TEX root = main.tex

\subsection{Proximity}\label{sec:proximity}

The result of \autoref{Gammaconvuptorever}  only implies a lower bound for $\MinF{}{}$ by minimizers of $\cF_n$. In the following, we are going to derive an upper bound from the inclusion
\begin{align*}
	\inter(\cA \cap \MinF{}{}) \subset \inter(\cA) \cap \MinF{n}{\delta_n}.
\end{align*}
Note that $\inter(\cA) \cap \MinF{n}{\delta_n} \subset \cC_n$ is not necessarily a subset of $\cC$. However, we may transport the inclusion via $\rever$ to $\cC$; if $\rever \circ \inter$ were the identity on $\cC$ and assuming consistency, one would obtain:
\begin{align*}
	\cA \cap \MinF{}{}
	&= (\rever \circ \inter)(\cA \cap \MinF{}{}) 
	\subset \cA \cap \rever(\MinF{n}{\delta_n}).
\end{align*}
For $\MinF{}{} \subset \cA$, this would yield the desired result
\begin{align*}
	\MinF{}{} \subset \Li_{n\to \infty}\rever(\MinF{n}{\delta_n}).
\end{align*}
In the case that $\cC$ is an infinite-dimensional Banach space and $\cC_n$ a finite-dimen\-sional Banach space, $\rever \circ \inter = \id_{\cC}$ cannot occur. Even worse: In this context, $\rever \circ \inter$ is a compact operator so the sequence cannot converge uniformly to $\id_{\cC}$. Hence we have to establish a sufficiently weak notion for $\rever \circ \inter$ being ``sufficiently close to $\id_\cC$'', a notion that does not imply uniform approximation.

We do this in a slightly more general way by discussing diagram \eqref{approxi}.
At times, it may be instructive for the reader to substitute $\cX = \cC$, $\varPsi = \id_\cC$ and $\varPsi_n= \rever$.
Again, for non-empty sets $\cA \subset \dom(\inter)$, $\cA_n \subset \dom(\rever)$, there are two quadrilaterals of interest:
\begin{equation}
\begin{tikzcd}[column sep=scriptsize, row sep=scriptsize]
	 &\cC \ar[dr, "\varPsi"] 
	\\ 
	\cA  \ar[ur,hook] \ar[dr, "\inter"']
	&& \cX
	\\
	& \cC_n \ar[ur, "\varPsi_n"']
\end{tikzcd}
\qand
\begin{tikzcd}[column sep=scriptsize, row sep=scriptsize]
	&\cC \ar[dr, "\varPsi"] 
	\\ 
	\cA_n  \ar[dr, hook] \ar[ur, "\rever"]
	&& \cX
	\\
	& \cC_n \ar[ur,"\varPsi_n"']
\end{tikzcd}
.
\end{equation}
Note that for our purposes, we do not require $\varPsi$, $\varPsi_n$ to be defined on all of $\cC$, $\cC_n$ but at least on the sets $\cA\cup \rever(\cA_n)$, $\cA_n \cup \inter(\cA)$, respectively. 
Each of the two diagrams has its own non-commutativity error:
 
\begin{definition}[Proximity]	
For non-empty sets $\cA \subset \dom(\inter)$, $\cA_n \subset \dom(\rever)$, define
\begin{enumerate}
	\item the \emph{sampling proximity error}
	\begin{align}
		\varepsilon^\interwo_n
		&\ceq \varepsilon(\varPsi,\varPsi_n,\inter,\cA) 
		\ceq \sup_{a\in \cA} \, d_\cX\paren{\varPsi_n \circ \inter(a) , \varPsi(a)},\label{appr1}
	\end{align}
	\item the \emph{reconstruction proximity error}
	\begin{align}
		\varepsilon^\reverwo_n
		&\ceq \varepsilon(\varPsi,\varPsi_n,\rever,\cA_n) 
		\ceq \sup_{a\in \cA_n} \, d_\cX\paren{\varPsi \circ \rever(a) , \varPsi_n(a)},\label{appr2}
	\end{align}
	\item the \emph{proximity error}
	\begin{align}
		\varepsilon_n \ceq \varepsilon(\varPsi,\varPsi_n,\inter,\rever,\cA,\cA_n) &\ceq \max \nparen{\varepsilon^\interwo_n, \varepsilon^\reverwo_n}.\label{appr}
	\end{align}
\end{enumerate}
We say that the sequence $\bigparen{(\inter,\rever)}_{n\in\N}$ is \emph{proximate with respect to $\bigparen{(\varPsi,\varPsi_n)}_{n\in\N}$ on $\bigparen{(\cA,\cA_n)}_{n\in\N}$}, if its proximity error $\varepsilon_n$ converges to $0$ as $n \to \infty$.
\end{definition}

\begin{definition}\label{dfn:thickenings}
Let $(X,d)$ be a metric space, $A \subset X$ a subset, and $r>0$. We define $\dist(x,A) \ceq \inf \set{d(x,a) | a \in A}$ and the \emph{open} and \emph{closed $r$-thickening} of $A$ by
\begin{align*}
	B(A,r) \ceq \set{x \in X | \dist(x,A) <r }
	\qand
	\barB(A,r) \ceq \set{x \in X | \dist(x,A) \leq r }.
\end{align*}
\end{definition}

\begin{lemma}\label{consandprox}
Let $\cA \subset \cC$, $\cA_n \subset \cC_n$ be valid.
Denote by $\delta_n$ the total consistency error.
Then one has for $\varrho \in \intervalcc{0,\infty}$
\begin{align*}
	\varPsi(\cA \cap \MinF{}{\varrho}) 
	&\subset \bar B \bigparen{\varPsi_n(\inter(\cA) \cap \MinF{n}{\varrho+\delta_n}), \varepsilon^{\interwo}_n},
	\\
	\varPsi_n(\cA_n \cap \MinF{n}{\varrho}) 
	&\subset \bar B \bigparen{\varPsi(\rever(\cA_n) \cap \MinF{}{\varrho+\delta_n}), \varepsilon^{\reverwo}_n}.
\end{align*}
\end{lemma}
\begin{proof}
For $x \in \varPsi(\cA\cap \MinF{}{\varrho})$ (if existent), fix an
$a \in \cA\cap \MinF{}{\varrho}$ with $x =\varPsi(a)$.
According to \autoref{lowerlevelsets}, we have that $\inter(a) \in \inter(\cA) \cap \MinF{n}{\varrho + \delta_n}$. Now, the definition of the sampling proximity error implies 
\begin{align*}
	d_\cX(\varPsi(a),\varPsi_n(\inter(a))) \leq \varepsilon_n^{\interwo},
\end{align*}
thus
$
	\varPsi(a) \in \barB \bigparen{\varPsi_n(\inter(\cA) \cap \MinF{n}{\varrho + \delta_n}) , \varepsilon_n^{\interwo}}.
$
The proof of the second statement is analogous. 
\qed
\end{proof}

\begin{lemma}\label{Kurtowskichain}
In addition to the previous lemma, suppose $\inter(\cA) \subset \cA_n$ and proximity, i.e., $\varepsilon_n \converges{n \to \infty} 0$. Then one has
\begin{align*}
	\varPsi(\cA\cap \MinF{}{})
	&\subset \Li_{n\to \infty} \varPsi_n(\cA_n \cap \MinF{n}{\delta_n})
	\\
	&\subset \Ls_{n\to \infty} \varPsi_n(\cA_n \cap \MinF{n}{\delta_n})
	\subset \Ls_{n\to \infty} \varPsi(\rever(\cA_n) \cap \MinF{}{2\delta_n}).
\end{align*}
\end{lemma}
\begin{proof}
We apply the previous lemma twice: once with $\varrho=0$, once with $\varrho=\delta_n$. By the triangle inequality, one has for all $n \in \N$:
\begin{align*}
	\varPsi(\cA\cap \MinF{}{}) 
	\subset \bar B \bigparen{\varPsi_n(\cA_n \cap \MinF{n}{\delta_n}), \varepsilon^{\interwo}_n}
%	\\	
%	&\subset \bar B \bigparen{\varPsi_n(\cA_n \cap \MinF{n}{\delta_n}), \varepsilon^{\interwo}_n}
	\subset \bar B \bigparen{\varPsi(\rever(\cA_n) \cap \MinF{}{2\delta_n}), \varepsilon^{\interwo}_n+\varepsilon^{\reverwo}_n}.
\end{align*}
Because of the monotonicity properties of $\Li$ and $\Ls$, we may apply $\Li$ and $\Ls$ to the second term and $\Ls$  to the third without invalidating the inclusions.
The statement then follows from thickening robustness (\autoref{lem:Kuratowskicinvafterthickening}). 
\qed
\end{proof}
We now establish conditions that allow us to deduce Kuratowksi convergence (or even Hausdorff convergence).
This will be the focus of \autoref{topstability}, where \autoref{Kurtowskichain} will be used.

% !TEX root = main.tex

\subsection{Stability}\label{sec:stability}

In order to deduce set convergence of $\varPsi_n(\MinF{n}{})$ to $\varPsi(\MinF{}{})$ from \autoref{consandprox} or \autoref{Kurtowskichain},
one requires a reasonable interplay between $\cF$ and $\varPsi$. %(and probably an analogous interplay between $\cF_n$ and $\varPsi_n$). 
We term the presence of such an interplay as \emph{stability}.
In the remaining part of this section, we present a qualitative notion of stability:
It provides a rather weak, purely qualitative condition for Kuratowski convergence and is strongly related to the concept of lower semi-continuity.
For this approach, we will have to transport variational information of $\cF$ forward to $\cX$ along $\varPsi$. This is why we introduce the (variational) pushforward first.

\subsubsection{Pushforward}

\begin{definition}
Let $F \colon Y \to \intervaloc{-\infty,\infty}$ be a function and $\psi \colon Y \to X$ a mapping to a topological space $X$.
With the convention $\inf (\emptyset) = \infty$, define the \emph{(variational) pushforward of $F$ along $\psi$} by
\begin{align*}
	(\psi_\#F)(x) \ceq \inf \set{ F(y) | y\in Y \colon \psi(y) =x} = \inf_{y\in \psi^{-1}(x)} F(y).
\end{align*}
\end{definition}

\begin{example}\label{pushforwardalonginjections}
For injections, the pushforward reduces to the well-known and frequently used extension by infinity:
Assume that $\varPsi \colon \cC \hookrightarrow \cX$ and $\varPsi_n \colon \cC_n \hookrightarrow \cX$ are injections. Then $\varPsi_\# \cF$ and $(\varPsi_n)_\# \cF_n$ are given by
\begin{align*}
	(\varPsi_\# \cF)(x) &= \begin{cases}
		\cF(x), &x \in \cC\\
		\infty, &\text{else}
	\end{cases},
	&
	((\varPsi_n)_\# \cF_n)(x) &= \begin{cases}
		\cF_n(x), &x \in \cC_n\\
		\infty, &\text{else}
	\end{cases}.
\end{align*}
This allows one to treat the optimization problems for $\cF$ and $\cF_n$ on a common space. 
\end{example}

We list some elementary properties of the pushforward:

\begin{lemma}\label{pushforwardproperties}
Let $F \colon Y \to \intervaloc{-\infty,\infty}$ be a function with $\inf (F) < \infty$, $\psi \colon Y \to X$ some mapping and $\varrho \in \intervalco{0,\infty}$. Then one has
$\inf (\psi_\# F) = \inf(F)$ and
\begin{align*}
	\psi(\argmin^\varrho(F)) \subset \argmin^\varrho(\psi_\#F) = \bigcap_{\sigma>\varrho} \psi(\argmin^\sigma(F)).
\end{align*}
Equality holds, e.g., if for each $x \in \argmin^\varrho(\psi_\#F)$, the function $F|_{\psi^{-1}(x)}$ attains its infimum.
\end{lemma}
\begin{proof}
First, note that
$(\psi_\# F)(\psi(y)) = \inf_{a \in \psi^{-1}(\psi(y))} F(a) \leq F(y)$ holds for all $y \in Y$.
This leads to $\inf(\psi_\# F) \leq \inf(F)$ and
\begin{align}
	\psi(\argmin^\varrho(F)) \subset \argmin^\varrho(\psi_\#F)
	\quad \text{for all $\varrho \in \intervalcc{0,\infty}$}.
	\label{eq:consistencyinclusions}
\end{align}
Because of $\inf(\psi_\# F)\leq \inf(F)<\infty$, there exists a sequence $(x_n)_{n \in \N}$ with $(\psi_\#F)(x_n) < \infty$ for all $n \in \N$  and $\liminf_{n\to \infty} (\psi_\#F)(x_n) = \inf(\psi_\#F)$. For each $n \in \N$, $x_n$ has to be in the image of $\psi$ since $(\psi_\#F)(x_n)$ is finite. So, we may choose $y_n \in \psi^{-1}(x_n)$ with $F(y_n) \leq (\psi_\#F)(x_n) + \tfrac{1}{n}$. This leads to
\begin{align*}
	\inf(F) \leq \liminf_{n\to \infty} F(y_n) \leq  \liminf_{n\to \infty} \bigparen{ (\psi_\#F)(x_n) + \tfrac{1}{n}} = \inf(\psi_\# F),
\end{align*}
thus $\inf(F)=\inf(\psi_\# F)$. The case $\inf(F) =-\infty$ is also included.

From now on, let $\varrho \in \intervalco{0,\infty}$ be finite. 
We are going to show 
\begin{align}
	\argmin^\varrho(\psi_\#F) \subset \psi(\argmin^\sigma(F)) 
	\quad 
	\text{for each $\sigma > \varrho$.}\label{eq:consistencyinclusions2}
\end{align}
To this end, let $x \in \argmin^\varrho(\psi_\#F)$.
Since one has $(\psi_\#F)(x) \leq \inf(\psi_\#F) + \varrho = \inf(F) + \varrho < \infty$, there is a minimizing sequence $(y_n)_{n\in \N}$ of $F|_{\psi^{-1}(x)}$, i.e., $y_n \in \psi^{-1}(x)$ and
\begin{align*}
	F(y_n) 
	&\leq  \inf_{a \in \psi^{-1}(x)} F(a) + \tfrac{1}{n} 
	= (\psi_\# F)(x) + \tfrac{1}{n}
	\\
	&\leq \inf(\psi_\# F) + \varrho + \tfrac{1}{n}
	= \inf(F) + \varrho + \tfrac{1}{n}.
\end{align*}
For $n>\frac{1}{\sigma-\varrho}$, one has $y_n \in \argmin^\sigma(F)$ and $x = \psi(y_n) \in \psi(\argmin^\sigma(F))$. This shows $\argmin^\varrho(\psi_\#F) \subset \psi(\argmin^\sigma(F))$ for each $\sigma > \varrho$.
Now, \eqref{eq:consistencyinclusions} and \eqref{eq:consistencyinclusions2} together yield
\begin{align*}
	\argmin^\varrho(\psi_\#F) 
	&\subset \bigcap_{\sigma>\varrho} \psi(\argmin^\sigma(F))
	\subset \bigcap_{\sigma>\varrho} \argmin^\sigma(\psi_\#F)
	=\argmin^\varrho(\psi_\#F)
\end{align*}
which is the first claim.
As for the second claim, let $x \in \argmin^\varrho(\psi_\#F)$ be a point such that the function $F |_{\psi^{-1}(x)}$ attains its infimum, say at $y \in Y$. Then one has for all $z \in Y$ that
\begin{align*}
	F(y) = \inf_{a \in \psi^{-1}(x)} F(a) = (\psi_\# F)(x)
	\leq (\psi_\# F)(\psi (z)) + \varrho \leq F(z) + \varrho,
\end{align*}
showing $y \in  \argmin^\varrho(F)$ and thus $x = \psi(y) \in \psi( \argmin^\varrho(F))$.
\qed
\end{proof}

\begin{remark}\label{rem:closedandcountcpt}
Colloquially, the sufficient condition for equality in the preceding lemma can be restated as: Non-empty $\psi$-slices $\psi^{-1}(x)$, $x \in X$ are ``small'' enough to allow $F|_{\psi^{-1}(x)}$ to be minimizable.

For example, this sufficient condition is met if there is a topology on $Y$ such that for each $\varrho \in \intervalco{0,\infty}$ and each $x \in X$, the intersection $\argmin^\varrho(F) \cap \psi^{-1}(x)$ is closed and countably compact. 

Another relevant case is when $F$ is constant on $\psi$-slices, e.g. when $\psi$ is the quotient map of a group action under which $F$ is invariant.
\end{remark}

\subsubsection{Topological stability}\label{topstability}

\begin{definition}
Let $F \colon Y \to \intervaloc{-\infty,\infty}$ be a function, $\psi \colon Y \to X$ a mapping to a topological space $X$, $K\subset X$ a closed set.
We call $F$ \emph{topologically stable along $\psi$ over $K$}, if
\begin{align*}
	K\cap \psi(\argmin(F)) 
%	= \Lt_{\varrho \to 0} K\cap \psi( \argmin^\varrho(F)).
	= \bigcap_{\varrho > 0} \overline{K\cap \psi( \argmin^\varrho(F))}.
\end{align*}
\end{definition}

The notion of topological stability is a generalization of lower semi-continuity in the context of test mappings. In particular, stability can be understood as "lower semi-continuity at minimizers".

\begin{example}
Let $X$ be a topological space, $K \subset X$ a closed set, and $F \colon X \to \intervaloc{-\infty,\infty}$ a lower semi-continuous function on $K$, i.e., the lower level sets of $F|_K$ are closed in $K$ (and thus in $X$ because $K$ is closed). Then $F$ is topologically stable along $\id_X$ over $K$. 
\end{example}

\begin{example}\label{ex:extensionbyinftystab}
Let $X$ be a topological space, $Y \subset X$ a closed set, and $F \colon Y \to \intervaloc{-\infty,\infty}$ lower semi-continuous on $Y$. Denote by $\psi \colon Y \hookrightarrow X$ the inclusion mapping. Since $\psi_\push F$ is the extension by infinity (see \autoref{pushforwardalonginjections}), it is lower semi-continuous, thus topologically stable on $X$.
\end{example}

\begin{lemma}\label{lem:lscimpliesstability}
Let $Y$, $X$ be topological spaces, $F \colon Y \to \intervaloc{-\infty,\infty}$, $\psi \colon Y \to X$, and $K\subset X$ a closed set.
Assume $\inf(F)<\infty$,
$K\cap \psi(\argmin^\varrho(F)) = K\cap \argmin^\varrho(\psi_\#F)$ for all $\varrho \in \intervalco{0, \infty}$, and that $\psi_\# F$ is lower semi-continuous on $K$.

Then $F$ is topologically stable along $\psi$ over $K$.
\end{lemma}
\begin{proof}
Note that the set $\argmin^\varrho(\psi_\# F)$ is closed for all $\varrho \in \intervalco{0, \infty}$ because $\psi_\# F$ is lower semi-continuous.
One has by \autoref{pushforwardproperties}
\begin{align*}
	K\cap \psi(\argmin(F)) 
	&= K \cap \argmin(\psi_\#F) 
%	= K \cap \Lt_{\varrho \to 0} \psi(\argmin^\varrho(F))
	= \bigcap_{\varrho>0} K \cap \argmin^\varrho(\psi_\# F)
	\\
	&
	= \bigcap_{\varrho>0} \overline{K \cap \argmin^\varrho(\psi_\# F)}
	= \bigcap_{\varrho>0} \overline{K \cap \psi(\argmin^\varrho(F))}
%	= \Lt_{\varrho \to 0} K \cap \psi(\argmin^\varrho(F))
	.
\end{align*}
\qed
\end{proof}

We arrive at the main theorem of this section.

\begin{theorem}
\label{topstabandconv}
Let $\cA \subset \cC$ and $\cA_n \subset \cC_n$ be valid sets and let $\cK \subset \cX$ be a closed set such that
$\varPsi(\rever(\cA_n)) \subset \cK$ holds for all sufficiently large $n$.
Assume  consistency and proximity, i.e., $\delta_n \converges{n \to \infty} 0$ and $\varepsilon_n \converges{n \to \infty} 0$, and topological stability of $\cF$ along $\varPsi$ over $\cK$.
Then one has
\begin{align*}
	\Ls_{n\to \infty}\varPsi_n(\cA_n \cap \MinF{n}{}) 
	\subset \cK \cap \varPsi(\MinF{}{}).
\end{align*}
If $\varPsi(\MinF{}{}) \subset \varPsi(\cA\cap\MinF{}{}) \cap \cK$ and $\inter(\cA) \subset \cA_n$ hold for all sufficiently large $n$, then one has Kuratowski convergence
\begin{align*}
	\varPsi(\MinF{}{}) 
	=
	\Lt_{n\to \infty} \varPsi_n(\cA_n \cap \MinF{n}{\delta_n}).
\end{align*}
\end{theorem}
\begin{proof}
From the second statement of \autoref{consandprox} with $\varrho=0$, we have for sufficiently large $n \in \N$:
\begin{align*}
	\Ls_{n\to \infty} \varPsi_n(\cA_n \cap \MinF{n}{}) 
	&\subset \Ls_{n\to \infty} \barB\bigparen{\varPsi(\rever(\cA_n) \cap \MinF{}{\delta_n}), \varepsilon^{\reverwo}_n}
	\\
	&\subset \Ls_{n\to \infty} \barB(\cK \cap \varPsi(\MinF{}{\delta_n}), \varepsilon^{\reverwo}_n).
\end{align*}
Using \autoref{lem:Kuratowskicinvafterthickening} with $A_n = \cK \cap \varPsi(\MinF{}{\delta_n})$ and $r_n = \varepsilon^{\reverwo}_n$, we obtain
\begin{align*}
	\Ls_{n\to \infty} \barB(\cK \cap \varPsi(\MinF{}{\delta_n}), \varepsilon^{\reverwo}_n)
	= \Ls_{n\to \infty} \cK \cap \varPsi(\MinF{}{\delta_n})
	= \bigcap_{n \in \N} \overline{\cK \cap \varPsi(\MinF{}{\delta_n})}.
\end{align*}
Now, topological stability of $\cF$ along $\varPsi$ over $\cK$ leads to the first statement.
In the same way, one shows
\begin{align*}
	\Ls_{n\to \infty} \varPsi(\rever(\cA_n) \cap \MinF{}{2\delta_n}) 
	\subset 	\Ls_{n\to \infty} \cK \cap \varPsi(\MinF{}{2\delta_n}) 
	\subset \cK \cap \varPsi(\MinF{}{}).
\end{align*}
The condition $\inter(\cA) \subset \cA_n$ allows us to use
\autoref{Kurtowskichain}, leading to
\begin{align*}
	\varPsi(\cA\cap \MinF{}{})
	&\subset \Li_{n\to \infty} \varPsi_n(\cA_n \cap \MinF{n}{\delta_n})
	\\
	&\subset \Ls_{n\to \infty} \varPsi_n(\cA_n \cap \MinF{n}{\delta_n})
	\subset \Ls_{n\to \infty} \varPsi(\rever(\cA_n) \cap \MinF{}{2\delta_n})
	\subset \cK \cap \varPsi(\MinF{}{}).
\end{align*}
If both $\varPsi(\cA\cap\MinF{}{})$ and $\cK$ contain all images of minimizers $\varPsi(\MinF{}{})$, the above chain of inclusions is closed. In particular, $\Lt_{n\to \infty} \varPsi_n(\cA_n \cap \MinF{n}{\delta_n})$ exists and coincides with $\varPsi(\MinF{}{})$.
\qed
\end{proof}

In compact metric spaces, Kuratowski and Hausdorff convergence are equivalent.
This leads us immediately to the following result:
\begin{corollary}\label{col:cptwouldbefine}
In addition to all the conditions in \autoref{topstabandconv}, assume that the sets $\cM$ and $\cA_n \cap \MinF{n}{\delta_n}$ are non-empty for all sufficiently large $n$ and that the set $\cK \subset \cX$  is compact. Then one has Hausdorff convergence, i.e.,
\begin{align*}
	\lim_{n \to \infty} \dist_\cX \paren{ \varPsi(\MinF{}{}) , \varPsi_n(\cA_n \cap \MinF{n}{\delta_n}) } = 0.
\end{align*}
\end{corollary}
\invisible{\begin{proof}\autoref{topstabandconv}\qed\end{proof}}

% !TEX root = main.tex

\section{Shape Space}\label{sec:ShapeSpace}

\subsection{The Space of Inner Products}\label{posdef}

We require a suitable shape space of immersed manifolds for our applications to minimal surfaces. Therefore, we define a space of parameterized (locally) Lipschitz immersions and equip it with a reparameterization-invariant distance. This distance is in terms of zeroth and first derivatives of immersions. It descends to a distance on \emph{shape space}, i.e., the quotient space of unparameterized, immersed manifolds.\footnote{Contrary to the meaning we associate to it, the term ``shape space'' is frequently used for certain classes of subsets of $\R^3$ \emph{modulo the action of the Euclidean group.}}

Before we study Lipschitz immersions we point out some properties of distances between inner products on finite dimensional vector spaces.

Let $V$ be a $k$-dimensional real vector space with $k \in \N$ and let $\PosDef(V)$ denote the space of positive definite bilinear forms on $V$.
The group $\Gl(V)$ acts from the right on $\PosDef(V)$ via pullback:
\begin{align*}
	\Gl(V) \times \PosDef(V) \to \PosDef(V),
	\quad
	(A,b) \mapsto A^\pull b = b(A \cdot, A \cdot).
\end{align*}
As an open set in the vector space $\SymBil(V) \subset V'\otimes V'$ of symmetric bilinear forms on $V$, the space $\PosDef(V)$ is a smooth manifold with tangent bundle given by $T_b \PosDef(V) \ceq \SymBil(V)$.
We equip $\PosDef(V)$ with the following Riemannian structure $g_\PosDef$:
\begin{align*}
	g_\PosDef\at_b(X,Y) \ceq \innerprod{X,Y}_b \quad \text{for all $X,\,Y \in T_b \PosDef(V) =\SymBil(V)$,}
\end{align*}
where $\innerprod{\cdot,\cdot}_b$ denotes the inner product on $V'\otimes V'$ that is induced by $b$. Note that the $\Gl(V)$-action above is isometric with respect to this Riemannian metric.

For a basis $e=(e_1 ,\dotsc,e_k)$ of $V$, we define the \emph{Gram mapping} which maps a bilinear form to its Gram matrix:
\begin{align*}
	\Gram_e \colon V' \otimes V' \to \Mat_{k\times k}(\R),
	\quad X \mapsto (X(e_i,e_j))_{1\leq i,j\leq k}.
\end{align*}
In terms of the Gram mapping, one has the following representation:
\begin{align}
	g_\PosDef(X,Y)\at_b=
	\innerprod{X,Y}_b = \tr (\Gram_e(b)^{-1} \Gram_e(X)^\transp \Gram_e(b)^{-1}\Gram_e(Y)),
	\label{eq:innerprodinGram}
\end{align}
where $b \in \PosDef(V)$ and $X$, $Y \in T_b \PosDef(V)$.

One can show that $(\PosDef(V),g_\PosDef)$ is a geodesically convex and complete Riemannian manifold (see \cite{MR2782125}). The geodesic starting at $b \in \PosDef(V)$ in direction $X \in T_b \PosDef(V)$ can be expressed as
\begin{align*}
	\gamma(t) = \Gram_e^{-1} \Bigparen{
	\mathbf L^\transp \exp\bigparen{t \, \mathbf L^{-\transp}\Gram_e(X) \mathbf L^{-1} } \, \mathbf L
	},
\end{align*}
where $\mathbf{L} \in \Mat_{k \times k}(V)$ can be any matrix with $\mathbf{L}^\transp\mathbf{L} = \Gram_e(b)$.
Moreover, the geodesic distance between $b$ and $c \in \PosDef(V)$ is given by
\begin{align}
	\dPosDef(b,c) 
%	= \nabs{\,H_c \at_b\,}_b
	= \nabs{\log \bigparen{\mathbf{L}^{-\transp} \Gram_e(c) \mathbf{L}^{-1}}}
	= \paren{\textstyle \sum_{i=1}^k \log(\lambda_i)^2}^\frac{1}{2}
	\label{cor:geodesicdistanceofPosDef}
\end{align}
Here, $\abs{\cdot}$ denotes the Frobenius norm of matrices and $\lambda_1, \dotsc, \lambda_k$ are the eigenvalues of $\Gram_e(g)$ with respect to $\Gram_e(b)$.

\subsection{Oriented Grassmannians}

For our metric on shape space we additionally require a notion of distance for maps between oriented Grassmannians. 

Let $V_1$, $V_2$, and $V$ be real vector spaces and denote by $\Inj(V_1;V_2) \subset \Hom(V_1;V_2)$ the space of all injective linear maps from $V_1$ to $V_2$.
As a short hand notation, we denote the Grassmannian $\orGr_k(V)$ of oriented $k$-dimensional linear subspaces of $V$ by $\Ray[k](V)$. Of course, $\Ray[k](V) = \emptyset$, if $k> \dim(V)$. 
Every element 
$W \in \Ray[k](V_1)$ can be represented by  an oriented linear independent system $(v_1,\dotsc,v_k)$ in $V_1$ and we write $W = [v_1,\dotsc,v_k]$.

This allows us to define the mapping 
$\Ray[k] \colon \Inj(V_1;V_2) \to C^0(\Ray[k](V_1);\Ray[k](V_2))$,
by setting
$
	\Ray[k](A)([v_1,\dotsc,v_k]) \ceq [A\,v_1,\dotsc,A\,v_k]
$	
for each $A \in \Inj(V_1;V_2)$
and each oriented linear independent system $(v_1,\dotsc,v_k)$ in $V_1$.
Observe that $\Ray[k](A)$ is also injective.

Let $\dRay[k]$ be a $\Ortho(V_2,g_2)$-invariant distance on $\Ray[k](V_2)$.
Since oriented Grassmannians are compact, one may define the distance $\dinfty$ on $C^0(\Ray[k](V_1);\Ray[k](V_2))$ by
\begin{align*}
	\dinfty(h_1,h_2) 
	\ceq \sup_{W \in \Ray[k](V_1)} \dRay[k] \paren{ h_1(W),h_2(W)}.
\end{align*}
Indeed, the compact-open topology on $C^0(\Ray[k](V_1);\Ray[k](V_2))$ is generated by $\dinfty$. Moreover, the image of $\Ray[k]$ is a closed subset of $C^0(\Ray[k](V_1);\Ray[k](V_2))$, thus $\Ray[k](V_1;V_2) \ceq \Ray[k](\Inj(V_1;V_2))$ equipped with the restriction of $\dinfty$ is a complete metric space.

Observe that $\Ray[k](A\,B) = \Ray[k](A)\circ \Ray[k](B)$ holds for $B \in \Inj(V_1;V_2)$ and $A \in \Inj(V_2;V)$, thus $\Ray[k]$ is a covariant functor. Moreover, $\dinfty$ is a $\Gl(V_1)$-invariant metric on $\Ray[k](V_1;V_2)$.

\begin{example}\label{ex:dRay}
We are mostly interested in the \emph{ray spaces} $\Ray(V_1;V_2)$. For an inner product $g_2$ on $V_2$, an $\Ortho(V_2,g_2)$-invariant distance $\dRay$ on $\Ray(V_2)$ is given by 
\begin{align*}
	\dRay([u],[v]) = \measuredangle(u,v) = \arccos \paren{\frac{g_2(u,v)}{\nabs{u}_{g_2}\,\nabs{v}_{g_2}}}.
\end{align*}
Then we have for $A$, $B \in \Inj(V_1;V_2)$:
\begin{align*}
	\dinfty( \Ray(A),\Ray(B)) 
	= \sup_{u \in V_1\setminus\{0\}} \measuredangle( A \,u , B\, u)
	= \sup_{u \in V_1\setminus\{0\}} \arccos \paren{
		\frac{g_2(A \,u , B\, u)}{\abs{A \,u }_{g_2} \abs{B\,u}_{g_2}}
		}.
\end{align*}
\end{example}

\subsection{Lipschitz Immersions}

Throughout this section, $\varSigma$ denotes a compact, $k$-dimensional smooth manifold with smooth boundary.
Let $g$ be a smooth Riemannian metric on $\varSigma$.
With a slight abuse of notation, we denote with $\vol_g$ not only the Riemannian density induced by $g$, but also the complete measure induced by it. Define the locally trivial fiber bundle $\pi \colon \PosDef_\varSigma \to \varSigma$ by
$\PosDef_\varSigma\at_x \ceq \PosDef(T_x\varSigma)$ for all $x\in \varSigma$, where $\PosDef(T_x\varSigma)$ denotes the manifold of positive-definite, symmetric bilinear forms on the tangent space $T_x\varSigma$ (see \autoref{posdef}). We equip each fiber with the distance function $\dPosDef$ and define
\begin{align*}
	\PosDefSec(\varSigma)
	\ceq 
	\Set{ 
	\text{$b \colon \varSigma \to \PosDef_\varSigma$ $\vol_g$-measurable.} 
	|
	\begin{array}{cc}	
	\text{$\pi \circ b = \id_\varSigma$ $\vol_g$-a.\,e.,}\\
	\text{$\esssup_{x\in\varSigma} \dPosDef(b\at_x,g\at_x) < \infty$} 
	\end{array}
	}\slash \sim
\end{align*}
with the equivalence relation $b_1 \sim b_2$ if $b_1 = b_2$ holds \hbox{$\vol_g$-almost} everywhere.

We introduce the distance $\dPosDefSec$ on $\PosDefSec(\varSigma)$ by
\begin{align*}
	\dPosDefSec( b_1, b_2) \ceq \esssup_{x\in \varSigma} \, \dPosDef \bigparen{ b_1\at_x,b_2\at_x}.
\end{align*}
Note that neither the space $\PosDefSec(\varSigma)$ nor the distance $\dPosDefSec$ depend on the choice of $g$ since we assume $\varSigma$ to be compact.

\begin{definition}\label{dfn:Imm}
Denote by $g_0$ the Euclidean metric on $\R^m$.
For $f \in W^{1,\infty}(\varSigma;\R^m)$, the tangent map $\dd f$ exists at almost every point $x \in \varSigma$ and we obtain an almost everywhere defined pullback $f^\pull g_0$.
Following \cite{MR3008339}, we define the \emph{space of Lipschitz immersions} as
\begin{align*}
	\Imm(\varSigma;\R^m) 
	&\ceq 
	\set{ f \in W^{1,\infty}(\varSigma;\R^m) | f^\pull g_0 \in \PosDefSec(\varSigma)}
\end{align*}
and equip it with the distance
\begin{align*}
	\dImm(f_1,f_2)
	\ceq \nnorm{f_1-f_2}_{L^\infty} 
	+ \dPosDefSec(f_1^\pull g_0,f_2^\pull g_0) + \esssup_{x \in \varSigma} \; \dinfty\bigparen{\Ray(\dd f_1\at_x) , \Ray(\dd f_2 \at_x)}.
\end{align*}
for $f_1$, $f_2 \in \Imm(\varSigma;\R^m)$.
\end{definition}
As $\Imm(\varSigma;\R^m)$ is an open subset of $W^{1,\infty}(\varSigma;\R^m)$ (see \autoref{cor:W1inftyembLip}), the following theorem comes as a surprise.
\begin{theorem}\label{theo:Immiscomplete}
The metric space $(\Imm(\varSigma;\R^m), \dImm)$ is complete.
\end{theorem}
\begin{proof}
Fix a smooth Riemannian metric $g$ on $\varSigma$
and a Cauchy sequence $(f_\alpha)_{\alpha\in \N}$ with respect to $\dImm$.
Then $(f_\alpha^\pull g_0)_{\alpha \in \N}$ is a Cauchy sequence in $(\PosDefSec(\varSigma),\dPosDefSec)$ and thus it is bounded, i.e., there is an $\ell >0$ with
$\ell_\alpha \ceq \dPosDefSec(g,f_\alpha^\pull g_0) < \ell$ for all $\alpha \in \N$.
Moreover, there is an $N \in \N$ such that one has 
$\dPosDefSec(f_{\alpha}^\pull g_0,f_{\beta}^\pull g_0) < \frac{5}{2}$
for all $\alpha$, $\beta \geq N$.
By \autoref{cor:LipembW1infty}, one obtains
\begin{align*}
	\nnorm{f_{\alpha}-f_{\beta}}_{W^{1,\infty}_g}
	\leq \exp(\ell_\alpha)\, \dImm(f_\alpha,f_\beta)
	\leq \exp(\ell)\, \dImm(f_\alpha,f_\beta).
\end{align*}
Thus, $(f_\alpha)_{\alpha\in \N}$ is also a Cauchy sequence in $W^{1,\infty}_g(\varSigma;\R^m)$. Hence there is an $f \in W^{1,\infty}_g(\varSigma;\R^m)$ such that $\nnorm{f_\alpha-f}_{W^{1,\infty}_g} \converges{\alpha \to \infty} 0$.
For sufficiently large $\alpha \in \N$, one has
$\nnorm{f_\alpha-f}_{W^{1,\infty}_g} < \tfrac{1}{3}\exp \paren{-\tfrac{3}{2} \ell}<\frac{1}{3}\exp(-\frac{3}{2} \ell_\alpha)$ such that we may apply \autoref{cor:W1inftyembLip}  in order to show that $f \in \Imm(\varSigma;\R^m)$ and to obtain
\begin{align*}
	\dImm(f_\alpha,f) 
	\leq 18 \, \ee^{\frac{3}{2}\ell_\alpha} \nnorm{f_\alpha-f}_{W^{1,\infty}_g}
	\leq 18 \, \ee^{\frac{3}{2}\ell} \nnorm{f_\alpha-f}_{W^{1,\infty}_g}
	\converges{\alpha\to \infty} 0.
\end{align*}
Hence $f_\alpha$ converges to $f$ in $(\Imm(\varSigma;\R^m),\dImm)$.
\qed
\end{proof}

We additionally require Lipschitz immersions that remain Lipschitz immersions when restricted to the boundary. If the embedding $j \colon \partial \varSigma \hookrightarrow \varSigma$ of the boundary and the Riemannian metric are suffciently smooth, then the \emph{trace operator} 
\begin{align*}
	j^\pull \colon W^{1,\infty}_g(\varSigma;\R^m) \to W^{1,\infty}_{j^\pull g}(\partial \varSigma; \R^m),
	\quad f \mapsto f|_{\partial \varSigma}
\end{align*}
is well-defined and continuous. Hence, the point of the following definition is solely to ensure that restrictions are immersions.
\begin{definition}\label{dfn:strongImm}
Let $\varSigma$ be a compact, smooth manifold with boundary.
We define the space of \emph{strong Lipschitz immersions} by
\begin{align*}
	\strongImm(\varSigma;\AmbSpace)
	\ceq 
	\set{ f \in \Imm(\varSigma;\AmbSpace) | \;f|_{\partial \varSigma}  \in \Imm(\partial \varSigma;\AmbSpace) }
\end{align*}
and equip it with the graph metric
\begin{align*}
	\dstrongImm(f_1,f_2) \ceq \dImm(f_1,f_2) + \dImm(f_1|_{\partial \varSigma},f_2|_{\partial \varSigma}),
	\quad
	\text{for $f_1$, $f_2 \in \strongImm(\varSigma;\AmbSpace)$.}
\end{align*}
Note that also $(\strongImm(\varSigma;\AmbSpace), \dstrongImm)$ is a complete metric space.
\end{definition}

To be honest, we currently do not know if $\strongImm(\varSigma;\AmbSpace)$ equals $\Imm(\varSigma;\AmbSpace)$ or not. For our purposes, it is sufficient to have the following trace ``theorem'' for the (maybe smaller) space $\strongImm(\varSigma;\AmbSpace)$:
\begin{lemma}\label{lem:strongImmhastrace}
Let $\varSigma$ be a compact, smooth manifold with boundary.
Then the trace mapping
\begin{align*}
	\res \colon (\strongImm(\varSigma;\AmbSpace),\dstrongImm)  \to (\strongImm(\partial\varSigma;\AmbSpace),\dstrongImm) ,
	\qquad
	f \mapsto f|_{\partial \varSigma}
\end{align*}
is Lipschitz continuous with Lipschitz constant $1$.
\end{lemma}
\invisible{
\begin{proof}
\qed
\end{proof}
}

For $\Curve \in \Imm(\partial \varSigma;\R^m)$, we may also define the space of immersions under boundary conditions:
\begin{align*}
	\Immg(\varSigma;\R^m) 
	\ceq 
	\set{f \in \strongImm(\varSigma;\R^m) | \res(f) =\Curve}.
\end{align*}

\subsection{Diffeomorphism Group}\label{diffeos}

Let $(\varSigma,g)$ be a compact, smooth Riemannian manifold.
We define the group of \emph{Lipschitz diffeomorphisms} by
\begin{align*}
	\Diff(\varSigma) \ceq
	\set{ \varphi \in W^{1,\infty}_g(\varSigma;\varSigma) | \text{$\varphi$ is a bi-Lipschitz homeomorphism}},
\end{align*}
where the group structure is given by
\begin{gather*}
	\mu \colon \Diff(\varSigma) \times \Diff(\varSigma) \to \Diff(\varSigma),
	\qquad (\varphi,\psi) \mapsto \varphi \circ \psi,
	\\
	\iota \colon \Diff(\varSigma) \to \Diff(\varSigma),
	\qquad \varphi \mapsto \varphi^{-1}.
\end{gather*}
Note that the space $\Diff(\varSigma)$ does not depend on the choice of the Riemannian metric~$g$.
The existence of global Lipschitz constants for $\varphi$ and $\varphi^{-1}$ implies $\varphi^\pull g$,  $(\varphi^{-1})^\pull g \in \PosDefSec(\varSigma)$ for every Riemannian metric $g$ on $\varSigma$.
The group $\Diff(\varSigma)$ acts from the right on $\Imm(\varSigma;\R^m)$ and $\strongImm(\varSigma;\R^m)$ via
$
	L_\varphi(f) \mapsto  f \circ \varphi,
$
for all $\varphi \in \Diff(\varSigma)$.
Since $\varphi$ is a bi-Lipschitz homeomorphism, the composition $f \circ \varphi$ is differentiable on the following set of full measure
\begin{align*}
	\varphi^{-1}(\set{z\in\varSigma | \text{$f$ is differentiable at $z$}})
	\cap	 \set{y\in\varSigma | \text{$\varphi$ is differentiable at $y$}}.
\end{align*}
Hence, we have the \emph{chain rule} 
\begin{align}
	\dd (f \circ \varphi)\at_x = \dd f\at_{\varphi(x)} \, T_x \varphi
	\quad
	\text{for almost all $x \in \varSigma$,}\label{eq:chainrule}
\end{align}
where $T_x\varphi \colon T_x \varSigma \to \T_x \varSigma$ denotes the tangent map.
We point out that this implies that $L_\varphi$ is an isometry with respect to $\dImm$ and $\dstrongImm$ for each $\varphi \in \Diff(\varSigma)$---a fact that we utilize to analyze the quotient metrics of $\dImm$ and $\dstrongImm$.
Moreover, note that the restriction $\varphi|_{\partial \varSigma}$ of a Lipschitz diffeomorpism $\varphi \in \Diff(\varSigma)$ is a Lipschitz diffeomorpism of $\partial \varSigma$. This constitutes a continuous mapping $\res \colon \Diff(\varSigma) \to \Diff(\partial \varSigma)$.

\subsection{Quotient Space}\label{quotspace}

\begin{definition}\label{dfn:strongShapeSpace}
\label{dfn:ShapeSpace}
We define the \emph{shape space}
\begin{align*}
	\strongShape(\varSigma;\R^m) &\ceq \strongImm(\varSigma;\R^m) \slash \Diff(\varSigma)
\end{align*}
and denote by 
$\strongShapeQuot \colon \strongImm(\varSigma;\R^m) \to \strongShape(\varSigma;\R^m)
$ the canonical map.
Since $\Diff(\varSigma)$ acts through isometries, the \emph{quotient semi-metric} $\dstrongShape$ can be written as
\begin{align*}
	\dstrongShape (\strongShapeQuot(f_1), \strongShapeQuot(f_2)) 
	&= \inf_{\varphi \in \Diff(\varSigma)} \dstrongImm(f_1, f_2 \circ \varphi)
	& &\text{for $f_1$, $f_2 \in \strongImm(\varSigma;\R^m)$.}
\end{align*}
\end{definition}

\begin{theorem}\label{quotisHsd}
Let $\varSigma$ be a compact smooth manifold with boundary.  \newline
Then $(\strongShape(\varSigma;\R^m),\dstrongShape)$ is a complete metric space.
\end{theorem}
\begin{proof}
First, we show that $\dstrongShape$ is definite. Let $f$, $h \in \strongImm(\varSigma;\R^m)$ and let $\varphi_\alpha \in \Diff(\varSigma)$ be a sequence with $\dstrongImm(f,h\circ\varphi_\alpha) \to 0$ for $\alpha \to \infty$. We have to find a $\varphi \in \Diff(\varSigma)$ with $f = h\circ\varphi$.	

We start by choosing a smooth Riemannian metric $g$ on $\varSigma$ such that the boundary (if it exists) is totally geodesic. This way, for every point $x\in\varSigma$, every neighborhood $U$ of $x$ contains a geodesically convex neighborhood of $x$. Such a Riemannian metric can be constructed, for example, by choosing a cylinder metric on a smooth collar of $\varSigma$ and extending it smoothly.\footnote{A \emph{smooth collar} of $\varSigma$ is a smooth embedding $\varPhi \colon \intervalco{0,1} \times \partial \varSigma \to \varSigma$ such that $\varPhi(0,x) =x$ holds for all $x \in \partial \varSigma$. Every paracompact smooth manifold with boundary has a smooth collar.}

Observe that $h_\alpha\ceq h\circ\varphi_\alpha$ converges uniformly to $f$. Moreover, being convergent, $h_\alpha^\pull g_0$ is a bounded sequence in $\PosDefSec(\varSigma)$. Hence there is some $\varLambda_0\geq 0$ with
\begin{align*}
	\nnorm{\dd h_\alpha^{\dagger_g}}_{L^\infty_g},\;\norm{\dd h_\alpha}_{L^\infty_g}
	\leq \varLambda_0.
\end{align*}
Here, $\dagger_g$ denotes the Moore-Penrose pseudoinverse with respect to the Riemannian metrics $g$ on $\varSigma$ and $g_0$ on $\R^m$.
The chain rule \eqref{eq:chainrule} yields
\begin{align*}
	\dd h_\alpha \at_x &= \dd h \at_{\varphi_\alpha(x)} \cdot T_x \varphi_\alpha  \quad \text{and}
	\\
	\dd h \at_x &= \dd h_\alpha \at_{\varphi_\alpha^{-1}(x)} \cdot T_x (\varphi_\alpha^{-1}),
\end{align*}
hence one obtains
$T_x \varphi_\alpha = (\dd h \at_{\varphi_\alpha(x)})^{\dagger_g} \cdot \dd h_\alpha \at_x$ for almost all $x \in \varSigma$. Thus, there is a $\varLambda \geq 0$ with
	\begin{align*}
		\nnorm{T \varphi_\alpha}_{L^\infty_g} 
		&\leq \nnorm{(\dd h)^{\dagger_g}}_{L^\infty_g} \cdot \norm{\dd h_\alpha}_{L^\infty_g} \leq \varLambda \quad \text{and}
		\\
		\nnorm{(T\varphi_\alpha)^{-1}}_{L^\infty_g} 
		&\leq \nnorm{(\dd h_\alpha)^{\dagger_g}}_{L^\infty_g} \cdot \nnorm{\dd h}_{L^\infty_g} \leq \varLambda,
	\end{align*}
showing that the families $(\varphi_\alpha)_{\alpha \in \N}$ and $(\varphi_\alpha^{-1})_{\alpha \in \N}$ are equicontinuous. Because $\varSigma$ is a compact metric space, the families $(\varphi_\alpha)_{\alpha \in \N}$ and $(\varphi_\alpha^{-1})_{\alpha \in \N}$ are also pointwise relatively compact.
Thus, the Arzel\`a-Ascoli theorem (see, e.g., \cite[Theorem 47.1]{Munkres2000}) implies the existence of a subsequence (which we also denote by $(\varphi_\alpha)_{\alpha\in\N}$) such that both
$\varphi_\alpha \to \varphi$ and $\varphi_\alpha^{-1} \to \varphi^{-1}$ converge in the compact-open topology on $C(\varSigma;\varSigma)$.

Up to now, we know that $\varphi \colon \varSigma \to \varSigma$ is a homeomorphism and that $f = h \circ \varphi$.
We are left to show that both $\varphi$ are $\varphi^{-1}$ are Lipschitz continuous. \invisible{\autoref{lem:limitisLipschitz}}
To this end, let $V_1,\dotsc, V_N$ with some $N \in \N$ be a covering of $\varSigma$ by open, relatively compact, and geodesically convex sets.
Choose a covering $U_1,\dotsc,U_M$ with some $M \in \N$ of $\varSigma$ by open, relatively compact and geodesically convex sets such that each $\varphi(U_i)$ is contained in some $V_j$.
Then one has for all $x$, $y \in U_i$:
\begin{align*}
	d_g( \varphi(x),\varphi(y) )
	&\leq d_g( \varphi(x),\varphi\alpha(x)) + d_g( \varphi_\alpha(x),\varphi_\alpha(y) ) + d_g( \varphi_\alpha(y),\varphi(y) )
	\\
	&\leq d_g( \varphi(x),\varphi_\alpha(x)) + \norm{T\varphi_\alpha}_{L^\infty_g} \, d_g(x,y) + d_g( \varphi_\alpha(y),\varphi(y) ).
\end{align*}
Applying $\limsup_{\alpha\to \infty}$ yields $\norm{T\varphi}_{L^\infty_g} \leq \Lambda$,
hence $\varphi$ is Lipschitz continuous. The same argument shows that $\varphi^{-1}$ is Lipschitz continuous, too.

Next, we show that $\dstrongShape$ is complete.
Let $(x_\alpha)_{\alpha \in \N}$ be a Cauchy sequence in $\strongShape(\varSigma;\R^m)$. It suffices to show that $(f_\alpha)_{\alpha \in \N}$ has a convergent subsequence. 
By passing over to an appropriate subsequence, we may suppose that $\dstrongShape(x_\alpha,x_\beta) \leq 2^{-(\min(\alpha,\beta)+1)}$ holds for all $\alpha$, $\beta \in \N$. 
Choose $f_1 \in \strongShapeQuot^{-1}(x_1)$ arbitrarily and choose $f_\alpha \in \strongShapeQuot^{-1}(x_\alpha)$ for $\alpha>1$ recursively with $\dstrongImm(f_{\alpha-1},f_\alpha) \leq \dstrongShape(x_{\alpha},x_{\alpha+1}) + 2^{-(\alpha+1)} \leq 2^{-\alpha}$. 
Thus, one has for each pair $\alpha$, $\beta \in \N$ with $\alpha \leq \beta$
\begin{align*}
	\dstrongImm(f_\alpha,f_\beta) \leq \sum_{i=\alpha}^{\beta-1} \dstrongImm(f_{i},f_{i+1}) 
	\leq \sum_{i=\alpha}^{\beta-1} 2^{-i} 
	= 2^{-\alpha} \sum_{i=0}^{\beta-\alpha-1}  2^{-i} \leq 2^{-\alpha+1}.
\end{align*}
Hence, $(f_\alpha)_{\alpha \in \N}$ is a Cauchy sequence and thus converges to a point $f \in \strongImm(\varSigma;\R^m)$. 
Now, $x_\alpha = \strongShapeQuot(f_\alpha)$ converges in to $x \ceq \strongShapeQuot(f)$
since $\strongShapeQuot$ is Lipschitz continuous.
\qed
\end{proof}

\begin{remark}
Notice that for $f_1$, $f_2\in \strongImm(\varSigma;\R^m)$, the distance $\dstrongShape(\strongShapeQuot(f_1),\strongShapeQuot(f_2))$ provides an upper bound for the Fr\'{e}chet distance 
\begin{align*}
	d_{\on{F}}(f_1,f_2) \ceq \inf_{\varphi \in \Diff(\varSigma)} \norm{f_1,f_2\circ \varphi}_{L^\infty}
\end{align*}
as well as for the Haussdorf distance between the sets $f_1(\varSigma)$ and $f_2(\varSigma)$. 

The \emph{Gauss mapping} induced by $f \in \Imm(\varSigma;\R^m)$ is the measurable mapping $\tau(f) \colon  \varSigma \to \Gr_k(\R^m)$ defined almost everywhere by $\tau(f)(x) = \dd_xf(T_x\varSigma)$.
Let $d_{\Gr_k}$ be an $\Ortho(\R^m)$-invariant metric on $\Gr_k(\R^m)$. Then there is a constant $C >0$ such that $d_{\Gr_k}( A(V),B(V)) \leq C \, \dRay(\Ray(A),\Ray(B))$ holds for all $A$, $B \in \Inj(V;\R^m)$.
For $f_1$, $f_2\in \strongImm(\varSigma;\R^m)$, define $h_i \ceq (f_i,\tau(f_i)) \colon \varSigma \to \R^m \times \Gr_k(\R^m)$ for $i \in \{1,2\}$.
Then $\dstrongShape(\strongShapeQuot(f_1),\strongShapeQuot(f_2))$ bounds also the Fr\'{e}chet distance between $h_1$ and $h_2$.

If $f_1$ and $f_2$ are embeddings of class $C^1$, then $M_1 \ceq f_1(\varSigma)$ and $M_2 \ceq f_2(\varSigma)$ are submanifolds of $\R^m$ and $\dstrongShape(\strongShapeQuot(f_1),\strongShapeQuot(f_2))$ yields an upper bound for the Hausdorff distance $d_{\on{Hsd}}(h_1(\varSigma),h_2(\varSigma))$
between their tangent bundles.
\end{remark}

\subsection{Volume Functionals}\label{volumefunctional}

Volume functionals on the space of Lipschitz immersions and on shape space are essential for the treatment of least volume problems. In this section, we establish their local Lipschitz continuity. We work with densities instead of the perhaps better known notion of volume forms in order to be able to treat nonorientable manifolds. 

To establish notation, we start with the linear case. Let $V$ be a $k$-dimensional real vector space.
A \emph{density} on $V$ is a function $\varrho \colon \prod_{i=1}^k V \to \R$ with the properties:
\begin{enumerate}
	\item For all $v_1,\dotsc,v_k \in V$ and all $\lambda_1,\dotsc,\lambda_k \in \R$ the following holds:
	\begin{align*}
		\varrho( \lambda_1 \,v_1 \dotsc, \lambda_n \, v_k) =\abs{\lambda_1\dotsm\lambda_k} \,\varrho (v_1 \dotsc, v_k).
	\end{align*}
	\item For all $v_1,\dotsc,v_k \in V$, $\lambda \in \R$ and $i \neq j$ the following holds:
	\begin{align*}
		\varrho( v_1 \dotsc, v_{i-1}, v_{i}+ \lambda v_j, v_{i+1},\dotsc,v_k)
		= \varrho( v_1 \dotsc, v_{i-1}, v_{i}, v_{i+1},\dotsc,v_k).
	\end{align*}
\end{enumerate} 
For a thorough introduction to densities, we refer to \cite[pp. 375--382]{MR1930091}. Here, we only collect some facts that we need for our purposes. A density $\varrho$ on $V$ is called \emph{positive}, if $\varrho(e_1,\dotsc,e_k) >0$ holds for all bases $e=(e_1,\dotsc,e_k)$ of $V$.
Denote the space of positive densities on $V$ by $\Vol(V)$.
For two positive densities $\varrho$ and $\sigma$ denote the unique positive number $t \in \R_{>0}$ with $\sigma= t \, \varrho$ by $\tfrac{\sigma}{\varrho}$.
We define the distance $\dVol$ by
\begin{align*}
	\dVol( \varrho, \sigma) \ceq \nabs{\log \bigparen{\tfrac{\sigma}{\varrho}}} \quad \text{for all $\varrho$, $\sigma \in \Vol(V)$.}
\end{align*}
Notice that every inner product $g \in \PosDef(V)$ on $V$ induces a unique positive density, the \emph{volume density} $\vol_g$ on $V$, that satisfies
$\vol_g( e_1,\dotsc,e_k) = 1$ for any $g$-orthonormal basis $e=(e_1, \dotsc, e_k)$ of $V$. 
Thus, one has
\begin{align*}
	\frac{\vol_b}{\vol_g} = \frac{ \det (\Gram_e(b) )}{\det (\Gram_e(g))}
	\qand
	\dVol( \vol_b, \vol_g) =\nabs{ \log \nparen{\det (\Gram_e(b) )}- \log \nparen{ \det (\Gram_e(g))}},
\end{align*}
for any two elements $g$, $b \in \PosDef(V)$ and \emph{any} basis $e$ of $V$. (Again, $\Gram_e(g)$ and $\Gram_e(b)$ denote the Gram matrices of $g$ and $b$ with respect to the basis $e$.)

\begin{lemma}\label{lem:volisLipschitz}
The mapping 
\begin{align*}
	\vol \colon \PosDef(V) \to \Vol(V), \quad g \mapsto \vol_g
\end{align*}
is Lipschitz-continuous with Lipschitz constant $\dim(V)^\frac{1}{2}$.
\end{lemma}
\begin{proof}
Let $g$, $b \in \PosDef(V)$ and choose an orthonormal basis $e=(e_1,\dotsc,e_k)$ of $g$ that diagonalizes $\Gram_e(b)$. Denoting the eigenvalues $\Gram_e(b)$ by $\lambda_1,\dotsc,\lambda_k$, we have
$\vol_b = \lambda_1 \dotsm \lambda_k \, \vol_g$. By H\"older's inequality, this leads to
\begin{align}
	\dVol(\vol_g,\vol_b) = \nabs{\log (\lambda_1 \dotsm \lambda_k)}
	\leq \sum_{i=1}^k \nabs{\log(\lambda_i)}
	\leq \sqrt{k} \, \dPosDef(g,b), \label{eq:voldist}
\end{align}
which shows the statement. 
\qed
\end{proof}

We now generalize this setup to manifolds. Throughout, let $\varSigma$ be a compact, \hbox{$k$-dimensional} smooth manifold with boundary.

\begin{lemma}\label{volumeislipschitz}
The \emph{volume functional}
\begin{align*}
	\cJ \colon (\PosDefSec(\varSigma),\dPosDefSec) \to \R, \qquad g\mapsto \int_\varSigma \vol_{g}
\end{align*}
has its modulus of continuity $\omega_\cJ(g,t)$ bounded by $\cJ(g) \, \ee^{\sqrt{k}\,t}\,\sqrt{k}\,t$, i.e.,
\begin{align*}
	\abs{\cJ(b) - \cJ(g)} \leq \cJ(g) \, \ee^{\sqrt{k}\, \dPosDefSec(g,b)} \, \sqrt{k} \, \dPosDefSec(g,b)
	\quad \text{for all $g$, $b \in \PosDefSec(\varSigma)$.}
\end{align*}
\end{lemma}
\begin{proof}
We abbreviate $\varLambda \ceq \frac{\vol_b}{\vol_g}$.
From \autoref{lem:volisLipschitz}, we know that
\begin{align*}
	\nabs{ \log (\varLambda(x))} 
	= \dVol( \vol_{g}\at_x,\vol_{b}\at_x)
	\leq \sqrt{k}\, \dPosDef(g \at_x, b\at_x) 
	\leq \sqrt{k} \,\dPosDefSec(g,b).
\end{align*}
Together with the estimate $\abs{t-1} \leq \nabs{\log(t)} \, \ee^{\nabs{\log(t)}}$ for all $t>0$, we obtain
\begin{align*}
	\abs{\cJ(b)-\cJ(g)}
	&= \bigabs{\int_\varSigma (\vol_{b}-\vol_{g})}
	\leq \int_\varSigma \abs{\Lambda-1} \, \vol_{g}
	\leq \cJ(g) \, \ee^{\sqrt{k}\, \dPosDefSec(g,b)} \, \sqrt{k} \, \dPosDefSec(g,b).
\end{align*}
\qed
\end{proof}

\begin{corollary}\label{volumeislipschitz1}
Let $\omega_\cF$ be the modulus of continuity of the \emph{volume functional}
\begin{align*}
	\cF \colon (\Imm(\varSigma;\R^m),\dImm) \to \R, \qquad f \mapsto \int_\varSigma \vol_{f^\pull g_0}.
\end{align*}
Then one has
$\omega_\cF(f,t) \leq \cF(f) \, \ee^{\sqrt{k}\,t}\,\sqrt{k}\,t$.
\end{corollary}
\invisible{
\begin{proof}
\autoref{volumeislipschitz}
\qed
\end{proof}
}

\begin{corollary}\label{volumeislipschitz2}
Let $\strongShapeQuot_\#\cF \colon (\strongShape(\varSigma;\R^m),\dstrongShape) \to \R$ be the variational pushforward of $\cF$ along $\strongShapeQuot$. Its modulus of continuity $\omega_{(\strongShapeQuot_\#\cF)}$ satisfies
\begin{align*}
	\omega_{\strongShapeQuot_\#\cF}(x,t) \leq (\strongShapeQuot_\#\cF)(x) \, \ee^{\sqrt{k}\,t}\,\sqrt{k}\,t.
\end{align*}
\end{corollary}
\begin{proof}
\invisible{\autoref{volumeislipschitz1}}
For $x$, $y \in \strongShape(\varSigma;\R^m)$ and $f \in \strongShapeQuot^{-1}(x)$, $h \in \strongShapeQuot^{-1}(y)$ observe
\begin{align*}
	\abs{(\strongShapeQuot_\#\cF)(y) - (\strongShapeQuot_\#\cF)(x)}
%	&= \inf_{h \in \strongShapeQuot^{-1}(y)} \abs{\cF(h) - \cF(f)}
%	\\
	&= \inf_{\varphi \in \Diff(\varSigma)} \, \abs{\cF(h \circ \varphi) - \cF(f)}
	\\
	&\leq \inf_{\varphi \in \Diff(\varSigma)}  \, \cF(f) \, \ee^{\sqrt{k}\, \dstrongImm(f,h \circ \varphi)} \, \sqrt{k} \, \dstrongImm(f,h \circ \varphi)
	\\
	&= (\strongShapeQuot_\#\cF)(x) \, \ee^{\sqrt{k}\, \dstrongShape(x,y)} \, \sqrt{k} \, \dstrongShape(x,y),
\end{align*}
where the last equality holds because of the invariance of the volume functional and the monotonicity of the function $t \mapsto \ee^{\sqrt{k}\, t} \, \sqrt{k} \, t$.
\qed
\end{proof}

\clearpage

% !TEX root = main.tex
\section{Minimal Surfaces}\label{sec:MinimalSurfaces}

As an application of the theory developed in \autoref{sec:paramopt}, in particular of \autoref{topstabandconv}, we discuss a variant of the \emph{Douglas-Courant problem} or \emph{least area/volume problem}: Among the immersed $k$-dimensional surfaces in $\R^m$ with prescribed topology and Dirichlet boundary conditions find those of minimal $k$-volume. For $k=2$, minimizers of class $C^2$ are examples of minimal surfaces.

We discretize this problem by searching for volume-minimizers among immersed $k$-dimen\-sion\-al simplicial meshes of fixed combinatorics bounded by a given, closed $(k-1)$-dimensional simplicial mesh. To some extent, this approach can be understood as a nonconforming Ritz-Galerkin method with first order Lagrange elements (piecewise linear finite elements).

The point we would like to make is this: Given a sufficiently well-posed Douglas-Courant problem, i.e., the boundary conditions are such that volume minimizers within the prescribed topological class exist \emph{and} have a certain uniform regularity, the set of solutions can be approximated by solutions of a discrete Douglas-Courant problem.

We start our exposition by giving a precise definition of minimal surfaces
and by stating both the Douglas-Courant problem and the least area/volume problem. 
Afterwards, we discretize the least area problem and identify the relevant entities occurring in \autoref{topstabandconv}, namely smooth and discrete configuration spaces, functionals and test mappings, as well as sampling and reconstruction operators. Our convergence result then follows from an analysis of consistency and proximity errors.

\subsection{Problem Formulation}\label{sec:MinimalSurfacesTheory}

\begin{definition}\label{dfn:minimalsurface}
Let $\varSigma$ be a $2$-dimensional manifold with boundary and let $(M_0,g_0)$ be a smooth Riemannian manifold of dimension $m \geq 3$.
A mapping $f \in C^0(\varSigma;M_0) \cap C^2(\varSigma^\circ;M_0)$ is called a \emph{minimal surface} if there is a Riemannian metric $g$ of class $C^1$ in the interior $\varSigma^\circ \ceq \varSigma \setminus \partial \varSigma$ and a function $\varrho \in C^1(\varSigma^\circ;\intervalco{0,\infty})$ with
\begin{align}
	\Delta^{g,g_0} f \ceq \tr_g \Hess^{g,g_0}(f)	 =0 
	\quad \text{and} \quad
	f^\pull g_0 = \varrho \, g \quad \text{in $\varSigma^\circ$}.\label{eq:minimalsurfacecond}
\end{align}
\end{definition}

The \emph{Douglas-Courant problem}, also called the \emph{Plateau-Douglas problem}, can be formulated as follows (see \cite{MR0000102} or \cite{MR0036317}):
\begin{problem}[Douglas-Courant]\label{prob:DouglasCourant}
Let $\varSigma$ be a $2$-dimensional smooth manifold with  boundary and let $\Curve \in C^0( \partial \varSigma ; M_0)$ be an embedding.  Find all minimal surfaces $f$ with $f|_{\partial \varSigma} = \Curve \circ \varphi$ for some homeomorphism $\varphi \colon {\partial \varSigma }\to \partial \varSigma$.
\end{problem}
In the case that $\varSigma=D$ is the closed unit disk and $M_0 =\R^3$, this is traditionally referred to as the \emph{Plateau problem}.

The notion of minimal surfaces has its origin in the least area problem, the 2-dimensional instance of the least volume problem. We give a formulation of this problem in terms of Lipschitz immersions:

\begin{problem}[Least volume problem]\label{Plateauprob}
Let $\varSigma$ be a compact, $k$-dimensional smooth manifold with boundary.
% and fix a smooth Riemannian metric $g$ on $\varSigma$.
Let $(M_0,g_0)$ be a smooth, $m$-dimensional Riemannian manifold with $m>k$ and let $\Curve \in \Imm(\partial \varSigma; M_0)$ be a Lipschitz immersion.
Given $\varSigma$ and $\Curve$, minimize the volume functional
\begin{align*}
	\cF(f) =
	\int_\varSigma \vol_{f^\pull g_0}
\end{align*}
on the space $\cC \ceq \Immg(\varSigma;M_0)$ of Lipschitz immersions that restrict to $\Curve$  on the boundary (see \autoref{sec:ShapeSpace}).
\end{problem}

\begin{remark}
Note that by using Lipschitz immersions as configuration space, we exclude ``hairy'' mappings, but we also exclude continuously differentiable mappings with isolated branch points. 
\end{remark}

We summarize the close relation between the Douglas-Courant problem and the least area problem for Lipschitz immersions in the following statement:
\begin{lemma}\label{lem:conntominimalsurfaces}
Let $\varSigma$ be a compact, $2$-dimensional smooth manifold with boundary and let $(M_0,g_0)$ be a smooth Riemannian manifold without boundary.
Suppose that $\Curve \in \Imm(\partial \varSigma;M_0)$ is a topological embedding and let $f \in \cC \cap C^2(\varSigma^\circ;M_0)$ be a Lipschitz immersion that is of class $C^2$ in the interior of $\varSigma$. 

Then $f$ is a minimal surface if and only if it is a critical point of $\cF|_\cC$.
\end{lemma}
\invisible{\begin{proof}\qed\end{proof}}

\subsection{Smooth Setting}

In the following we use the abbreviations $\widetilde \cC \ceq \strongImm(\varSigma;\R^m)$, $\cC \ceq \Immg(\varSigma;\R^m)$, and $\cG \ceq \Diff(\varSigma)$.
Define $\cX \ceq \strongShape(\varSigma;\R^m) = \widetilde \cC\slash\cG$ and denote by $\varPsi \ceq \strongShapeQuot \circ \iota_\Curve \colon \cC \to \cX$ the composition of the inclusion $\iota_\Curve \colon \cC \hookrightarrow \widetilde \cC$ and the quotient map $\strongShapeQuot \colon \widetilde \cC \to \cX$. By \autoref{quotisHsd}, $\cX$ equipped with the quotient metric $d_\cX \ceq \dstrongShape$ induced by $\dstrongImm$ is a metric space. 

Fix an (arbitrary) Riemannian metric $g$ on $\varSigma$. In the upcoming convergence analysis, we assume that $\Curve$ is an embedding of class $\Imm(\partial \varSigma ; \R^m) \cap W^{2,\infty}_g(\partial \varSigma; \R^m)$. In particular, $\Curve$ is a bi-Lipschitzian homeomorphism onto its image.
As \emph{a~priori information}, we assume that there is a $s\geq 0$ with $\varPsi(\cA^s\cap \MinF{}{}) \supset \varPsi(\MinF{}{})$,
where
\begin{align*}
	\cA^s \ceq \set{f \in \cC\cap W^{2,\infty}_g(\varSigma;\R^m) | \text{$\dPosDefSec(g,f^\pull g_0)\leq s$, $\norm{f}_{W^{2,\infty}_g} \leq s$}}
	\quad \text{for all $s\geq 0$}.
\end{align*}
That means, every minimizer $\varPsi(f) \in \varPsi(\MinF{}{})=\argmin(\varPsi_\# \cF)$ allows for a ``nice'' parameterization $f \colon \varSigma \to \R^m$ with injective differentials, controlled distortion, and controlled $W^{2,\infty}_g$-norm.\footnote{Most of the results of this section remain true if one uses---instead of $W^{2,\infty}_g(\varSigma;\R^m)$---any other Banach space $\cB \supset C^\infty(\varSigma;\R^m)$  that embeds compactly into $W^{1,\infty}_g(\varSigma;\R^m)$. Other natural choices are $C^{\ell,\alpha}_g(\varSigma;\R^m)$ for $\ell \geq 1$, $\alpha \in \intervaloo{0,1}$ or $W^{\ell,p}_g(\varSigma;\R^m)$ for $\ell \geq 2$, $p > k$. The same applies to the Banach space that describes the regularity of the boundary condition $\Curve$. Note however, that proximity and consistency error rates may be impaired.}
The assumption $\varPsi(\cA^s\cap \MinF{}{}) \supset \varPsi(\MinF{}{})$ is satisfied in certain cases: We refer to the detailed regularity theory  in Section 2.3 of \cite{MR2760441} and in \cite{MR554379}.  
We point out that we do \emph{not} state, that our a~priori assumptions are always satisfied---not even in the case $k=2$, $m=3$---but at least for a variety of pairs $(\varSigma,\Curve)$. %For high dimension $k$ or codimension $m-k$, there are actually known counterexamples.

\subsection{Discrete Setting}\label{sec:MinimalSurfacesDiscrete}

In order to introduce discrete minimal surfaces we require the notion of smooth triangulations. Let $\Delta_k$ be the
$k$-dimensional \emph{standard simplex} of $\R^{k+1}$ with vertices 
$e_0,\dotsc,e_k$, i.e., the standard basis of $\R^{k+1}$. % and define the set of \emph{$d$-faces}
Let $\varSigma$ be a smooth, $k$-dimensional manifold with boundary. For a smooth embedding $\sigma \colon \Delta_k \to \varSigma$, define the \emph{vertex set}
\begin{align*}
	V(\sigma) \ceq \set{\sigma(e_0), \dotsc, \sigma(e_k)}. %= \sigma(\Faces_0).
\end{align*}
A \emph{smooth triangulation} of $\varSigma$ is a family
\begin{align*}
	\cT \subset \set{ \sigma \colon \Delta_k \to \varSigma | \text{$\sigma$ is a smooth embedding}}
\end{align*}
with the following properties:
\begin{enumerate}
	\item $\varSigma = \bigcup_{\sigma \in \cT} \sigma(\Delta_k)$.
%	\item For $\sigma$, $\tau \in \cT$ with $\sigma \neq \tau$ one has 
%	$\sigma(\on{int{\Delta_k}}) \cap \tau(\on{int{\Delta_k}}) = \emptyset$
	\item For each pair $\sigma$, $\tau \in \cT$ with $\sigma(\Delta_k) \cap \tau(\Delta_k) \neq \emptyset$, both $\sigma^{-1}(\sigma(\Delta_k) \cap \tau(\Delta_k))$ and $\tau^{-1}(\sigma(\Delta_k) \cap \tau(\Delta_k))$ are $d$-faces of $\Delta_k$ for some $0\leq d \leq k-1$ and the mapping
	\begin{align*}
		\tau^{-1} \circ \sigma \colon \sigma^{-1}(\sigma(\Delta_k) \cap \tau(\Delta_k)) \to 
		\tau^{-1}(\sigma(\Delta_k) \cap \tau(\Delta_k))
	\end{align*}
	is affine.
	\item For each $\sigma \in \cT$ with 
	$\sigma(\Delta_k) \cap \partial \varSigma \neq \emptyset$, the set $\sigma^{-1}(\sigma(\Delta_k) \cap \varSigma)$ is a $d$-face of $\Delta_k$ for some $0\leq d \leq k-1$.
\end{enumerate}
We distinguish between \emph{boundary vertices} and \emph{interior vertices}:
\begin{align*}
	V(\cT) \ceq \bigcup_{\sigma \in \cT} V(\sigma),
	\quad
	V_{\on{b}}(\cT) \ceq V(\cT) \cap \partial \varSigma
	\qand
	V_{\on{i}}(\cT) \ceq V(\cT)\setminus V_{\on{b}}(\cT).
\end{align*}
A smooth triangulation is called \emph{finite} if its cardinality is finite. Note that every smooth manifold with boundary admits a smooth triangulation (see \cite{MR0002545}). For compact manifolds with boundary, there is always a finite smooth triangulation. A smooth triangulation $\cT$ of a $k$-dimensional smooth manifold $\varSigma$ with boundary induces a smooth triangulation $\cT|_{\partial \varSigma}$ of the boundary $\partial \varSigma$ in the following way:
\begin{align*}
	\cT|_{\partial\varSigma} \ceq \{ \sigma \at_A \mid \text{$\sigma \in \cT$ and $A$ is a face of $\Delta_k$ such that $\sigma(A) \subset \partial \varSigma$}\}
\end{align*}
Moreover, we define the \emph{Lagrange basis functions} as the unique collection $\lambda_p$, $p \in V(\cT)$ of continuous functions satisfying the following conditions:
\begin{enumerate}
	\item $\lambda_p(p)=1$.
	\item $\lambda_p(q)=0$ for all $q\in V(\cT)\setminus\{p\}$.
	\item $\lambda_p \circ \sigma \colon \Delta_k \to \R$ is the restriction of an affine function for each $\sigma \in \cT$.
\end{enumerate}
We formulate the discrete area minimization problem in the following way:

\begin{problem}[Discrete least area problem]\label{prob:discreteproblem}
Minimize the \emph{discrete volume functional}
\begin{align*}
	\cF_\cT(f) 
	\ceq \sum_{\sigma \in \cT} \int_{\conv( f(V(\sigma)))} \dd \cH^{k},
\end{align*}
on the \emph{discrete configuration space}
\begin{align*}
	\cC_\cT \ceq 
	\Set{ f \colon V(\cT) \to \R^m | 
	\begin{array}{c}
	f|_{V_{\on{b}}(\cT)} = \Curve|_{V_{\on{b}}(\cT)},\\
	\text{$\forall \sigma \in \cT \colon f(V(\sigma))$ in general position}
	\end{array}
	} \subset \R^{m\, \on{card}(V(\cT))}.
\end{align*}
Here, $\conv$ denotes the convex hull operator and $\cH^{k}$ is the $k$-dimensional Hausdorff measure.
\end{problem}

Note that the manifold $\varSigma$ and its triangulation $\cT$ do not occur explicitly in the formulation of the problem, only the combinatorics of $\cT$. 
 
With the Lagrange basis functions $\lambda_p$, $p \in V(\cT)$, we define an interpolation operator
\begin{align*}
	\treverwo_\cT \colon \R^{m\, \on{card}(V(\cT))} \to W^{1,\infty}(\varSigma;\R^m), 
	\qquad 
	\treverwo_\cT(f)(x) \ceq \sum_{p \in V(\cT)} \lambda_{p}(x) \, f(p).
\end{align*}
Note that for every $f \in \cC_\cT$, the image of $\treverwo_\cT(f)$ in $\R^m$ is a union of non-degenerate $k$-dimensional Euclidean simplices (hence a triangle mesh if $k=2$),
which implies $\treverwo_\cT(f) \in \strongImm(\varSigma;\R^m)$.
By construction, we have $\cF(\treverwo_\cT(f)) = \cF_\cT(f)$.
We define the piecewise smooth mapping
\begin{align}
	\Curve_\cT \colon \partial \varSigma \to \R^m,
	\qquad 
	\Curve_\cT(x) \ceq \sum_{p \in V_{\on{b}}(\cT)} \lambda_p (x) \, \Curve(p).\label{eq:GammaT}
\end{align}
Observe that for each $f \in \cC_\cT$, the interpolation $\treverwo_\cT(f)$ restricted to $\partial \varSigma$ is identical to $\Curve_\cT$. Moreover, the image of $\Curve_\cT$ is a union of embedded $(k-1)$-dimensional simplices, if $\cT$ is sufficiently fine.\footnote{This tells us that $\cC_\cT=\emptyset$ may occur if the triangulation $\cT$ is too coarse at the boundary $\partial\varSigma$.}
In general, $\Curve_\cT$ and $\Curve$ do not coincide which is why we have to modify $\treverwo_\cT$ later in order to obtain a reconstruction operator $\reverwo_\cT \colon \cC_\cT \to \cC$ (see \autoref{sec:reconstructionoperator}). However, $\treverwo_\cT(f)$ represents the triangle mesh that is used for actual numerics and we aim at comparing $\MinF{}{}$ to $\treverwo_\cT(\MinF{\cT}{})$. Thus, we define the discrete test mapping $\varPsi_\cT$ by
\begin{align*}
	\varPsi_\cT \colon \cC_\cT \to \cX, 
	\qquad 
	f \mapsto (\strongShapeQuot\circ \treverwo_\cT)(f).
\end{align*}
We suppose that the discrete minimizers $\cM_\cT$ fulfill $\cM_\cT \subset \cA_\cT^r$ with some $r>0$ and with the set of \emph{discrete a~priori information} defined by
\begin{align*}
	\cA^r_\cT 
	\ceq \set{f \in \cC_\cT | 
	\text{$\dPosDefSec(g, \treverwo_\cT(f)^\pull g_0) \leq r$
	}}.
\end{align*}
This assumption reflects the desire that all simplices of discrete minimizers should be uniformly nondegenerate in the sense that the aspect ratios of the embedded simplices 
\begin{align*}
	\set{(f\circ \sigma)(\Delta_k) | f \in \MinF{\cT}{},\; \sigma \in \cT}
\end{align*}
are uniformly bounded (see also \autoref{rem:MinSurfDiscussion1} below).

\subsection{Sampling Operator}\label{sec:SamplingOperator}

We introduce the operator
\begin{align*}
	\interwo_\cT \colon \cA^s  \subset \cC \to \R^{m\, \on{card}(V(\cT))},
	\qquad
	f\mapsto f|_{V(\cT)}.
\end{align*}
It will turn out to be a sampling operator if $s>0$ and $\cT$ are appropriately chosen (see \autoref{goodapriories}). We also define the \emph{relative approximation errors} $\varrho^{0}(\cT)$, $\varrho^{1}(\cT)$ of the smooth triangulation $\cT$ by
\begin{align*}
 	\varrho^{0}(\cT) 
 	\ceq 
 	\sup_{\substack{f \in W^{2,\infty}_g(\varSigma;\R^m)\\ \dd f \neq 0}} 
 	\frac{\nnorm{f- f_\cT }_{L^{\infty}_g}}{\nnorm{\dd f}_{W^{1,\infty}_g}}
 	\quad
 	\text{and}
 	\quad
 	\varrho^{1}(\cT) 
 	\ceq 
 	\sup_{\substack{f \in W^{2,\infty}_g(\varSigma;\R^m)\\ \dd f \neq 0}} 
 	\frac{\nnorm{\dd f- \dd f_\cT}_{L^{\infty}_g}}{\nnorm{\dd f}_{W^{1,\infty}_g}},
\end{align*} 
where $f_\cT \ceq (\treverwo_\cT \circ \interwo_\cT)(f)$. Moreover, we introduce the following abbreviation:
\begin{align*}
	\varrho(\cT) \ceq \max \{\,
	\varrho^{(0)}(\cT),\;
	\varrho^{(0)}(\cT|_{\partial \varSigma}),\;
	\varrho^{(1)}(\cT),\;
	\varrho^{(1)}(\cT|_{\partial \varSigma})
	\,\}.
\end{align*}

\begin{lemma}\label{existenceoffinesmoothtriangulation}
Let $(\varSigma,g)$ be a compact, smooth Riemannian manifold with boundary. Then there are finite smooth triangulations with arbitrary small relative approximation errors, i.e., for every $\varepsilon>0$ there is a finite smooth triangulation $\cT$ of $\varSigma$ with $\varrho(\cT) \leq \varepsilon$. 
\end{lemma}
\begin{proof}
In order to construct a sequence $(\cT_n)_{n \in \N}$ with $\varrho(\cT_n) \stackrel{n \to \infty}{\longrightarrow} 0$, one may start with an arbitrary smooth triangulation and successively apply an affine, aspect ratio preserving subdivision scheme to the standard simplex $\Delta_k$. For example, this can be achieved by $4:1$-subdivision in the case $k=2$.
\qed
\end{proof}

\subsection{Convergence Theorem}

We have now all ingredients to state the main theorem of this section. Let $\varSigma$ be a compact, $k$-dimensional smooth manifold with boundary, let $g$ be a smooth Riemannian metric on $\varSigma$ and let $\Curve \in \Imm(\partial\varSigma;\R^m)\cap W^{2,\infty}(\partial \varSigma;\R^m)$ be an embedding and hence a bi-Lipschitz homeomorphism onto its image.
Let $(\cT_n)_{n\in\N}$ be a sequence of smooth triangulations of $(\varSigma,g)$ with $\varrho_n \ceq \varrho(\cT_n) \stackrel{n \to \infty}{\longrightarrow} 0$. Instead of $\cC_{\cT_n}$, $\cF_{\cT_n}$, $\cA^r_{\cT_n},\dotsc$ we shall write $\cC_n$, $\cF_n$, $\cA^r_n,\dotsc$
As in \autoref{sec:paramopt}, we denote the sets of $\delta$-minimizers by $\cM^\delta \ceq \argmin^\delta(\cF|_\cC)$ and $\cM_n^\delta \ceq \argmin^\delta(\cF_n)$.

\begin{theorem}\label{theo:ConvofMinimalSurfaces}
Suppose $\emptyset \neq \varPsi(\MinF{}{}) \subset \varPsi(\cA^s\cap\MinF{}{})$ for some $s\in \intervaloo{0,\infty}$ and that the sets $\cA_{n}^r$ are valid\footnote{see \autoref{def:valid}.} for some $r \in \intervaloo{s,\infty}$ and all $n \in \N$.
Then there is a constant $C\geq 0$ depending on $\varSigma$, $g$, $\Curve$, $s$, $r$ only, such that
\begin{align*}
	\Ls_{n\to \infty} \varPsi_n(\cA^r_n \cap \MinF{n}{}) 
	\subset \varPsi(\MinF{}{}) 
	= \Lt_{n\to \infty} \varPsi_n(\cA^r_n \cap \MinF{n}{\delta_n})
\end{align*}
holds with $\delta_n \leq C \, \varrho_n$ for sufficiently large $n \in \N$. The convergence is with respect to the topology generated by $\dstrongShape$.
\end{theorem}
\begin{proof}
We apply \autoref{topstabandconv} with $\cK \ceq \cX$. We are left with verifying the assumption of this theorem.
By \autoref{quotisHsd}, $(\cX,d_\cX)$ is a metric space.
Validity of $\cA^s$ with respect to $(\cF|_{\cC},\inter)$ follows from $\emptyset \neq \varPsi(\cM) \subset \varPsi(\cA^s\cap \cM)$ and the validity of $\cA^r_n$ was imposed as a condition.

Checking the remaining assumptions is the subject of the remainder of this section. We construct reconstruction operators $\reverwo_{\cT_n}$ in \autoref{sec:reconstructionoperator}. That $\interwo_{\cT_n}$ and $\reverwo_{\cT_n}$ are indeed sampling and reconstruction operators, i.e., that
\begin{align*}
	\interwo_{\cT_n}(\cA^s) \subset \cC_{\cT_n}
	\qand 
	\reverwo_{\cT_n}(\cA^r_{\cT_n}) \subset \cC
\end{align*}
for sufficiently ``fine'' triangulations, will be established in \autoref{goodapriories}. The very same lemma will also provide $\interwo_{\cT_n}(\cA^s) \subset \cA^r_{\cT_n}$.
Proximity and consistency errors will be computed in \autoref{lem:MinSurfProx} and \autoref{lem:MinSurfCons}. Finally, \autoref{lem:Fislowersemicont} shows that $\cF$~is topologically stable along $\varPsi$ over $\cX$. 
\qed
\end{proof}

\subsection{Reconstruction Operator}\label{sec:reconstructionoperator}

For $\Curve \in \Imm(\partial \varSigma;\R^m) \cap W^{2,\infty}(\partial \varSigma;\R^m)$, $f \in \Imm(\varSigma;\R^m) \cap W^{2,\infty}(\varSigma;\R^m)$, and $\Curve_\cT$ as defined in \eqref{eq:GammaT} we obtain the relative approximation errors:
\begin{align}
	\norm{\Curve-\Curve_\cT}_{W^{1,\infty}_g} 
	&\leq \nnorm{\dd \Curve}_{W^{1,\infty}_{g}} \, \varrho(\cT),
	\label{Lipschitzdistancebnd}
	\\
	\nnorm{f - (\treverwo_\cT\circ \interwo_\cT)(f)}_{W^{1,\infty}_g} 
	&\leq  \nnorm{\dd f}_{W^{1,\infty}_g} \, \varrho(\cT).
	\label{proximityoftrever}
\end{align}
Let $\on{ext} \colon W^{1,\infty}_{g|_{\partial \varSigma}}(\partial \varSigma;\R^m) \to W^{1,\infty}_g(\varSigma;\R^m)$ be a continuous, linear extension operator\footnote{Such an operator can be obtained, e.g., by choosing a smooth collar $\varPhi \colon \partial \varSigma \times \intervalco{0,1} \stackrel{\sim}{\to} U \subset \varSigma$ and by using the function $\chi \colon \intervalco{0,1} \to \R$, $\chi(t) = \exp\nparen{\frac{t^2}{t^2-t}}$: Then
\begin{align*}
	\on{ext}(u)(x) \ceq 
	\begin{cases}
		(u \otimes \chi)\circ \varPhi^{-1}(x), & x\in U,\\
		0, & x\in \varSigma\setminus U.
	\end{cases}
\end{align*}
is the desired extension operator.
}
 and 	let $u_\cT \ceq \on{ext}(\Curve-\Curve_\cT)$.
Now, \eqref{Lipschitzdistancebnd} provides us with the estimate
\begin{align}
	\norm{u_\cT}_{W^{1,\infty}_g} 
	\leq \norm{\on{ext}}\, \norm{\Curve-\Curve_\cT}_{W^{1,\infty}_g}
	\leq \norm{\on{ext}}\, \nnorm{\dd \Curve}_{W^{1,\infty}_{g}} \, 
	\varrho(\cT).\label{boundaryextension}
\end{align}
We define the following operator:
\begin{align*}
	\reverwo_\cT \colon \cA_\cT^r \to W^{1,\infty}_g(\varSigma;\R^m),
	\qquad 
	f \mapsto \treverwo_\cT(f) + u_\cT.
\end{align*}
The following lemma assures us that $\interwo_\cT$ and $\reverwo_\cT$ are indeed sampling and reconstruction operators, respectively, i.e., $\interwo_\cT(\cA^s) \subset \cC_\cT$ and $\reverwo_\cT(\cA^r_\cT) \subset \cC$---at least for sufficiently ``fine'' triangulations. It also verifies the condition $\interwo_\cT(\cA^s) \subset \cA^r_\cT$ of \autoref{topstabandconv}:
\begin{lemma}\label{goodapriories}
Let $r>s>0$ and $c>0$. Then there is $\varrho_0 >0$ such that for every smooth triangulation $\cT$ with $\varrho(\cT) \leq \varrho_0$ the following hold:
\begin{align*}
	\interwo_\cT(\cA^s) \subset \cA^r_\cT
	\qand
	\reverwo_\cT(\cA^r_\cT) \subset \cC.
\end{align*}
\end{lemma}
\begin{proof}
Let $f\in \cA^s$ and put $f_\cT\ceq (\treverwo_\cT \circ \interwo_\cT)(f)$. 
By \eqref{proximityoftrever}, we have $\norm{f-f_\cT}_{W^{1,\infty}_g} \leq s\, \varrho_0$.
\autoref{cor:W1inftyembLip} tells us how small $\varrho_0$ has to be (depending on $s$ only) so that $f_\cT \in \Imm(\varSigma;\R^m)$ and thus $\interwo_\cT(f) \in \cC_\cT$.
From \autoref{PosDefopeninSym},
we infer the inequality
\begin{gather*}
	\dPosDefSec(g, f_\cT^\pull g_0) 
	\leq \dPosDefSec(g, f^\pull g_0) + \dPosDefSec(f^\pull g_0,f_\cT^\pull g_0) 
	\leq s + C(s) \, \varrho_0
\end{gather*}
which shows that $ \interwo_\cT(f) \in \cA^r_\cT$ if $\varrho_0$ is sufficiently small.

Now, let $f \in \cA^r_\cT$. We have $\treverwo_\cT(f)\in \strongImm(\varSigma;\R^m)$ and $\dPosDefSec(g, \treverwo_\cT(f)^\pull g_0)\leq r$. By \eqref{boundaryextension}, we obtain
\begin{align*}
	\nnorm{\reverwo_\cT(f) - \treverwo_\cT(f)}_{W^{1,\infty}_g} 
	= \norm{u_\cT}_{W^{1,\infty}_g} 
	\leq C \, \nnorm{\dd \Curve}_{W^{1,\infty}_g} \, \varrho_0.
\end{align*}
\autoref{cor:W1inftyembLip} tells us how $\varrho_0$ has to be chosen depending on $r$ such that $\reverwo_\cT(f) \in \widetilde \cC$. Since $\reverwo_\cT(f)$ fulfills the boundary conditions by construction, we even have $\reverwo_\cT(f) \in \cC$.
\qed
\end{proof}

\subsection{Proximity}

\begin{lemma}\label{lem:MinSurfProx}
Let $r>s>0$. Then there is $\varrho_0 >0$ and a constant $C \geq 0$ such that 
for all smooth triangulations with $\varrho(\cT) \leq \varrho_0$, 
one has the following estimates for sampling and reconstruction proximity errors $\varepsilon_\cT^{\interwo}$, $\varepsilon_\cT^{\reverwo}$ of $(\varPsi,\varPsi_\cT,\interwo_\cT, \reverwo_\cT)$ on $(\cA^s,\cA^r_\cT)$:
\begin{align*}
	\varepsilon_\cT^{\interwo} \leq  C \, \varrho(\cT)
	\quad
	\text{and}
	\quad
	\varepsilon_\cT^{\reverwo} \leq  C \, \varrho(\cT).
\end{align*}
\end{lemma}
\begin{proof}
From \autoref{cor:W1inftyembLip}, \eqref{proximityoftrever}, and
\eqref{Lipschitzdistancebnd},
one obtains for the sampling proximity error:
\begin{align*}
	\varepsilon_\cT^{\interwo} 
	&=
	\sup_{f\in \cA^s} d_\cX \bigparen{ \varPsi(f), (\varPsi_n\circ \interwo_\cT) (f)}
	\\
	&=
	\sup_{f\in \cA^s} d_\cX \bigparen{ \strongShapeQuot(f), (\strongShapeQuot\circ \treverwo_\cT \circ \interwo_\cT) (f)}
	\\
	&\leq	
	\sup_{f\in \cA^s} \dstrongImm\bigparen{ f, (\treverwo_\cT \circ \interwo_\cT) (f)}
	\\
	&\leq	
	\sup_{f\in \cA^s} \dImm\bigparen{ f, (\treverwo_\cT \circ \interwo_\cT) (f)}
	+ \dImm ( \gamma, \gamma_\cT)
	\\
	&\leq
	C(s) \, \sup_{f\in \cA^s} \nnorm{f-(\treverwo_\cT \circ \interwo_\cT) (f)}_{W^{1,\infty}_g} + C\, \nnorm{\gamma-\gamma_\cT}_{W^{1,\infty}_g}
	\leq C\, \varrho(\cT).
\end{align*}
\autoref{cor:W1inftyembLip}, \eqref{boundaryextension}, and \eqref{Lipschitzdistancebnd} imply for the reconstruction proximity error:
\begin{align*}
	\varepsilon_\cT^{\reverwo} 
	&=
	\sup_{f\in \cA^s_\cT} d_\cX \bigparen{ \varPsi_\cT(f), (\varPsi \circ \reverwo_\cT)(f)}
	\\
	&=
	\sup_{f\in \cA^s_\cT} d_\cX \bigparen{ (\strongShapeQuot \circ \treverwo_\cT)(f), (\strongShapeQuot \circ \reverwo_\cT)(f)}
	\\
	&\leq
	\sup_{f\in \cA^s_\cT} \dstrongImm \bigparen{ \treverwo_\cT(f), \reverwo_\cT(f)}
	\\
	&\leq
	\sup_{f\in \cA^s_\cT} \dImm \bigparen{ \treverwo_\cT(f), \reverwo_\cT(f)} + \dImm ( \gamma_\cT, \gamma)
	\\
	&\leq C(s) \, \sup_{f\in \cA^s_\cT} \nnorm{\treverwo_\cT(f)-\reverwo_\cT(f)}_{W^{1,\infty}_g} + C\, \nnorm{\gamma-\gamma_\cT}_{W^{1,\infty}_g}
	\\
	&\leq C(s) \, \nnorm{u_\cT}_{W^{1,\infty}_g} + C\, \nnorm{\gamma-\gamma_\cT}_{W^{1,\infty}_g}
	\leq C\, \varrho(\cT).
\end{align*}
\qed
\end{proof}

\subsection{Consistency}

\begin{lemma}\label{lem:MinSurfCons}
Let $r>s>0$. Then there is $\varrho_0 >0$ and a constant $C \geq 0$ such that 
for all smooth triangulations with $\varrho(\cT) \leq \varrho_0$, 
one has the following estimates for sampling and reconstruction consistency errors $\delta_\cT^{\interwo}$, $\delta_\cT^{\reverwo}$ of $(\cF_\cT,\interwo_\cT, \reverwo_\cT)$ with respect to $\cF$ on $(\cA^s,\cA^r_\cT)$:
\begin{align*}
	\delta_\cT^{\interwo} \leq  C \, \varrho(\cT)
	\quad
	\text{and}
	\quad
	\delta_\cT^{\reverwo} \leq  C \, \varrho(\cT).
\end{align*}
\end{lemma}
\begin{proof}
Fix $f \in \cA^s$ and abbreviate $f_\cT \ceq (\treverwo_\cT \circ \interwo_\cT)(f)$. Observe $(\cF_\cT \circ \interwo_\cT)(f) = \cF(f_\cT)$.
By \autoref{volumeislipschitz1}, \autoref{cor:W1inftyembLip}, and \eqref{proximityoftrever}, we obtain
\begin{align*}
	(\cF_\cT \circ \interwo_\cT)(f)  - \cF(f)
	&= \cF (f_\cT)  - \cF(f)
	\\
	&\leq \cF(f) \, \ee^{\sqrt{k} \, \dImm(f_\cT,f)}\,\sqrt{k} \, \dImm(f_\cT,f)
	\\
	&\leq C(s) \, \norm{f_\cT-f}_{W^{1,\infty}_g} \leq C(s) \, \varrho(\cT),
\end{align*}
thus the sampling consistency error $\delta^\interwo_\cT$ is bounded by $\delta^\interwo_\cT \leq C \, \varrho(\cT)$.

For $f \in \cA_\cT^r$, put $f_\cT = \treverwo_\cT(f)$.
From local Lipschitz continuity of $\cF$ (see \autoref{volumeislipschitz1}) and \eqref{boundaryextension}, we deduce
\begin{align*}
	(\cF \circ \reverwo_\cT)(f) - \cF_\cT(f)
%	&= (\cF \circ \reverwo_\cT)(f) - \cF(f_\cT)
%	\\
	&= \cF (f_\cT + u_\cT) - \cF(f_\cT)
	\\
	&\leq \cF(f_\cT) \,	C(r) \, \norm{u_\cT}_{W^{1,\infty}_g} \leq C(r) \, \varrho(\cT).
\end{align*}
Thus, we obtain $\delta^\reverwo_\cT \leq C \, \varrho(\cT)$ as desired. 
\qed
\end{proof}

\subsection{Topological Stability}

\begin{lemma}\label{lem:Fislowersemicont}
Let $\gamma \in C^1(\partial \varSigma;\R^m)$ be a bi-Lipschitz homeomorphism.
Then the function $\varPsi_\push\iota_\Curve^\pull \cF$ is topologically stable along $\varPsi$ over $\strongShape(\varSigma;\R^m)$.
\end{lemma}
\begin{proof}
We assume $\Immg(\varSigma;\R^m) \neq \emptyset$, thus $\inf(\iota_\Curve^\pull\cF) < \infty$.
For each $x \in \strongShape(\varSigma;\R^m)$ with $(\varPsi_\push\iota_\Curve^\pull \cF)(x) < \infty$ there is an $f \in \Immg(\varSigma;\R^m)$ with $\strongShapeQuot(f) =x$. Moreover,
$\varPsi^{-1}(x)$ is contained in the $\Diff(\varSigma)$-orbit $f \circ \Diff(\varSigma)$ and $\cF$ is constant on all orbits.
Thus, $\iota_\gamma^\pull \cF$ is constant on $\varPsi^{-1}(x)$ and we have
$(\iota_\Curve^\pull\cF)(f) = \inf_{h \in \varPsi^{-1}(x)} (\iota_\Curve^\pull\cF)(h)$.
Applying \autoref{pushforwardproperties}, we obtain
\begin{align*}
	\varPsi(\argmin^\varrho(\iota_\Curve^\pull \cF))
	=
	\argmin^\varrho(\varPsi_\push\iota_\Curve^\pull \cF).
\end{align*}
Furthermore, this also shows that
\begin{align*}
	(\varPsi_\push\iota_\Curve^\pull \cF)(x)
	=
	\begin{cases}
		(\strongShapeQuot_\push \cF)(x), & \text{$x \in \strongShapeQuot(\Immg(\varSigma;\R^m))$,}
		\\
		\infty, &\text{else}.
	\end{cases}
\end{align*}
By \autoref{lem:closednessofImmginstrongShape} below, the set $\strongShapeQuot(\Immg(\varSigma;\R^m))$ is closed  and by \autoref{volumeislipschitz2}, the function
$(\strongShapeQuot_\push \cF)$ is continuous. Hence $\varPsi_\push\iota_\Curve^\pull \cF$ is lower semi-continuous.
Finally, \autoref{lem:lscimpliesstability} finishes the proof.
\qed
\end{proof}

\begin{lemma}\label{lem:closednessofImmginstrongShape}
Let $\varSigma$ be a compact smooth manifold with boundary and let $\Curve$ be a \emph{continuously} differentiable and bi-Lipschitzian homeomorphism onto its image.
\newline
Then $\strongShapeQuot(\Immg(\varSigma;\R^m))$ is a closed subset of $\strongShape(\varSigma;\R^m)$.
\end{lemma}
\begin{proof}
Fix a smooth Riemannian metric $g$ on $\varSigma$. 
Let $(x_n)_{n \in \N}$ be a  sequence in the set $\strongShapeQuot(\Immg(\varSigma;\R^m))$ that converges to some $x \in \Shape(\varSigma;\R^m)$. In particular, $(x_n)_{n \in \N}$ is a Cauchy sequence. Let $F_n \in \Immg(\varSigma;\R^m)$ such that $\strongShapeQuot(F_n)=x_n$.
We recursively choose $\varphi_n \in \Diff(\varSigma)$ fulfilling
$
	\dstrongImm(F_{n+1} \circ \varphi_{n+1},F_n \circ \varphi_n) \leq \dShape(x_{n+1},x_n) + 2^{-n}. 
$
This way, $(f_n \ceq F_n \circ \varphi_n)_{n\in\N}$ is also a Cauchy sequence with respect to $\dstrongImm$ and thus converges to some $f \in \strongImm(\varSigma;\R^m)$ (see \autoref{theo:Immiscomplete}).
By construction, we have 
$\strongShapeQuot(f) = x$ so that it suffices to prove the existence of a $\varphi \in \Diff(\varSigma)$ with $f \circ \varphi^{-1} \in \Immg	(\varSigma;\R^m)$.
We do this by first discussing its restriction onto the boundary.

\textbf{Claim I:} \emph{$\psi_n \ceq \res(\varphi_n)$ converges in $W^{1,\infty}$ to a $\psi \in \Diff(\varSigma)$.}\newline
We have $h_n \ceq \res(f_n) \converges{n \to \infty} h \ceq \res(f)$ in $\dstrongImm$ by \autoref{lem:strongImmhastrace}.
Observe that $h_n = \Curve \circ \psi_n$, thus $h_n$ is actually a bi-Lipschitz homeomorphism onto $\Curve(\partial \varSigma)$. 
Since $\Curve(\partial \varSigma)$ is compact and thus closed, we also have $h(\partial \varSigma) \subset \Curve(\partial\varSigma)$, so that we may define $\psi \ceq \Curve^{-1} \circ h$.
Since $\dd \Curve$ is uniformly continuous, this leads to
$\psi_n = \Curve^{-1} \circ h_n \converges{n\to \infty} \psi$ in $W^{1,\infty}_g$.
By the chain rule \eqref{eq:chainrule}, we have
$\dd_{\psi_n^{-1}(x)} h_n = \dd_x \gamma  \cdot T_{\psi_n^{-1}(x)} \psi_n$.
As $\dd_{\psi_n^{-1}(x)} h_n$ and $\dd_x \gamma$ have the same image in $\R^m$, we may compute
\begin{align*}
	T_{x}(\psi_n^{-1}) 
	=  (T_{\psi_n^{-1}(x)}\psi_n)^\dagger 
	=  \big( (\dd_x \gamma)^\dagger \cdot \dd_{\psi_n^{-1}(x)} h_n \big)^\dagger 	
	= (\dd_{\psi_n^{-1}(x)} h_n)^\dagger\cdot \dd_x \Curve
	\qquad
	\text{for $x \in \partial \varSigma$.} 
\end{align*}
Note that $\psi_n$ is a bijection, hence this shows
$\nnorm{T(\psi_n^{-1})}_{L^\infty_g} \leq \nnorm{\dd  h_n^\dagger}_{L^\infty} \, \nnorm{\dd \Curve}_{L^\infty}$.
Since $(h_n)_{n\in\N}$ is bounded in $\dstrongImm$, the sequence $(\psi_n^{-1})_{n \in \N}$ is equicontinuous. 
Thus there is a subsequence of $(\psi_n^{-1})_{n\in\N}$ that converges to $\psi^{-1} \in C^0(\partial \varSigma;\partial\varSigma)$ in the compact-open topology and $\psi^{-1}$ is Lipschitz continuous. Thus $\psi$ is a Lipschitz diffeomorphism.

\textbf{Claim II:} \emph{$\psi$ can be extended to a $\varphi \in \Diff(\varSigma)$, i.e., one has $\res(\varphi) = \psi$.}\footnote{Note that the restriction mapping $\res \colon \Diff(\varSigma) \to \Diff(\partial \varSigma)$ need not be surjective. This can already be observed when $\varSigma$ is a two-dimensional cylinder.}\newline
Let $r>0$ be the geodesic radius of $(\partial \varSigma,g)$, i.e.,
for each $x \in \partial \varSigma$ the Riemannian exponential map
$\exp_x \colon K_x \to \partial \varSigma$
with $K_x \ceq \set{u \in T_x \partial \varSigma | \abs{u}_g \leq r}$ is a smooth diffeomorphism onto its image. 
Thus, for $n \in \N$ with $\sup_{x \in \varSigma} d_g(\psi(x),\psi_n(x)) <r$, there is a unique $v_n \in W^{1,\infty}(\partial \varSigma;T\partial\varSigma)$ with $v_n(x) \in T_{\psi(x)}\varSigma$ and $\psi_n(x) = \exp( v_n(x))$ for all $x \in \partial \varSigma$. Now define
$H_n \colon \partial\varSigma \times \intervalcc{0,1} \to \partial\varSigma \times \intervalcc{0,1}$ by
$H_n(x,t) \ceq (\exp_x(t\, v_n(x)),t)$ and observe that $H_n$ is a Lipschitz diffeomorphism for sufficiently large $n$ and that $H_n(x,0) = (\psi(x),0)$ and $H_n(x,1) = (\psi_n(x),1)$ hold for all $x \in \partial \varSigma$.
We fix such an $n \in \N$ now.

In order to transport $H_n$ to $\varSigma$, we choose a smooth collar $\varPhi \colon \partial \varSigma \times \intervalco{0,2} \to \varSigma$, i.e., a smooth embedding with $\varPhi(x,0)=x$ for all $x \in \partial \varSigma$.
We define the closed sets
$A \ceq \varPhi(\varSigma \times \intervalcc{0,1})$
and
$B \ceq \varSigma\setminus \varPhi(\partial \varSigma \times \intervalco{0,1})$.
We also shoose a smooth diffeomorphism $f \colon \varSigma \to B$ 
with $f\circ \varPhi(x,0) = \varPhi(x,1)$ for all $x \in \partial \varSigma$.
Note that $f|_{\partial \varSigma} \colon \partial \varSigma \to A\cap B$ is also a smooth diffeomorphism.
The mappings $\varphi_B \ceq f \circ \varphi_n \circ f^{-1} \colon B \to B$
and $\varphi_{A\cap B} \ceq (f|_{\partial \varSigma}) \circ \psi_n \circ (f|_{\partial \varSigma})^{-1} \colon A\cap B \to A \cap B$
are Lipschitz diffeomorhisms and one has $\varphi_B|_{A\cap B} = \varphi_{A\cap B}$.
Moreover, the mapping $\varphi_A = \varPhi|_{(\varSigma \times \intervalcc{0,1})} \circ H_n \circ (\varPhi|_{(\varSigma \times \intervalcc{0,1})})^{-1}$ 
is a Lipschitz diffeomorphism of $A$ and
satisfies $\varphi_A|_{\partial \varSigma} = \psi$ and $\varphi_A|_{A \cap B} = \varphi_{A\cap B}$.
Hence there is unique continuous map $\varphi \colon \varSigma \to \varSigma$
that makes the following diagramm commutative:
\begin{equation*}
\begin{tikzcd}[row sep=scriptsize, column sep=scriptsize]
	& A\cap B 
		\ar[dl, hook'] 
		\ar[rr, hook] 
		\ar[dr, hook] 		
		\ar[dd, near end, "\varphi_{A\cap B}"'] 
	& & 
	B 
		\ar[dl, hook'] 
		\ar[dd, "\varphi_B"] 
	\\ 
	A 
		\ar[rr, crossing over, hook] 
		\ar[dd, "\varphi_A"'] 
	& & 
	\varSigma 
	\\
	& A\cap B 
		\ar[dl, hook'] 
		\ar[rr, hook] 
		\ar[dr, hook] 		
	& & 
	B 
		\ar[dl, hook'] 
	\\ 
	A 
		\ar[rr, hook] 
	& & 
	\varSigma
		\ar[from=uu, crossing over, dashed, near start, "\varphi"]	
\end{tikzcd}
\end{equation*}
The inverse $\varphi^{-1}$ exists and is also continuous. Both $\varphi$ and $\varphi^{-1}$ are piecewise of class $W^{1,\infty}$ and continuous. Thus,  we have $\varphi$, $\varphi^{-1} \in W^{1,\infty}(\varSigma;\varSigma)$, hence $\varphi \in \Diff(\varSigma)$. For $x \in \partial \varSigma \subset A$, we have $\varphi(x) = \varphi_A(x) = \varPhi( H_n(x,0))= \varPhi( \psi(x,0)) =\psi(x)$, hence $\res(\varphi) =\psi$ as desired.
\qed
\end{proof}

\subsection{Concluding Remarks}
\begin{figure}[ht]
\begin{center}
\begin{minipage}{0.32\textwidth}
\includegraphics[width=\textwidth]{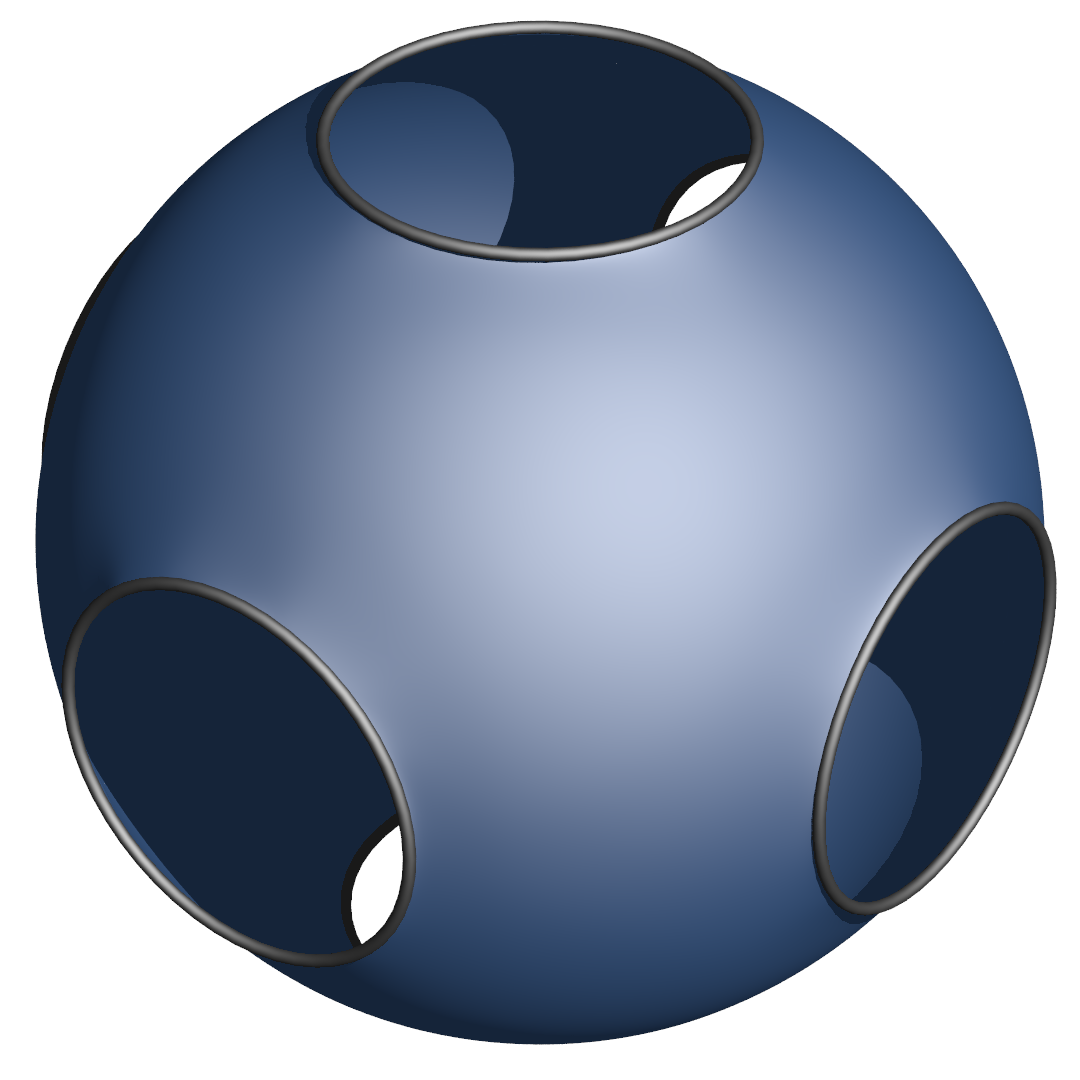}
\end{minipage}
\begin{minipage}{0.32\textwidth}
\includegraphics[width=\textwidth]{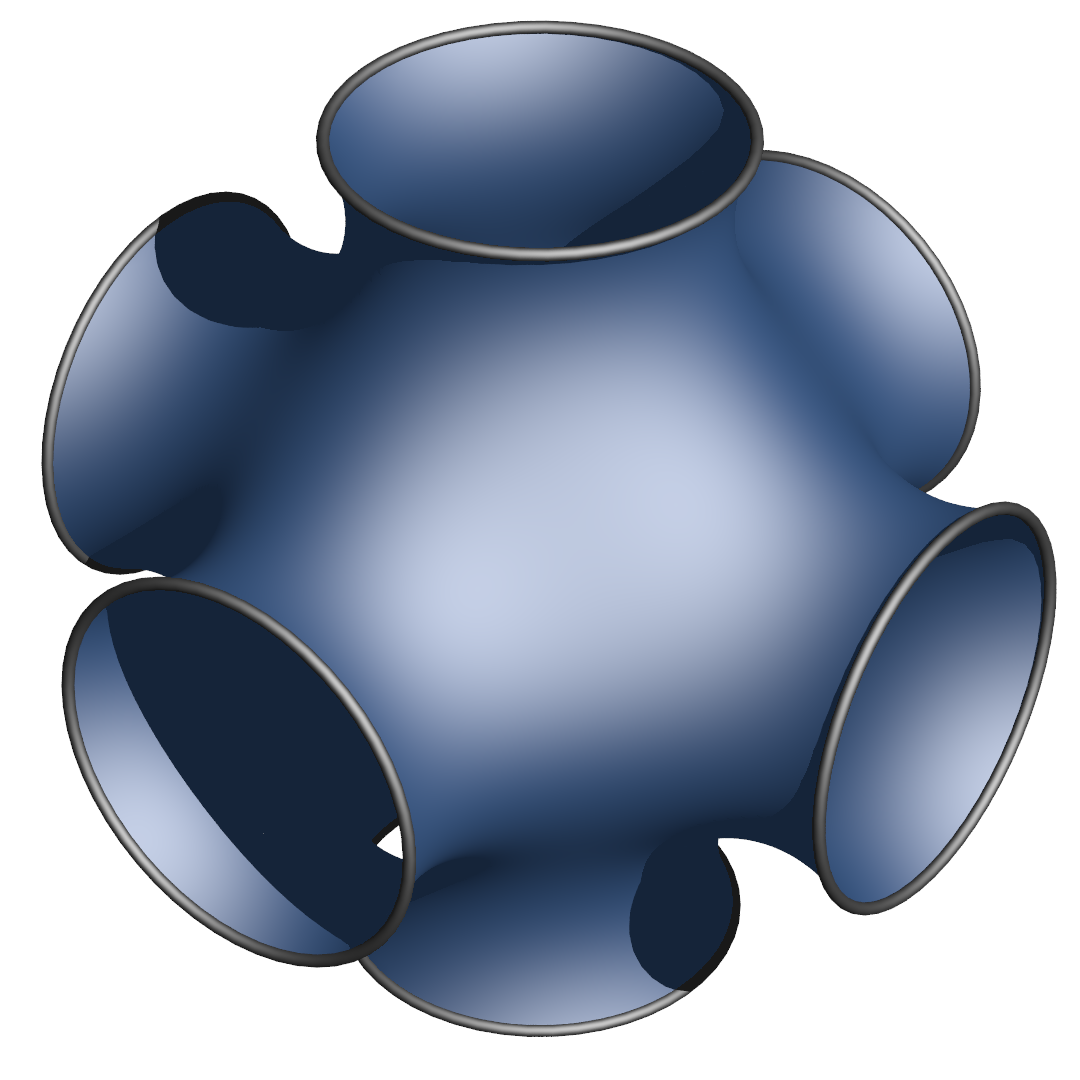}
\end{minipage}
\begin{minipage}{0.32\textwidth}
\includegraphics[width=\textwidth]{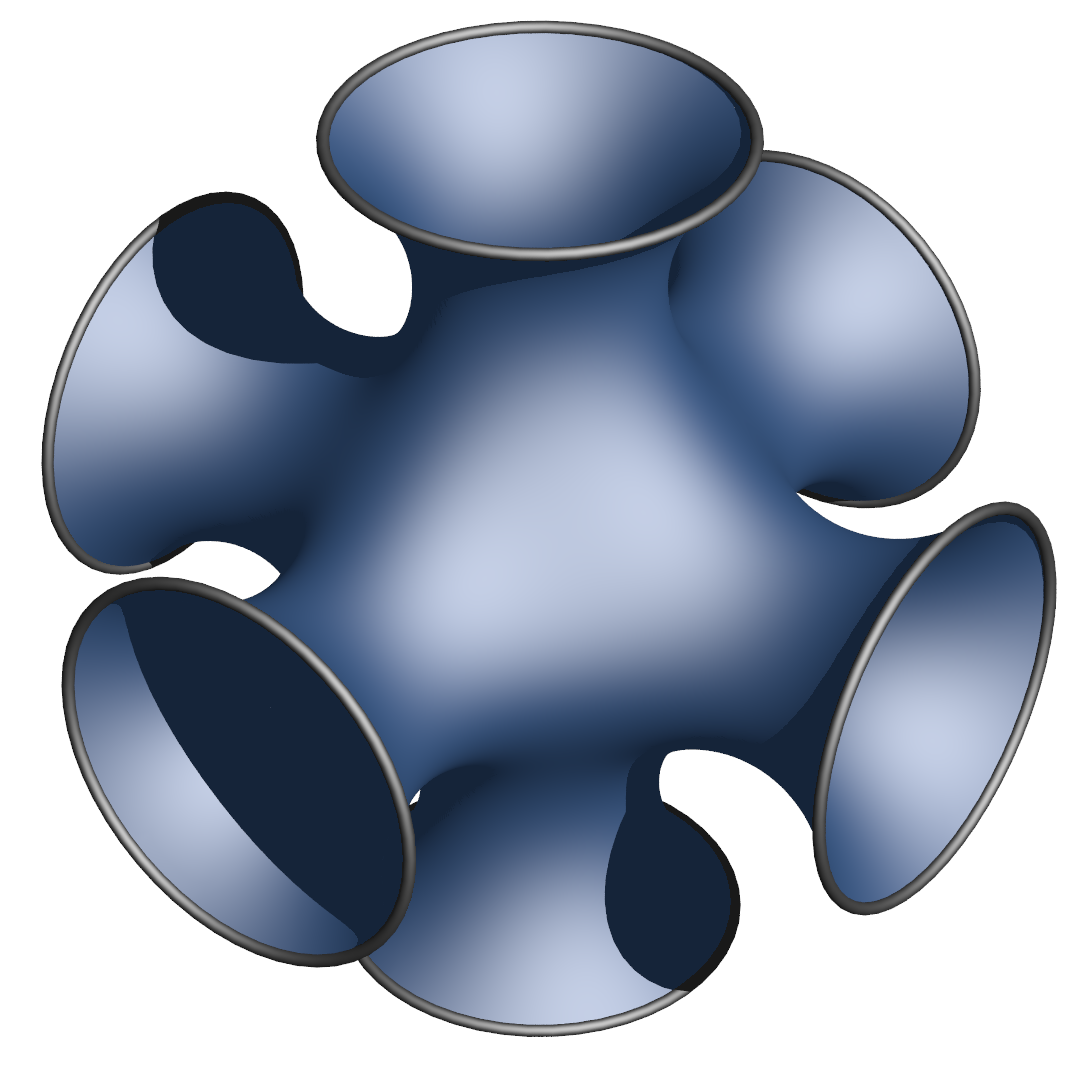}
\end{minipage}
\begin{minipage}{0.32\textwidth}
\includegraphics[width=\textwidth]{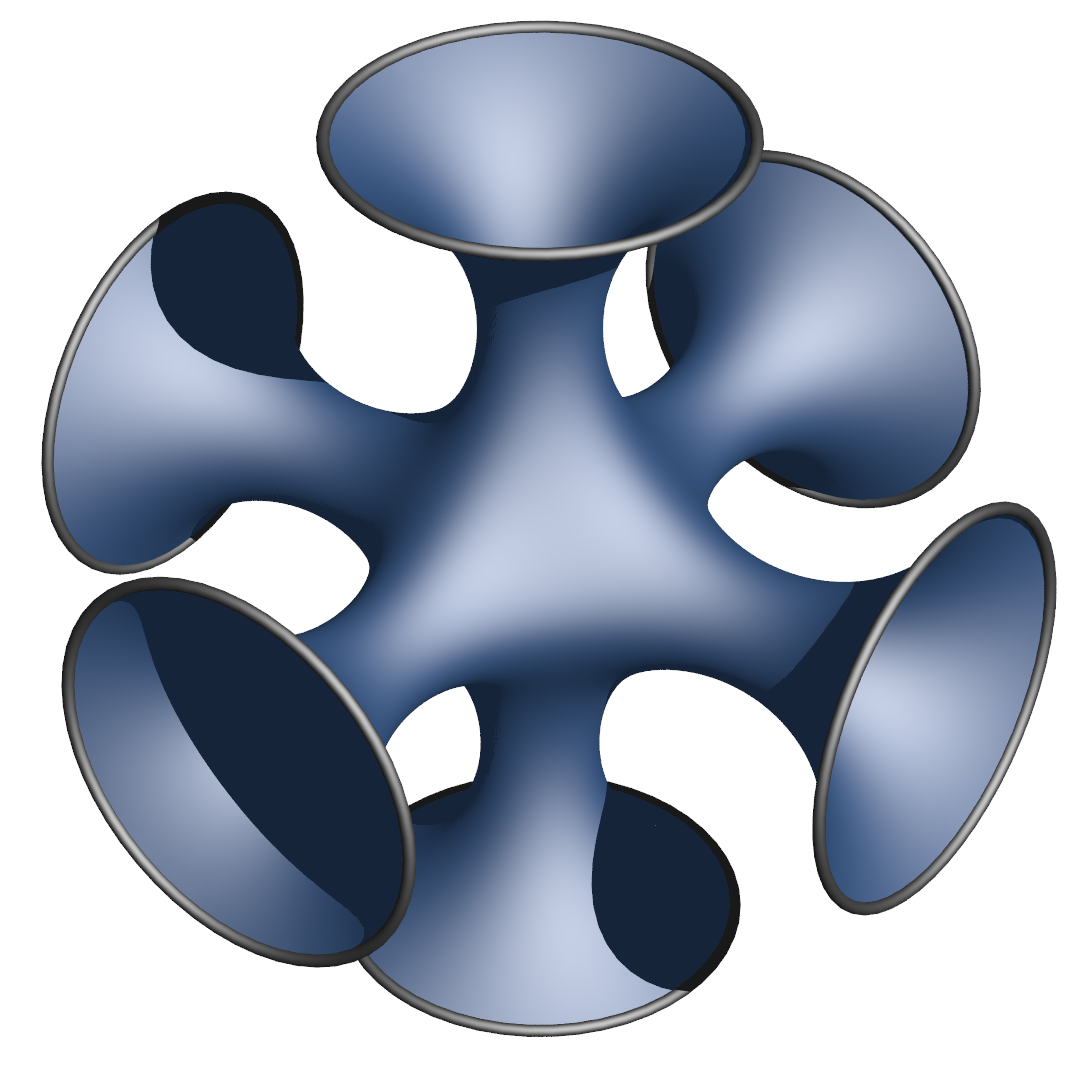}
\end{minipage}
\begin{minipage}{0.32\textwidth}
\includegraphics[width=\textwidth]{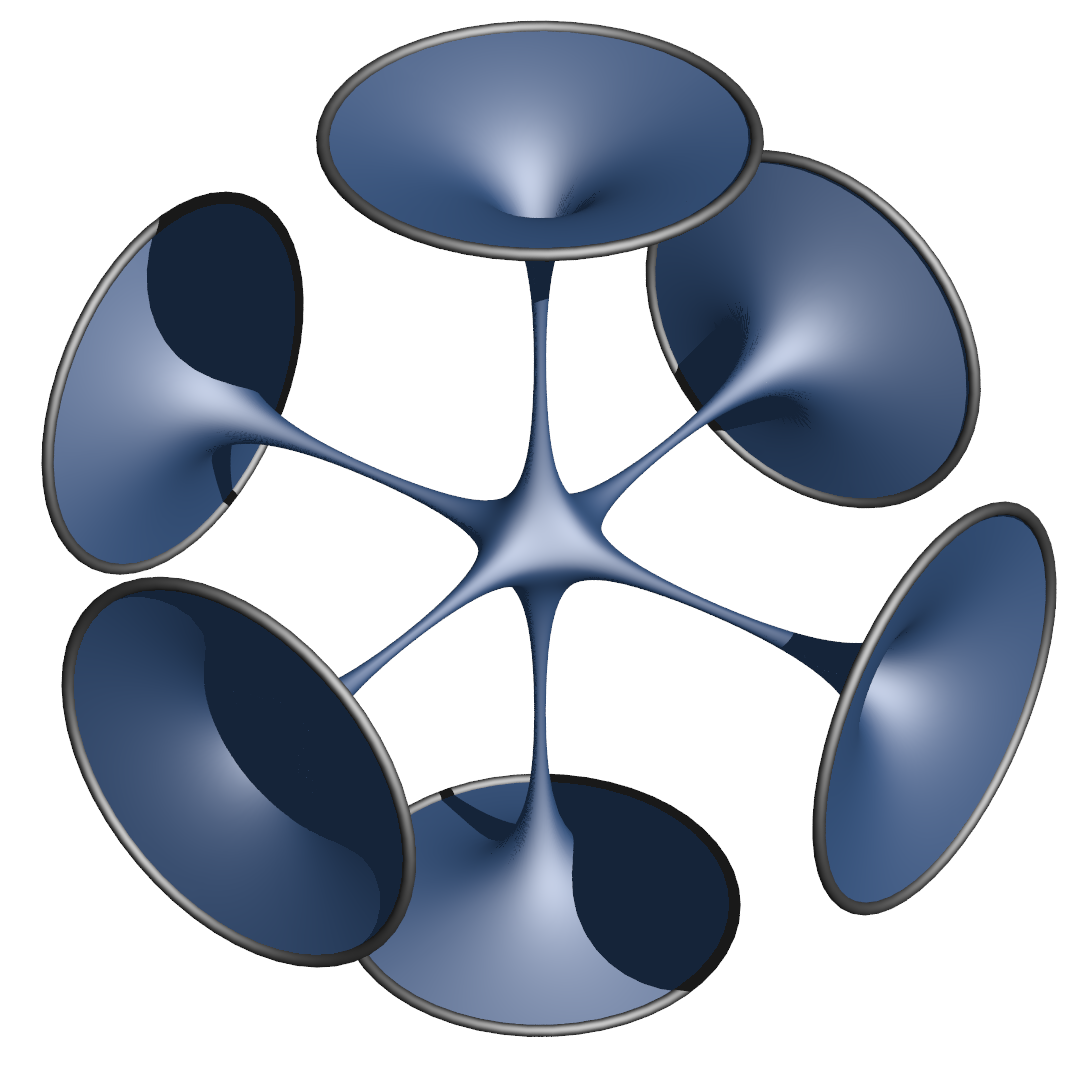}
\end{minipage}
\begin{minipage}{0.32\textwidth}
\includegraphics[width=\textwidth]{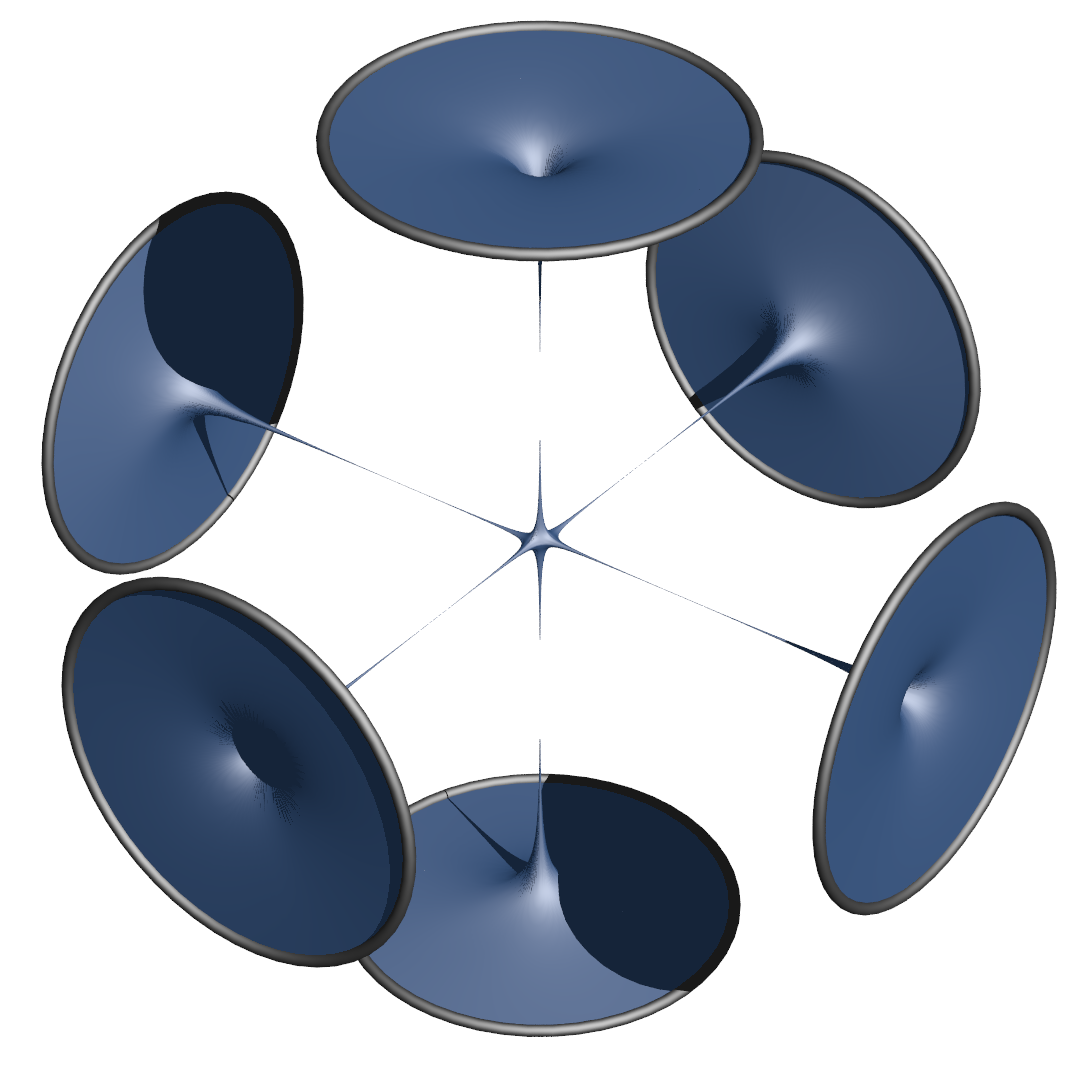}
\end{minipage}
\caption[Sequence of surfaces obtained by a descending flow.]{Sequence of surfaces obtained by a descending flow starting from a sixfold perforated sphere and degenerating into six disks.}
\label{fig:catastrophe}
\end{center}
\end{figure}

\begin{remark}\label{rem:MinSurfDiscussion1}
The validity of the discrete a~priori assumption $\cA_{\cT}^r$ amounts to the existence of a non-degenerating minimizing sequence of simplicial meshes in $\R^m$. More precisely, $\cA_{\cT}^r$ is valid if and only if for each $l \in \N$ there is $f_l \in \cC_\cT$ such that $(\treverwo_\cT(f_l))_{l \in \N}$ is bounded in $\dImm$ and $\lim_{l \to \infty} \cF_\cT(f_l) = \inf(\cF_\cT)$.
In particular, for each $\sigma \in \cT$, the affine mappings $h_l \ceq \treverwo_\cT(f_l) \circ \sigma \colon \Delta_k \to \R^m$ have to be bounded in $( \Imm(\Delta_k; \R^m), \dImm)$. This implies uniform bounds on $\nnorm{\dd h_l}$ and $\nnorm{\dd h_l^\dagger}$ for all $l \in \N$. Note that $\nnorm{\dd h_l}$ and $\nnorm{\dd h_l^\dagger}$ are descriptors for the quality of the simplex $h_l(\Delta_k)$ since $h_l$ is affine. In fact, our experiments have shown, that a degenerating minimizing sequence, such as depicted in \autoref{fig:catastrophe}, may occur. In the depicted example, this is caused by the non-existence of minimizers in the prescribed topological class.\footnote{Still, the sequence seems to converge in a weaker topology to a ``minimizer'', e.g., in the sense of integral currents.}
\end{remark}

\begin{remark}\label{rem:MinSurfDiscussion2}
It might be possible to infer convergence rates from a thorough analysis of the spectral gap of the Hessians of the area functional in the vicinity of smooth minimizers. However, this is beyond the scope of this paper. 
\end{remark}

\begin{figure}[ht]
\begin{center}
\begin{minipage}{0.49\textwidth}
\includegraphics[width=\textwidth]{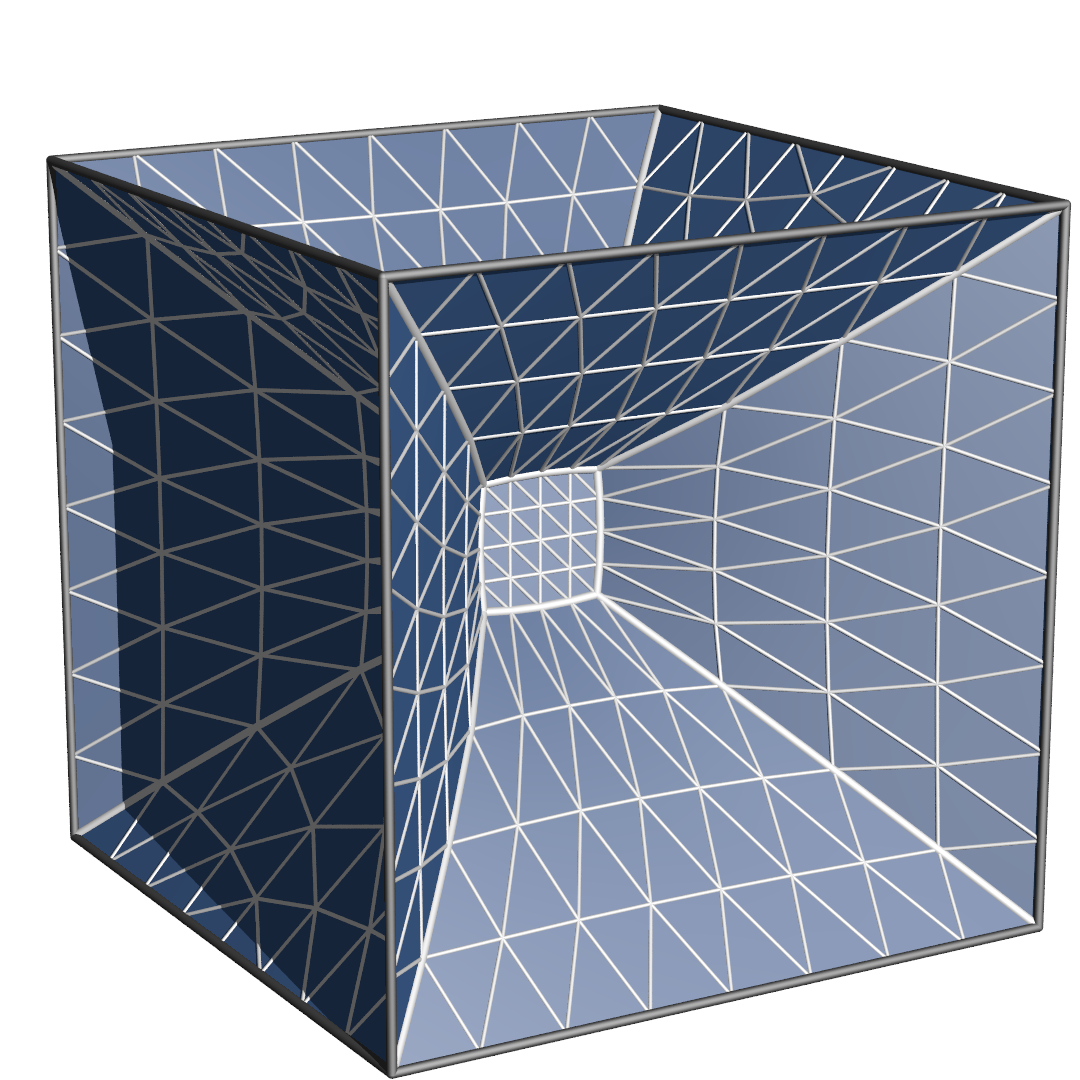}
\end{minipage}
\begin{minipage}{0.49\textwidth}
\includegraphics[width=\textwidth]{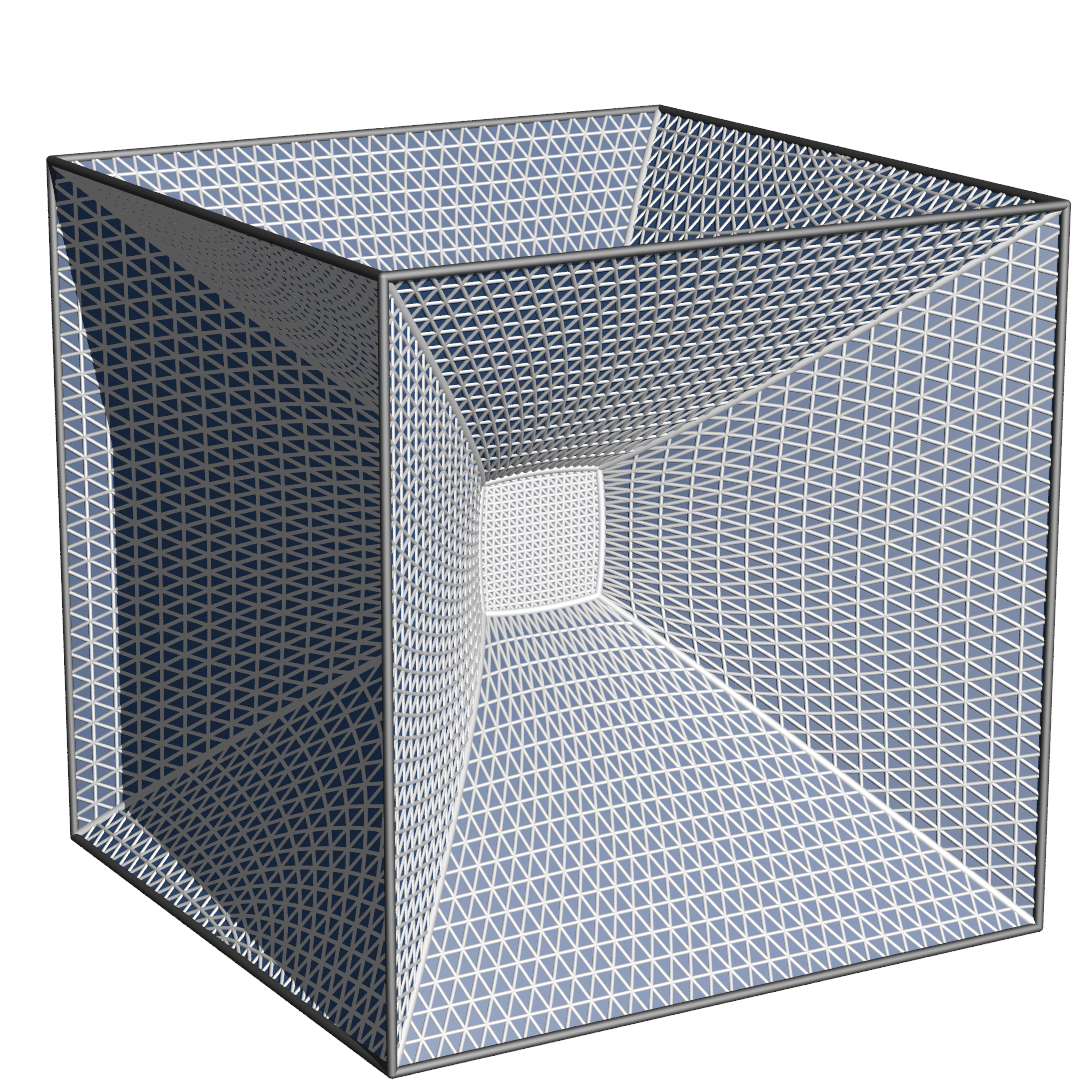}
\end{minipage}
\caption{
Discrete area minimizers of nonmanifold type spanned into the edge skeleton of a cube. 
Triple lines are indicated by thick white lines.
Observe that the aspect ratios of triangles stabilize under refinement, indicating that the discrete a~priori set $\cA_\cT^r$ is indeed valid.
}
\label{fig:Cube}
\end{center}
\end{figure}

\appendix
\newpage
% !TEX root = main.tex

\section{Thickening Robustness of Kuratowski Limits}\label{sec:Kuratowski}

For the tickenings of set as introduced in \autoref{dfn:thickenings}, we require the following result.

\begin{lemma}\label{lem:Kuratowskicinvafterthickening}
Let $(X,d)$ be a metric space, $r_n \geq 0$ with $\lim_{n\to \infty} r_n =0$ and $A_n\subset X$.
Then one has
\begin{align*}
%	 \Li_{n\to \infty} B(A_n, r_n) = 
	 \Li_{n\to \infty} \bar B(A_n, r_n)  =  \Li_{n\to \infty} A_n
	 \qand
%	 \Li_{n\to \infty} B(A_n, r_n) = 
	 \Ls_{n\to \infty} \bar B(A_n, r_n)  =  \Ls_{n\to \infty} A_n.
\end{align*}
\end{lemma}
\begin{proof}
Monotonicity of the Kuratowski limits leads to
\begin{align*}
%	\check A \ceq 
	\Li_{n\to \infty} A_n \subset \Li_{n\to \infty} \bar B(A_n, r_n),
	\qand
%	\hat A \ceq 
	\Ls_{n\to \infty} A_n \subset \Ls_{n\to \infty} \bar B(A_n, r_n).
\end{align*}
Fix arbitrary $a \in \Ls_{n\to \infty} \bar B(A_n, r_n)$, $\varepsilon >0$, and $N \in \N$ with $r_k < \tfrac{\varepsilon}{3}$ for all $k \geq N$. By definition of the Kuratowski limit superior, for each $n \geq N$ there is a $k \geq n$ and an $a_k \in B(a,\tfrac{\varepsilon}{3}) \cap \bar B(A_k,r_k)$. Now, choose $x_k \in A_k \cap B(a_k, 2 r_k)$ and observe $d(a,x_k) \leq d(a,a_k) + d(a_k,x_k) <\tfrac{\varepsilon}{3} + 2 r_k \leq \varepsilon$. Thus, we have $x_k \in B(a,\varepsilon) \cap A_k \neq \emptyset$ for all such $k$. This shows $a \in \Ls_{n\to \infty} A_n$. 

Analogously, one shows $\Li_{n\to \infty} A_n \supset \Li_{n\to \infty} \bar B(A_n, r_n)$. \qed
\end{proof}
\newpage
% !TEX root = main.tex

\section{Estimates for Lipschitz Immersions}

In \autoref{cor:W1inftyembLip}, we give a lower bound on the amount of $W^{1,\infty}_g$-perturbation that may be applied to a given Lipschitz immersion without leaving the space of Lipschitz immersions. Moreover, it allows us to to bound $\dImm$ locally from above by the $W^{1,\infty}_g$-distance.  
\autoref{cor:LipembW1infty} provides us with a reverse local bound.

In the following, we denote the operator norm of a linear operator $A$ by $\nnorm{A}$, while we use $\nabs{A}$ for the Frobenius norm. We will often write inner products with which these norms are defined as an index in order to resolve ambiguities.

\begin{lemma}\label{PosDefopeninSym}
Fix $b$, $g \in \PosDef(V)$ and let $X \in T_b \PosDef(V) =\Sym(V)$ with $\abs{X}_{g} < \ee^{-\dPosDef(b,g)}$. Then $b+X$ is also contained in $\PosDef(V)$ and one has
\begin{align*}
	\dPosDef(b,b + X) &\leq \ee^{\dPosDef(g,b)} \, \abs{X}_g.
%	\\
%	\dPosDef(g,b + X) &\leq \dPosDef(g,b) + \ee^{\dPosDef(g,b)} \, \abs{X}_g.
\end{align*}
In particular, $\PosDef(V)$ is open in $\Sym(V)$.
\end{lemma}
\begin{proof}
Choose a $g$-orthonormal basis $e$ of $V$ and define $\mathbf{B} \ceq \Gram_e(b)$, $\mathbf{X} \ceq \Gram_e(X)$.
Let $0<\lambda_1 \leq \dotsm \leq \lambda_k$ be the eigenvalues of $\mathbf{B}$. Observe that
\begin{align*}
	\nnorm{\mathbf{B}^{-1}} 
	&= \lambda_1^{-1}
	\leq \exp (\nabs{\log(\lambda_1)})
	\leq \exp (\nabs{(\log(\lambda_1),\dotsc, \log(\lambda_k))})
%	\leq \exp \left(\textstyle \sum_{i=1}^k \nabs{\log(\tfrac{1}{\lambda_1})}^2 \right)^\frac{1}{2}
	= \ee^{\dPosDef(b,g)}.
\end{align*}
The estimate
\begin{align*}
	\nnorm{\mathbf{B}^{-\frac{1}{2}}\mathbf{X}\mathbf{B}^{-\frac{1}{2}}} 
	\leq \nnorm{\mathbf{B}^{-1}} \, \nnorm{\mathbf{X}}
	\leq \nnorm{\mathbf{B}^{-1}} \, \nabs{\mathbf{X}}
	\leq \ee^{\dPosDef(b,g)} \cdot \abs{X}_g<1
\end{align*}
shows that 
$
\mathbf{B} + t \, \mathbf{X} =  \mathbf{B}^{\frac{1}{2}}\nparen{\mathbf{I}_{n} + t\, \mathbf{B}^{-\frac{1}{2}}\mathbf{X}\mathbf{B}^{-\frac{1}{2}}}\mathbf{B}^{\frac{1}{2}}
$
is invertible for all $t \in \intervalcc{0,1}$. This implies that $\mathbf{B} + \mathbf{X}$ and thus $b + X$ are positive definite.
Now, we have
\begin{align*}
	\dPosDef(b,b+X)
	&= \nabs{\log \nparen{\mathbf{B}^{-\frac{1}{2}}(  \mathbf{B} + \mathbf{X})\mathbf{B}^{-\frac{1}{2}} } }
	\leq 
	\nabs{\log \nparen{  \mathbf{I}_k + \mathbf{B}^{-\frac{1}{2}}\mathbf{X}\mathbf{B}^{-\frac{1}{2}} } }
	\\
	&\leq 
	\nabs{\mathbf{B}^{-\frac{1}{2}}\mathbf{X}\mathbf{B}^{-\frac{1}{2}}  }
	=\nabs{\mathbf{B}^{-1}\mathbf{X}}
	\leq 
	\nabs{\mathbf{B}^{-1}} \, \nabs{\mathbf{X}}
	\leq \ee^{\dPosDef(g,b)} \, \abs{X}_g,
\end{align*}
from which the stated estimate follows.
\qed
\end{proof}

\begin{lemma}\label{lem:HomembInj1}
Let $(V_i,g_i)$ be finite-dimensional Euclidean spaces for $i \in \{1,2\}$ and let $A \in \Inj(V_1;V_2)$ be injective and let $B \in \Hom(V_1;V_2)$ with 
$\abs{A-B}_{g_1,g_2} < 	\frac{1}{3} \exp(-\frac{3}{2} \ell)$,
where $\ell \ceq \dPosDef(g_1,A^\pull g_2)$.
Then also $B$ is injective and one has the estimate
\begin{align*}
	\dPosDef \nparen{A^\pull g_2,B^\pull g_2} 
	\leq 
	\tfrac{7}{3} \, \ee^{\frac{3}{2}\ell} \abs{A-B}_{g_1,g_2}.
\end{align*}
\end{lemma}
\begin{proof}
We use the preceding lemma with $g=g_1$, $b\ceq A^\pull g_2$, and  $X\ceq B^\pull g_2 - A^\pull g_2$. 
Choose $g_i$-orthonormal bases $e_i$ of $V_i$ for $i \in \{1,2\}$ and write $\mathbf{X} \ceq \Gram_{e_1}(X)$. Let $\mathbf{A}$ and $\mathbf{B}$ be the matrix representations of $A$ and $B$, respectively, with respect to these chosen bases.
Since
$\mathbf{X} 	= \mathbf{A}^\transp (\mathbf{B}-\mathbf{A}) + (\mathbf{B}-\mathbf{A})^\transp \mathbf{A} + (\mathbf{B}-\mathbf{A})^\transp\mathbf(\mathbf{B}-\mathbf{A})$ 
one obtains
\begin{align*}
	\abs{X}_g = \nabs{\mathbf{X}}
	&\leq 2 \nnorm{\mathbf{A}^\transp\mathbf{A}}^\frac{1}{2} \nabs{\mathbf{B}-\mathbf{A}}+ \nabs{\mathbf{B}-\mathbf{A}}^2
	\leq 2 \, \ee^{\frac{1}{2}\ell} \cdot \abs{A-B}_{g_1,g_2} + \abs{A-B}_{g_1,g_2}^2
	\\
	&< \nparen{2 \, \ee^{\frac{1}{2}\ell} + \tfrac{1}{3} \ee^{-\frac{3}{2} \ell}} \cdot \abs{A-B}_{g_1,g_2}
	<  \ee^{-\ell},
\end{align*}
whenever $\abs{A-B}_{g_1,g_2} < \frac{1}{3} \ee^{-\frac{3}{2} \ell}$. 
Finally, one has
\begin{align*}
	\dPosDef(b,b+X) \leq
	\ee^{\ell} \, \abs{X}_g
	\leq
	\nparen{2 \, \ee^{\frac{3}{2}\ell} + \tfrac{1}{3} \ee^{-\frac{1}{2} \ell}} \cdot \abs{A-B}_{g_1,g_2}
	\leq 
	\tfrac{7}{3} \, \ee^{\frac{3}{2}\ell} \abs{A-B}_{g_1,g_2}
\end{align*}
by \autoref{PosDefopeninSym}. 
\qed
\end{proof}

\begin{lemma}\label{lem:HomembInj2}
Let $(V_i,g_i)$ be finite-dimensional Euclidean spaces for $i \in \{1,2\}$, let $\dRay$ be defined as in \autoref{ex:dRay}, and let $A \in \Inj(V_1;V_2)$ be injective and let $B \in \Hom(V_1;V_2)$ with 
$\abs{A-B}_{g_1,g_2} < 	\frac{1}{2} \nnorm{A^{\dagger_{g_1,g_2}}}_{g_2,g_1}^{-1}$.
Then also $B$ is injective and one has the estimate
\begin{align*}
	\dinfty( \Ray(A),\Ray(B))  \leq \uppi \, \sqrt{24} \, \nnorm{A^{\dagger_{g_1,g_2}}}_{g_2,g_1} \, \nabs{A-B}_{g_1,g_2}.
\end{align*}
\end{lemma}
\begin{proof}
For $u \in V_1'\setminus \{0\}$ and $U \in \Hom(V_1;V_2)$ with $\abs{U} < \frac{1}{2}\nnorm{A^{\dagger}}^{-1}$ define the function
\begin{align*}
	F_u(U) \ceq \frac{\ninnerprod{A \,u , (A+U)\,u}}{\abs{A \,u } \abs{(A+U)\,u}}.
\end{align*}
This way, one has
\begin{align*}
	\dinfty( \Ray(A),\Ray(A+U)) 
	= \sup_{u \in V_1'\setminus \{0\}} \arccos (F_u(U))
	&\leq 	\uppi  \sup_{u \in V_1'\setminus \{0\}} \sqrt{2-2 F_u(U)}	.
\end{align*}
We use a first order Taylor expansion of $F_u$ in order to bound $\sqrt{2-2 F_u(U)}$ from above.
A direct computation shows that
\begin{align*}
	F_u(0)=1, \quad DF_u(0)=0, \qand
	\nabs{D^2 F_u(U)\, (V,V)}\leq 24 \, \nabs{V}^2  \nnorm{A^\dagger}^2
\end{align*}
hold for all $U$, $V \in \Hom(V_1;V_2)$ with $\nabs{U} < \frac{1}{2} \nnorm{A^\dagger}$.
Taylor's theorem implies the existence of some $X \in \Hom(V_1;V_2)$ with $\abs{X} \leq \abs{U}$ such that
\begin{align*}
	F_u(U) 
	= F_u(0) + DF_u(0) \, U + \frac{1}{2} D^2 F_u(X)\,(U,U)
	= 1 + \frac{1}{2} D^2 F_u(X)\,(U,U).
\end{align*}
Thus, we obtain the bound
\begin{align*}
	\arccos(F_u(U)) 
	\leq  \uppi \sqrt{2-2 F_u(U)}
	&= \uppi \sqrt{-D^2 F_u(X) \, (U,U)}
	\leq  \uppi \, \sqrt{24} \,\nnorm{A^\dagger}\, \abs{U}.
\end{align*}
\qed
\end{proof}

\begin{lemma}\label{cor:W1inftyembLip}
Let $f \in \Imm(\varSigma;\R^m)$, $g \in \PosDefSec(\varSigma)$, and $h \in W^{1,\infty}(\varSigma;\R^m)$ with 
$
	\nnorm{f-h}_{W^{1,\infty}_g} <\frac{1}{3} \exp(-\tfrac{3}{2} \ell),
$
where $\ell\ceq \dPosDefSec(g,f^\pull g_0)$.
Then also $h$ is a Lipschitz immersion and one has
\begin{align*}
	\dImm(f,h) \leq 18 \, \exp(\tfrac{3}{2} \ell) \, \nnorm{f-h}_{W^{1,\infty}_g}.
\end{align*}
In particular, $\Imm(\varSigma;\R^m)$ is an open subset in $W^{1,\infty}_g(\varSigma;\R^m)$. and
%\begin{align*}
%	\id \colon (\Imm(\varSigma;\R^m),\nnorm{\cdot}_{W^{1,\infty}_g}) \to (\Imm(\varSigma;\R^m),\dImm)
%\end{align*}is locally Lipschitz continuous.
\end{lemma}
\begin{proof}
By \autoref{lem:HomembInj1}, one has $\dd_x h\in \Inj(T_x\varSigma;\R^m)$ for almost all $x \in \varSigma$.
\autoref{lem:HomembInj1} and \autoref{lem:HomembInj2} together imply
\begin{align*}
	\dImm(f,h)
	&\leq \tfrac{7}{3} \, \ee^{\frac{3}{2}\ell} \nabs{\dd f-\dd h}_{L^\infty_g}
	+ \uppi \, \sqrt{24} \, \bigparen{
		\esssup_{x \in \varSigma} \,\nnorm{\dd_x f^{\dagger_{g,g_0}}}
	}\, \nabs{\dd f-\dd h}_{L^\infty_g}.
\end{align*}
Now, $\nnorm{\dd_x f^{\dagger_{g,g_0}}}  \leq \ee^{\frac{1}{2}\ell} \leq  \ee^{\frac{3}{2}\ell}$  and 
$\frac{7}{3} + \uppi\sqrt{24} \leq 18$ complete the proof.
\qed
\end{proof}

\begin{lemma}\label{lem:InjembHom}
Let $(V_i,g_i)$ be finite-dimensional Euclidean spaces for $i\in \{1,2\}$ and let $A$, $B \in \Inj(V_1;V_2)$ with $\dPosDef(A^\pull g_2,B^\pull g_2) <\frac{5}{2}$.
Then one has the estimate
\begin{align*}
	\nnorm{A-B}_{g_1,g_2}
	\leq 
	\exp(\dPosDef(g_1,A^\pull g_2))\, \Bigparen{ \dRay(\Ray(A),\Ray(B)) + \dPosDef(A^\pull g_2,B^\pull g_2)}.
\end{align*}
\end{lemma}
\begin{proof}
Let $u \in V_1\setminus\{0\}$.
Let $v$ and $w \in V_2$ be the unique unit vectors in the rays $\Ray(A)([u])$ and $\Ray(B)([u])$, respectively. Let $\lambda \ceq \nabs{A \, u}_{g_2}$ and $\mu \ceq \nabs{B \, u}_{g_2}$. With the triangle inequality, we obtain
\begin{align}
	\nabs{(A-B)\,u}_{g_2}
	&= \nabs{A \,u - B\, u}_{g_2}
	= \nabs{\lambda \, v - \mu \, w}_{g_2}
	\leq \nabs{\lambda \, v - \lambda \, w}_{g_2} + \nabs{\lambda \, w - \mu \, w}_{g_2}
	\notag\\
	&\leq \lambda \nabs{v - w}_{g_2} + \nabs{\lambda -\mu }
	\leq \lambda\, \paren{ 2 \sin( \tfrac{1}{2} \measuredangle(v,w)) + \nabs{1-\tfrac{\mu}{\lambda} }}
	\notag\\
	&\leq \nnorm{A}_{g_1,g_2} \abs{u}_{g_1}\, \Bigparen{ \dRay(\Ray(A),\Ray(B)) + \nabs{1-\tfrac{\mu}{\lambda} }
%	\dPosDef(A^\pull g_2,B^\pull g_2)
	}.\label{eq:InjembHom1}
\end{align}
Next, we choose an orthonormal basis $e$ of $(V_1,A^\pull g_1)$ and denote the eigenvalues of $\Gram_e(B)$ by $0<\lambda_1 < \dotsm < \lambda_ k \leq \exp(\dPosDef(A^\pull g_2,B^\pull g_2))$.
One has
$\sqrt{\lambda_1} \leq \frac{\mu}{\lambda} \leq 	\sqrt{\lambda_k}$,
thus
$
	\nabs{1-\tfrac{\mu}{\lambda} }
	\leq
	\max_{i=1,\dotsc,k} \nabs{\sqrt{\lambda_i} -1}
$.
From the inequality $\nabs{\sqrt{t}-1} \leq \nabs{\log(t)}$ for all $t \in \intervaloo{0,\exp(5/2)}$, it follows that
$\nabs{1-\tfrac{\mu}{\lambda} } \leq \dPosDef(A^\pull g_2,B^\pull g_2)$
Substitution of this along with the bound $\nnorm{A}_{g_1,g_2} \leq \exp(\dPosDef(g_1,A^\pull g_2))$ into \eqref{eq:InjembHom1} and taking the supremum over all nonzero vectors $u$ leads to the desired result.
\qed
\end{proof}
This leads us directly to the following lemma.
\begin{lemma}\label{cor:LipembW1infty}
Let $(\varSigma,g)$ be a compact Riemannian manifold with boundary
and let $f$, $h \in \Imm(\varSigma;\R^m)$ be Lipschitz immersions with
$\dPosDefSec(f^\pull g_0,h^\pull g_0) <\frac{5}{2}$. Then one has the estimate
\begin{align*}
	\nnorm{f-h}_{W^{1,\infty}_g} \leq \exp( \dPosDefSec(g, f^\pull g_0) ) \; \dImm(f,h).
\end{align*}
\end{lemma}
\newpage

\printbibliography

\end{document}